\documentclass[10pt]{amsart}
\usepackage{a4wide}
\usepackage{amsmath,amssymb,amsfonts,amsthm}
\usepackage{esint}

\usepackage{color}
\def\de{\delta}
\def\R{\mathbb{R}}

\def\p{\phi}

\def\vp{\varphi}
\def\eto{\eta^\circ}
\def\po{\phi^\circ}
\def\pso{\psi^\circ}
\def\Div{\textup{div}}
\def\dist{\textup{dist}}
\def\spt{\textup{Supp}\,}
\def\e{\varepsilon}

\renewcommand{\H}{\mathcal{H}}
\newcommand{\beq}{\begin{equation}}
\newcommand{\eeq}{\end{equation}}
\newcommand{\pa}{\partial}
\newcommand{\noop}[1]{} %% does nope

\newcommand{\Pozar}{Po{\v{z}}{\'a}r}

\theoremstyle{plain}
\newtheorem{theorem}{Theorem}[section]
\newtheorem{proposition}[theorem]{Proposition}
\newtheorem{corollary}[theorem]{Corollary}
\newtheorem{lemma}[theorem]{Lemma}
\newtheorem{definition}[theorem]{Definition}
\theoremstyle{remark}
\newtheorem{remark}[theorem]{Remark}

\newcommand{\dd}{d}
\newcommand{\N}{\mathbb N}
\newcommand{\Z}{\mathbb Z}

\def\ovc{\beta}

\def\pso{\psi^\circ}

\numberwithin{equation}{section}

\title[Crystalline mean curvature flows]{Existence and uniqueness for
anisotropic
and crystalline mean curvature flows} % with general mobilities}
\author[A. Chambolle \and M. Morini \and M. Novaga\and M. Ponsiglione]{Antonin Chambolle \and Massimiliano Morini \and Matteo Novaga\and Marcello Ponsiglione }

\begin{document}

\maketitle

\begin{abstract} 
\small{An existence  and  uniqueness  result, up  to
fattening, for crystalline mean curvature  flows 
{with  forcing and arbitrary  (convex) mobilities, is proven.}
This is achieved by introducing a new notion of solution to
the corresponding level set formulation. Such a solution satisfies the
comparison  principle and  a stability  property with  respect to  the
approximation by suitably regularized problems.  The results are valid
in any dimension and for arbitrary, possibly unbounded, initial closed
sets.
The approach  accounts   for  the possible  presence of  a
time-dependent   bounded   forcing   term,  with   spatial   Lipschitz
continuity.   As a  byproduct  of  the analysis,  the  problem of  the
convergence of the Almgren-Taylor-Wang  minimizing movements scheme to
a  unique (up  to  fattening) ``flat  flow'' in  the  case of  {general,  possibly crystalline,
anisotropies} is settled.
 
\vskip.1truecm \noindent 2000 Mathematics Subject Classification: 
%\vskip.1truecm \noindent 2000 Mathematics Subject Classification:
53C44, 	49M25, 35D40. %49J45, 74N05, 74N15, 74G70, 74G65, 74C15, 74B15, 74B10.
}
\end{abstract}

\tableofcontents

\bibliographystyle{plain}

\section{Introduction}
In this paper we deal with crystalline mean curvature flows; that is,  flows of sets $t\mapsto E(t)$ (formally) governed by the law 
\begin{equation}\label{oee} 
V(x,t) = -\psi(\nu^{E(t)}) (\kappa^{E(t)}_{\p}(x) + g(x,t)),
\end{equation}
where  $V(x,t)$ stands for the (outer) normal velocity of the boundary $\pa E(t)$ at $x$, $\p$ is a given norm on $\R^N$ representing the {\it surface tension}, $\kappa^{E(t)}_{\p}$ is the {\em anisotropic mean curvature} of $\pa E(t)$ associated with the anisotropy $\p$,  $\psi$ is a norm  evaluated at the outer unit normal $\nu^{E(t)}$ to $\pa E(t)$, and $g$ is a  bounded spatially Lipschitz continuous forcing term. The factor $\psi$ plays the role of a {\em mobility}\footnote{Strictly speaking, the
mobility is $\psi(\nu^{E(t)})^{-1}$.}.
We recall that when $\p$ is differentiable in $\R^{N}\setminus\{0\}$, then $\kappa^{E}_{\p}$ is given by
\begin{equation}\label{kappaphi}
\kappa^{E}_{\p}=\Div\left(\nabla \p(\nu^E)\right) ,
\end{equation}
{however in this work we will be interested mostly in the
``crystalline case'', which is whenever the level sets of $\p$ are
polytopes and \eqref{kappaphi} should be replaced with
\begin{equation}\label{kappaphinonsmooth}
\kappa^{E}_{\p}\in \Div\left(\partial \p(\nu^E)\right),
\end{equation}
and which we will describe later on.
}

Equation \eqref{oee} is relevant in  Materials Science and Crystal Growth, see for instance \cite{Taylor78, Gurtin93} and the references therein.
Its  mathematical well-posedness is classical in the smooth setting, that is when  {$\p$, $\psi$, $g$} and the initial set are sufficiently smooth (and $\p$ satisfies suitable ellipticity conditions).  However, it is also well-known that in dimensions $N\geq 3$  singularities may form in finite time even in the smooth case.  When this occurs the strong formulation of  \eqref{oee}  ceases to be applicable and  one needs a weaker notion of solution leading to a (possibly unique) globally defined evolution. 

Among the different approaches that have been proposed in the literature for the classical mean curvature flow (and for several other ``regular'' flows) in order to overcome this difficulty, we start by mentioning the so-called {\em level set approach}~\cite{OS,EvansSpruckI,EvansSpruckII, CGG, GigaBook}, which consists
in  embedding  the initial set in the one-parameter  family of sets given by the  sublevels    of some initial function $u^0$, and then in letting   all  these sets evolve  according to the same geometric law.   The evolving sets are themselves the sublevels of a time-dependent function $u(x,t)$, which turns out to  solve a (degenerate) parabolic equation for $u$ (with the prescribed initial datum $u^0$). The crucial point is that such a parabolic Cauchy problem is  shown to admit a global-in-time  unique viscosity solution for many relevant geometric motions. When this happens, the evolution of the sublevels of $u$ defines  a {\em generalized motion} (with initial set given by the corresponding sublevels of $u^0$), which exists for all times and agrees with the classical one until the appearance of singularities (see \cite{EvansSpruckII}). Moreover, such  a generalized motion satisfies the comparison principle and is unique whenever the level sets of $u$ have an empty interior.  
Let us mention that the appearance of a nontrivial  interior (the so called fattening phenomenon)  may in fact occur  even starting from  a smooth  set  (see for instance \cite{AngIlCh}). On the other hand,  such a phenomenon   is  rather rare:  for instance, one can show that given any uniformly continuous initial function  $u^0$, all its   sublevels, with the  exception of  at most countably many,   will not generate any fattening. 
 
 The second approach which is relevant for the present treatment is represented by  the minimizing movements scheme devised by
Almgren, Taylor and Wang~\cite{ATW} and, independently, by  Luckhaus and Sturzenhecker~\cite{LS}. It  is  variational in nature and hinges on the gradient flow structure of the geometric motion. More precisely, 
it consists in building  a family of discrete-in-time evolutions by an iterative minimization preocedure and in considering 
 any limit  of these  evolutions (as the time step vanishes) as an admissible    global-in-time solution to the  geometric motion, usually referred to as a {\em flat flow} (or ATW flat flow).
As a matter of fact, it is somewhat convenient to combine the variational approach with the level set point of view, by implementing the Almgren-Taylor-Wang scheme (ATW) for all the sublevels of the initial function $u^0$ (level set ATW). It turns out that the two approaches produce in general the same solutions: A simple proof of convergence of the level set ATW to
the viscosity solution of the level set equation  in the case of  anisotropic   mean curvature flows (with smooth anisotropy) is given in \cite{ChambolleNovaga} (see also \cite{ATW,Chambolle}); such a result implies in turn the convergence of the ATW to the aforementioned generalized motion whenever fattening does not occur. { A general consistency result between flat flows and generalized level set motions holding for a rather large class of nonlocal (yet ``regular'') geometric flows has been proved in  \cite{CMP3}.}

We now focus on the main case of interest for this paper, that is,  when the  anisotropy is crystalline. Due to the lack of smoothness, all the results mentioned before for  regular anisotropies become  much more difficult (and, in fact, some of them are still largely    open)  in the crystalline case, starting from the very definition of crystalline curvature which cannot be given  by \eqref{kappaphi} anymore{, but rather by~\eqref{kappaphinonsmooth}: }
one has to consider %look at
 a  suitable selection $z$ of the (multivalued) subdifferential map $x\mapsto \pa\p(\nu^E(x))$ {(of ``Cahn-Hoffmann fields'')},  such that the  tangential divergence $\Div_\tau z$  has minimal $L^2$-norm  among all possible selections. The crystalline curvature is then  given by the tangential divergence $\Div_\tau z_{\rm opt}$ of the optimal Cahn-Hoffman field (see \cite{BeNoPa,GigaGigaPozar}) and thus, in particular,   has a nonlocal character.  

Let us now briefly recall what is known about  the mathematical well-posedness of \eqref{oee} in the crystalline case. 
 In  two dimensions, the problem has been essentially settled in \cite{GigaGiga01} (when $g$ is   constant) by developing a crystalline version of the viscosity approach for the level-set equation, see also ~\cite{Taylor78,AlmTay95, AngGu89, GigaGiga98, GiGu96} for important former work. The viscosity approach adopted in \cite{GigaGiga01} applies in fact to more general equations of the form 
 \beq\label{oeegiga}
 V=f(\nu,-\kappa^E_\p)\,,
 \eeq
 with $f$ continuous and non-decreasing with respect to the second variable,
however without spatial dependence. 
{Former studies were rather treating the problem as a system
of coupled ODEs describing the relative motion of each facet of an
initial crystal~\cite{Taylor78,AlmTay95,AngGu89}.}
We mention also the recent paper
\cite{ChaNov-Crystal15}, where short time existence and uniqueness
of strong solutions for initial ``regular'' sets
(in a suitable sense) is  shown.

In dimension $N\geq 3$ the situation was far less clear until  very recently.
Before commenting on the new developments, let us remark that
before  these,
the only general available notion of global-in-time solution was  that of
{a} flat flow associated with the ATW {scheme, defined as
the limit of a converging subsequence of time discrete approximations.}
However, no general uniqueness and comparison results were available, except for  special classes of initial data \cite{CaCha, BelCaChaNo, GiGuMa98} or for  very specific anisotropies \cite{GigaGigaPozar}.  
As mentioned before, substantial progress   in this direction  has been made only very recently, in   \cite{GigaPozar} and \cite{CMP4}.

In \cite{GigaPozar}, the authors succeed in extending the viscosity approach of \cite{GigaGiga01} to $N=3$. {They  are able to deal with very general equations of the general form \eqref{oeegiga}
% (which requires the forcing term $g$ to be constant),
establishing existence and uniqueness for the corresponding level set formulations. In a forthcoming paper,
they show how to extend their approach to any dimension, which
is a major breakthrough~\cite{GigaPozar17} (moreover the new
proof is considerably simpler than before).
It seems that their method, as far as we know, still requires
a purely crystalline anisotropy $\p$ (so mixed situations are not allowed),
bounded initial sets, and the only possible forcing term is a constant.}
 
 In \cite{CMP4},  the first  global-in-time existence and uniqueness (up to fattening) result for the crystalline mean curvature flow  valid in all dimensions, for arbitrary (possibly unbounded) initial sets, and for general (including crystalline) anisotropies $\p$ {was} established,  but under the particular choice $\psi=\p$ (and $g=0$) in \eqref{oee}. It is based on a  new stronger distributional  formulation of the problem in terms of distance functions, which is reminiscent of, but not quite the same as, the distance formulation
 proposed and studied in \cite{Soner93} (see also \cite{BaSoSou,AmbrosioSoner,CaCha, AmbrosioDancer}). Such a formulation enables the use of parabolic  PDE's arguments to prove comparison results, but of course makes it more difficult to prove existence. The latter is established by implementing the variant of the ATW  scheme devised in ~\cite{Chambolle} (see also \cite{CaCha}). 
 The methods of \cite{CMP4} yield, as a byproduct, the uniqueness, up to fattening, of the ATW flat flow for the equation \eqref{oee}
 with $\psi=\p$ and $g=0$. But it leaves open the uniqueness issue for the general form of  \eqref{oee}  and, in particular, for  the constant mobility case
 \beq\label{oeeATW}
 V=-\kappa^E_{\p}\,,
 \eeq
originally appearing in   \cite{ATW}.  
{The main reason is technical:} the distributional formulation introduced in \cite{CMP4} becomes effective in yielding uniqueness results only if, roughly speaking, the level sets of the $\pso$-distance function from any closed set  ($\pso$ being the norm polar to the  mobility $\psi$) have 
(locally) bounded crystalline curvatures. This is certainly the case when $\p=\psi$ (and explains such a  restriction in \cite{CMP4}).

In this paper   we  remove the  restriction $\p=\psi$ and extend the existence and uniqueness results of \cite{CMP4} to the general equation \eqref{oee}.  In order to  deal with general mobilities, we cannot rely anymore on a  distributional formulation in the spirit of \cite{CMP4}, but instead we   
 extend the  notion of solution via an  approximation procedure by suitable regularized versions of \eqref{oee}. 

We now describe more in details the contributions and the methods of the paper.
Before addressing the general mobilities, we consider the case where $\psi$ may be different from $\p$ but satisfies a suitable regularity assumption, namely   
we assume that the Wulff shape associated with $\psi$ (in short the $\psi$-Wulff shape) admits an  inner tangent $\p$-Wulff shape at all points of its boundary. We call such  mobilities  {\em $\p$-regular} (see Definition~\ref{def:phiregular}). The $\p$-regularity assumption implies in turn that the level sets of the $\pso$-distance function from any closed set have 
locally bounded crystalline curvatures and makes it possible to extend the distributional formulation (and the methods)  of \cite{CMP4}  to \eqref{oee} (Definition~\ref{Defsol}),  to show that such a notion of solution bears a comparison principle (Theorem~\ref{th:compar}) and that 
the ATW scheme converge to it (Theorem~\ref{th:ATW}).  %Moreover,
As is classical, we then
% we use 
these results to build a unique level set flow (and a corresponding generalized motion), which satisfies   comparison and geometricity properties (Theorem~\ref{th:phiregularlevelset}).

Having accomplished this,  we deal with the general case of $\psi$ being any norm. As mentioned before, the idea here is to 
build a level set flow by means of approximation, after the easy observation that for any norm $\psi$ there exists a sequence $\{\psi_n\}$ of $\p$-regular mobilities such that $\psi_n\to \psi$.  More precisely, we say that $u$ is a  solution to the level set flow associated with \eqref{oee} if there exists an approximating sequence $\{\psi_n\}$ of $\p$-regular mobilities such that 
 the corresponding level set flows $u_n$ constructed in Section~\ref{sec:phiregular} locally uniformly converge to  $u$
  (Definition~\ref{deflevelset2}). 
  
  In Theorems~\ref{th:maingenmob} and ~\ref{th:proplevelset} we establish the main results of the paper:  we show that  for any norm $\psi$ a solution-via-approximation $u$  always exists; moreover $u$ satisfies the following properties: 
  \begin{itemize}
  \item[(i)] (Uniqueness and stability): The solution-via-approximation $u$ is unique in that it is   independent of the choice of the approximating sequence of $\p$-regular mobilities $\{\psi_n\}$. In fact, it is stable with respect the convergence of any sequence of mobilities and anisotropies. 
  \item[(ii)] (Comparison): if   $u^0\leq v^0$, then the corresponding level set solutions $u$ and $v$ satisfy $u\leq v$. 
  \item[(iii)] (Convergence of the level set ATW): $u$ is the unique limit of the level set ATW.
  \item[(iv)] (Generic non-fattening): As in the classical case, for any given uniformly continuous initial datum $u^0$ all but countably many   sublevels  do not produce any fattening. 
  \item[(v)] (Comparison with other notions of solutions): Our solution-via-approximation $u$ coincides with the classical viscosity solution in the smooth case and with the Giga-\Pozar\ viscosity solution \cite{GigaPozar,GigaPozar17} whenever such a solution is well-defined, that is, when %$N=3$,
 $g$ is constant, $\p$ is purely crystalline and the initial set is bounded.
  \item[(vi)] (Phase-field approximation): When $g$ is constant, a phase-field Allen-Cahn  type approximation of  $u$ holds. 
 \end{itemize}
We finally mention that   property (iii) implies the convergence of the ATW scheme, whenever no fattening occurs and { thus settles the 
long-standing problem of the uniqueness (up to fattening) of the flat flow corresponding to \eqref{oee}  (and in particular for  \eqref{oeeATW}) when the anisotropy is crystalline.}
In a forthcoming paper~\cite{CMNP-visco},
we will show that it is also possible to build
crystalline flows by approximating of the anisotropies with smooth
ones, however this variant, even if slightly simpler, does not show that
flat flows are unique.

The plan of the paper is the following. In Section~\ref{sec:dist} we extend the distributional formulation of \cite{CMP4} to our setting and we study the main properties of the corresponding notions of sub and supersolutions. The main result of the section is the comparison principle established in {Section}~\ref{sec:comp}. 

In Section~\ref{Minimizing movements} we set up the minimizing movements algorithm and we start paving the way for the main results of the paper by establishing some preliminary results. In particular, the density  estimates and the  barrier argument of 
{Section}~\ref{nuova}, which do not require any regularity assumption on the mobility $\psi$,  will be crucial for the stability analysis of the  ATW scheme needed to deal with the general mobility case and developed in {Section}~\ref{subsec:stability}.

In Section~\ref{sec:phiregular} we develop the existence and uniqueness theory under the assumption of $\p$-regularity for the mobility $\psi$. More precisely, we establish the convergence of the ATW scheme to a distributional solution of the flow, whenever fattening does not occur. Uniqueness then follows from the results of Section~\ref{sec:dist}. 

Finally, in Section~\ref{sec:genmob} we establish the main results of the paper, namely the existence and uniqueness of a solution via approximation by $\p$-regular mobilities. As already mentioned,  the approximation procedure requires  a delicate stability  analysis of the ATW scheme with respect to {varying} mobilities.  Such estimates are  established in {Section}~\ref{subsec:stability} and represent the main technical achievement of
 {Section}~\ref{sec:genmob}.

\section{A distributional formulation of the  curvature flows}\label{sec:dist}

In this section we generalize the approach introduced in \cite{CMP4} by introducing a suitable distributional formulation of   \eqref{oee} and we show that such a formulation yields a comparison principle and is equivalent to the standard viscosity formulation when the anisotropy $\p$ and the mobility $\psi$ are sufficiently regular.

The existence %and uniqueness
 of the distributional solution defined in this section   will be established in Section~\ref{sec:phiregular} under the  additional assumption that the mobility $\psi$ satisfies a suitable regularity assumption (see Definition~\ref{def:phiregular} below).

\subsection{Preliminaries}
{We}  introduce the main objects and notation used throughout the paper.

Given a norm $\eta$ on $\R^N$
(a convex, even, one-homogeneous real-valued function with $\eta(\nu)>0$
if $\nu\neq 0$), we define a polar norm $\eto$ by
$\eto(\xi):=\sup_{\p(\nu)\le 1}\nu\cdot\xi$ 
and an associated anisotropic perimeter $P_\eta$ as
\[
P_\eta(E):=\sup\biggl\{\int_E\Div \zeta\, dx: \zeta\in C^1_c(\R^N; \R^N),\, \eto(\zeta)\leq 1\biggr\}\,.
\]
As is well known, $(\eto)^\circ=\eta$ so that when the set $E$ is smooth enough
one has
\[
P_\eta(E) = \int_{\partial E}\eta(\nu^E)d\H^{N-1}\,,
\]
which is the perimeter of $E$ weighted by the surface tension $\eta(\nu)$.

% We will make repeated use of the following identities
% \beq\label{subp}
% \partial\eta(\nu) =\{\xi\,:\,\eto(\xi)\le1\textup{ and }\xi\cdot \nu \color{blue} \ge \eta(\nu)\} = \{\xi\,:\,\eto(\xi)=1\textup{ and }\xi\cdot \nu=\eta(\nu)\}
% \eeq
% {DOMANDA: E veramente necessario tutto questo??}
% (and the symmetric statement for $\eto$) for $\nu\neq 0$.
% Moreover, 
{We will often use the following characterization:
\beq\label{subp}
 \partial\eta(\nu) =\{\xi\,:\,\eto(\xi)\le1\textup{ and }\xi\cdot \nu 
{ \ge \eta(\nu)\} }
%= \{\xi\,:\,\eto(\xi)=1\textup{ and }\xi\cdot \nu=\eta(\nu)\}
 \eeq
% {DOMANDA: E veramente necessario tutto questo??}
 (and the symmetric statement for $\eto$). In particular,
if $\nu\neq 0$ and $\xi\in\partial\eta(\nu)$, then $\eto(\xi)=1$, and
 $\partial\eta(0)=\{\xi\,:\,\eto(\xi)\le 1\}$.
% while $\partial\eto(0)=\{\xi\,:\,\eta(\xi)\le 1\}$.
}
For $R>0$ we denote
$$
W^{\eta}(x,R):=\{y\,:\,\eto(y-x)\le R\}\,.
$$ 
Such a set is called the {\em Wulff shape} (of radius $R$ and center $x$) associated with the norm $\eta$ and represents
the unique (up to translations) solution of the anisotropic isoperimetric problem 
$$
\min\left\{P_\eta(E)\,:\,  |E|=|W^{\eta}(0,R)|\right\},
$$
see for instance~\cite{FonsecaMuller}.

We denote by $\dist^\eta(\cdot, E)$ the  distance from $E$ induced by the norm $\eta$, that is, for any $x\in \R^N$
\begin{equation}\label{polardist}
\dist^\eta(x, E):=\inf_{y\in E}\eta(x-y)
\end{equation}
if $E\neq\emptyset$, and  $\dist^\eta(x,\emptyset):= + \infty$.  
Moreover,  we denote by $\dd^\eta_E$ the signed distance from $E$ induced by $\eta$, i.e., 
\[
\dd^\eta_E(x):= \dist^\eta(x,E) - \dist^\eta(x,E^c)\,.
\]
so that $\dist^\eta(x,E)=\dd^\eta_E(x)^+$ and $\dist^\eta(x,E^c)=\dd^\eta_E(x)^-$, where we adopted  the  standard notation  $t^+:=t\lor  0$ and $t^-:=(-t)^+$).
Note that by \eqref{subp} we have $\eta(\nabla d^{\eto}_{E})=\eto(\nabla d^{\eta}_{E})=1$ a.e.~in $\R^N\setminus \partial E$.
We will write $\dist(\cdot, E)$ and $\dd_E$ without any superscript to denote the Euclidean distance and signed distance from $E$, respectively. 

Finally we recall that a sequence of closed sets $(E_n)_{n\ge 1}$ in $\R^m$ converges to a closed set $E$ in the {\em Kuratowki sense}:
if the following conditions are satisfied
\begin{itemize}
\item[(i)] if $x_n\in E_n$ for each $n$, any limit point of $(x_n)_{n\ge 1}$ belongs to $E$;
\item[(ii)] any $x\in E$ is the limit of a sequence $(x_n)_{n\ge 1}$, with $x_n\in E_n$ for each $n$.
\end{itemize}
We write in this case:
$$
E_n\stackrel{\mathcal K}{\longrightarrow} E\,.
$$
{It is easily checked} that $E_n\stackrel{\mathcal K}{\longrightarrow} E$ if and only if
(for any norm $\eta$) $\dist^{\eta}(\cdot, E_n)\to \dist^{\eta}(\cdot, E)$ locally uniformly in $\R^m$. In particular, Ascoli-Arzel\`a Theorem shows that
any sequence of closed sets admits a  converging subsequence  in the Kuratowski sense.

\subsection{The distributional  formulation}\label{stass}

In this subsection we give the precise  formulation of the crystalline mean curvature flows we will deal with. 
Throughout the paper the norms $\p$  and  $\psi$   will stand for the  {\em anisotropy} and  the {\em mobility}, respectively, appearing in \eqref{oee}. Note that we do not assume any regularity on $\p$ (nor on $\psi$) and in fact we are mainly interested in the case when $\p$ is crystalline, that is, when the associated unit ball is a polytope.  

Moreover,  we will assume throughout the paper that the  forcing term $g:\R^N\times [0, +\infty)\to \R$ satisfies the following two hypotheses:
\begin{itemize}
\item[H1)] $g\in L^{\infty}(\R^N\times (0, \infty))$;
\item[H2)] there exists $L>0$ such that $g(\cdot, t)$ is L-Lipschitz continuous (with respect to the metric $\pso$) for a.e. $t>0$.
\end{itemize}
\begin{remark}
Assumption H1) can be in fact weakened and replaced by 
\begin{itemize}
\item[H1)'] for every $T>0$, $g\in L^{\infty}(\R^N\times (0, T))$.
\end{itemize}
 Indeed under the weaker assumption H1)', all the arguments and the estimates presented throughout the paper continue to work in any  time interval $(0,T)$, with some of the constants involved possibly depending on $T$. {In the 
same way, if one restricts our study to the evolution of sets with
compact boundary, then one could assume that $g$ is only locally
bounded in space.}
We assume H1) instead of H1)' only to simplify the presentation. 
\end{remark}
We are now ready to provide  a suitable  distributional  formulation of the curvature flow \eqref{oee}.
\begin{definition}\label{Defsol}
Let $E^0\subset\R^N$ be 
a closed set. 
Let $E$ be a closed set in $\R^N\times [0,+\infty)$ and
for each $t\geq 0$ denote $E(t):=\{x\in \R^N\,:\, (x,t)\in E\}$. We
say that $E$ is a {\em superflow} of \eqref{oee} with
initial datum $E^0$ if
\begin{itemize}
\item[(a)] {\sc Initial Condition:} $E(0)\subseteq {E}^0$;
\item[(b)] {\sc Left Continuity:}
$E(s)\stackrel{\mathcal K}{\longrightarrow} E(t)$ as $s\nearrow t$ for all $t>0$;
\item[(c)]
If  ${E}(t)=\emptyset$ for some  $t\ge 0$, then $E(s)=\emptyset$ for all $s > t$.
\item[(d)] {\sc Differential Inequality:} 
{Set $T^*:=\inf\{t>0\,:\, E(s)=\emptyset \text{ for $s\geq t$}\}$, 
and 
$$
d(x,t):=\dist^{\pso}(x, E(t)) \qquad \text{ for all } (x,t)\in \R^N\times (0,T^*)\setminus E.
$$ }
 Then there exists $M>0$ such that the inequality
\begin{equation}\label{eq:supersol}
 \partial_t d \ge \Div z+g- Md
\end{equation}
holds in the distributional sense in $\R^N\times (0,T^*)\setminus E$
for a suitable $z\in L^\infty(\R^N\times (0,T^*))$ such that
$z\in \partial\p(\nabla d)$~a.e., $\Div z$ is a Radon measure in $\R^N\times (0,T^*)\setminus E$, and 
{$(\Div z)^+\in L^\infty(\{(x,t)\in\R^N\times (0,T^*):\, d(x,t)\geq\delta\})$} for every $\delta\in (0,1)$.
\end{itemize}

We say that $A$, open set in $\R^N\times [0,+\infty)$, is
a subflow of \eqref{oee} with initial datum $E^0$ if $A^c$ is a superflow of \eqref{oee} with $g$ replaced by $-g$ and with initial datum $(\mathring{E}^0)^c$.

Finally, we say that $E$, closed set in $\R^N\times [0,+\infty)$,  is a solution of \eqref{oee} with initial datum $E^0$ if it is a superflow and if $\mathring{E}$ is a subflow, both with initial datum $E^0$.
\end{definition}
In Subsection~\ref{subsec:convergence} we will prove the existence of  solutions satisfying \eqref{eq:supersol} with $M=L$.
\begin{remark} \label{rem:sci}
{Notice} that the closedness of $E$ yields that if $(x_k,t_k)\to (x,t)$, with $t<T^*$,
since there exist $y_k\in E(t_k)$ with $\pso(x_k-y_k)=d(x_k,t_k)$ and
since any limit point of $(y_k,t_k)$ is in $E$,
one has $d(x,t)\le\liminf_k d(x_k,t_k)$, that is, $d$ is lower semicontinuous. 
On the other hand, condition (b) implies that
$d(\cdot,t)$ is left-continuous.
{
Moreover, by condition  (d) of Definition~\ref{Defsol},  the distributional derivative 
$ \partial_t d$ is a Radon
measure in $\R^N\times (0,T^*)\setminus E$,
so that $d$ is locally a function with bounded
variation; using the fact that the distance functions are uniformly Lipschitz, we
can deduce 
that for any $t\in  [0,T^*)$, 
$d(\cdot,s)$ converges locally uniformly in $\{x\,:\, d(x, t)>0\}$ as $s\searrow t$
to  some function} $d^r$ with $d^r \geq d(\cdot, t)$ in $\{x\,:\, d(x, t)>0\}$, while $d(\cdot,s)$ converges locally uniformly to $d(\cdot, t)$
as $s\nearrow t$ (\textit{cf}~\cite[Lemma~2.4]{CMP4}).
\end{remark}

\begin{remark}\label{rm:measure}
Notice that the initial condition for subflows may be rewritten as $ \mathring E^0 \subseteq A(0)$.
In particular,  if $\partial{E}^0=\partial\mathring{E}^0$ and $E$ is a solution according to the previous definition, then $E(0)=E^0$.

\end{remark}

We now introduce the corresponding notion of sub- and supersolution to the level set flow associated with \eqref{oee}.

 \begin{definition}[Level set subsolutions and supersolutions]\label{deflevelset1}
 Let  $u^0$ be a  uniformly continuous function on $\R^N$. We will say that a  lower semicontinuous   function $u:\R^N\times [0, +\infty)\to \R$ is a {\em  supersolution  to the level set flow} corresponding to \eqref{oee} ({\em level set supersolution} for short), with initial datum $u^0$,  if $u(\cdot, 0)\geq u^0$ and if for a.e. $\lambda \in \R$ the closed sublevel set $\{(x,t)\,:\, u(x,t)\leq \lambda\}$  is a superflow {of}  \eqref{oee} in the sense of Definition~\ref{Defsol}, with initial datum $\{u_0\leq \lambda\}$.
 
 We will say that an  upper-semicontinuous   function $u:\R^N\times [0, +\infty)\to \R$ is a {\em  subsolution to the level set flow} corresponding to \eqref{oee}  ({\em level set subsolution}  for short), with initial datum $u^0$, if~$-u$ is a superlevel set flow  in the previous sense, with initial datum $-u_0$ {and with $g$ replaced by~$-g$.} 
 
 Finally, we will  say that a continuous   function $u:\R^N\times [0, +\infty)\to \R$ is a {\em  solution} to the level set flow corresponding to \eqref{oee} if it is both a level set subsolution and  supersolution. 
 \end{definition}

\subsection{The comparison principle}\label{sec:comp}
 
In this subsection we establish a comparison principle between sub- and superflows as defined in the previous subsection. 
A first technical result is a (uniform) left-continuity
estimate for the distance function to a superflow.
\begin{lemma}\label{lem:uniformcontrol}
Let $E$ be a superflow in the sense of Definition~\ref{Defsol}, and
$d(x,t)=\dist^{\pso}(x,E(t))$ the associated distance function.
Then, there exists $\tau_0,\chi$ depending on $N,\|g\|_\infty,M$
such that for any $x,t\ge 0$ and any $s\in [0,\tau_0]$,
\begin{equation}\label{eq:controlrightdist}
d(x,t+s) \ge d(x,t)e^{-5Ms} - \chi\sqrt{s},
\end{equation}
and { (for any $s\in [0,\tau_0]$ with $ s \le t$)}
\begin{equation}\label{eq:controlleftdist}
d(x,t-s) \le d(x,t)e^{5Ms} + \chi\sqrt{s}.
\end{equation}
\end{lemma}
\begin{proof}
The proof follows {the lines of the proof of~\cite[Lemma~3.2]{CMP4} up to minor changes that we will briefly describe in the following.  
By definition of a superflow we have 
\[
\partial_t d \ge \Div z - Md - \|g\|_\infty,
\]
wherever $d>0$.}
Consider $(\bar x,\bar t)$ with $d(\bar x,\bar t)=R>0$.
For $s>0$, let $\tau(s):=\log(1+Ms)/M$ and define
\[
\delta(x,s) = d(x,\bar t+\tau(s))(1+Ms) + \|g\|_\infty s \ge 0.
\]
We have that $\delta(x,0)=d(x,\bar t)$, while{,
in $\{d>0\}$, $\delta(x,\cdot)$ is $BV$ in time, the
singular part $\partial_s^s\delta$ is nonnegative (as the singular
part $\partial_t^s d$ is nonnegative thanks to~\eqref{eq:supersol} and the
assumption on $\Div z$) and the absolutely continuous part
satisfies}
\begin{equation*}
{\partial_s^a} \delta (x,s)= ({\partial_t^a} d(x,\tau(s))\tau'(s)(1+Ms)
+ Md(x,\tau(s))+\|g\|_\infty
\ge \Div z (x,\tau(s)).
\end{equation*}
As $z(x,\tau(s))\in\partial \p(\nabla \delta(x,s))$, we obtain that
$\delta$ is a supersolution of the $\p$-total variation
flow starting from $d(\cdot,\bar t)$,
and we can reproduce the proof of~\cite[Lemma~3.2]{CMP4}:
we find that there exists a constant $\chi_N$ such that
$\delta(\bar x,s)\ge R-\chi_N\sqrt{s}$ for $s\ge 0$
as long as this bound ensures that $d(\bar x,t+\tau(s))>3R/4$,
which is as long as
\begin{equation}\label{eq:limit_s}
\frac{R}{4} -\chi_N\sqrt{s} - \left(\|g\|_\infty + \frac{3MR}{4}\right)s>0.
\end{equation}
{Now we prove that  for any $s\ge 0$
\begin{equation}\label{nwp}
d(\bar x,\bar t+\tau(s))(1+Ms)\ge R-4\chi_N\sqrt{s}-(4\|g\|_\infty+3MR)s.
\end{equation}
Indeed, as long as \eqref{eq:limit_s} holds true we have 
\begin{multline*}
R-4\chi_N\sqrt{s}-(4\|g\|_\infty+3MR)s \le R-\chi_N\sqrt{s}-\|g\|_\infty s  \\
\le \delta(\bar x,s) -\|g\|_\infty s = d(\bar x,\bar t+\tau(s))(1+Ms).
\end{multline*}
%the right-hand side of
%this expression is less than $R-\chi_N\sqrt{s}-\|g\|_\infty s$ and
%this follows from the estimate $\delta(\bar x,s)\ge R-\chi_N\sqrt{s}$,
On the other hand, for later times  the left-hand side of \eqref{nwp}  is
(always) nonnegative and the right-hand side becomes nonpositive.
Notice that \eqref{nwp} can be rewritten as}
\[
d(\bar x,\bar t+\tau(s))(1+Ms) \ge d(\bar x,\bar t)(1-3Ms) - 4\chi_N\sqrt{s}
-4\|g\|_\infty s,
\]
and since this holds for any $s\ge 0$ and does not depend on the
particular value of $R$, it holds in fact for any $\bar x,\bar t$ and
we denote this point simply by $x,t$ in the sequel.

Since $s=(e^{M\tau(s)}-1)/M$, we deduce that for any 
$x,t\ge 0,\tau\ge 0$,
\noop{ %\color{blue}
\begin{multline*}%\label{eq:controlalltime}
d(x, t+\tau)\ge d( x, t)(4e^{-M\tau}-3)
- 4\|g\|_\infty \frac{1-e^{-M\tau}}{M} - 4\chi_N \sqrt{\frac{e^{M\tau}-1}{M}}e^{-M\tau}
\\
\ge 
d( x, t)(4e^{-M\tau}-3) - 4\chi_N \sqrt{\frac{e^{M\tau}-1}{M}}e^{-M\tau}.
\end{multline*}
Lemma~\ref{lem:uniformcontrol} easily follows  by Taylor expansion}.
{
\begin{equation*}%\label{eq:controlalltime}
d(x, t+\tau)\ge d( x, t)(4e^{-M\tau}-3)
- 4\|g\|_\infty \frac{1-e^{-M\tau}}{M} - 4\chi_N \sqrt{\frac{e^{M\tau}-1}{M}}e^{-M\tau}
%\\
%\ge 
%d( x, t)(4e^{-M\tau}-3) - 4\chi_N \sqrt{\frac{e^{M\tau}-1}{M}}e^{-M\tau}.
\end{equation*}
Lemma~\ref{lem:uniformcontrol} follows  by Taylor expansion}.
%{DOMANDA: non capivo la formula precedente, che mi sembrava falsa?}
\end{proof}

We can now show the following important comparison result.
\begin{theorem}\label{th:compar}
Let $E$ be a superflow  with initial datum $E^0$ and
$F$ be a subflow with initial datum $F^0$ in the sense of Definition~\ref{Defsol}. Assume that
$\dist^{\pso}(E^0,{F^0}^c)=:\Delta>0$. 
{Then, 
$$\dist^{\pso}(E(t),F^c(t))\ge \Delta \mathrm{e}^{-Mt} \qquad \text{ for all } t\ge 0,$$
where $M>0$ is as in 
\eqref{eq:supersol} { for both $E$ and $F$}.}
\end{theorem}
\begin{proof}
Let $T^*_E$ and $T^*_F$ be the maximal existence time for $E$ and $F$. For all $t> \min\{T^*_E, T^*_F\}=:T^*$ we have that either $E$ or $F^c$ is empty. For all such $t$'s   the conclusion clearly holds true. 

Thus, we may assume without loss of generality that $T^*_E, T^*_F >0$ and we consider the case $t\leq T^*$. 
By iteration %(thanks to property (b2) )
(thanks to the left-continuity of $d$)
it is clearly enough to show the conclusion of the theorem for a time interval $(0, t^*)$ for some $0<t^*\leq T^*$.

 Let us fix $0<\eta_1<\eta_2<\eta_3<{\Delta}/2$. We denote by $z_E$ and $z_{F^c}$ the  fields appearing in the definition of superflow (see Definition \ref{Defsol}), corresponding to $E$ and $F^c$, respectively.  Consider the set 
$$
S:=\{x\in \R^N\,:\, d^{\pso}_E(x,0)>\eta_1\}\cap \{x\in \R^N\,:\, d^{\pso}_{F^c}(x,0)>\eta_1\} \,.
$$
We now set 
\begin{align*}
& \widetilde d_E:=d^{\pso}_E\lor (\eta_2+Ct)\,,\\
&  \widetilde d_{F^c}:=d^{\pso}_{F^c}\lor (\eta_2+Ct)\,,
\end{align*}
with $C>0$ to be chosen later. 
By our assumptions $(\widetilde d_E+\widetilde d_{F^c})(\cdot, 0)\geq \Delta$. Moreover, since by construction 
$$
\widetilde d_E+\widetilde d_{F^c}\ge \Delta + (\eta_2-\eta_1)\text{ on }\pa S\times  \{0\}\,,
$$
it follows from Lemma~\ref{lem:uniformcontrol}
that  there exists $t^*\in (0,1\land T^*)$ such that
\beq\label{parabolicbd}
\widetilde d_E+\widetilde d_{F^c}\geq \Delta\text{ on }\pa S\times  (0, t^*)\,.
\eeq
Relying again on  Lemma~\ref{lem:uniformcontrol} and arguing similarly we also have (for a possibly smaller $t^*$)

$$
E(t)\subset\subset F(t) \qquad\text{for }t\in (0, t^*)\,.
$$
\beq\label{claim2}
\widetilde d_E=d^{\pso}_E \quad\text{and}\quad \widetilde d_{F^c}=d^{\pso}_{F^c}\qquad\text{in }S''\times (0, t^*)\,,
\eeq
where
$$
S'':=\{x\in \R^N\,:\, d^{\pso}_E(x,0)>\eta_3\}\cap \{x\in \R^N\,:\, d^{\pso}_{F^c}(x,0)>\eta_3\}\,,
$$ 
and 
$$
S\subset\biggl\{x\in \R^N\,:\, d^{\po}_E(x,t)>\frac{\eta_1}{2}\biggr\}\cap 
\biggl\{x\in \R^N\,:\, d^{\po}_{F^c}(x,t)>\frac{\eta_1}{2}\biggr\}\quad\text{for all }t\in (0, t^*)\,.
$$
Since $d^{\pso}_E$ is Lipschitz continuous in space and  $ \partial_t d^{\pso}_E $ is a measure wherever $d^{\pso}_E$ is positive, it follows that $d^{\pso}_E$ (and in turn $\widetilde d_E$)
is a function in $BV_{loc}(S\times (0,t^*))$ and its
distributional time derivative has the form
\[
 \partial_t d^{\pso}_E  = \sum_{t\in J}[ d^{\pso}_E(\cdot,t+0)- d^{\pso}_E(\cdot,t-0)]dx + \partial_t^d  d^{\pso}_E 
\]
where $J$ is the (countable) set of times where $d^{\pso}_E$ 
jumps and $\partial^d_t  d^{\pso}_E$ is the diffuse 
part of the derivative. It turns out that (see Remark~\ref{rem:sci})
$d^{\pso}_E(\cdot,t+0)-d^{\pso}_E(\cdot,t-0)\ge 0$ for each $t\in J$. Moreover, since the positive part
of  $\Div z_E$ is absolutely continuous with respect to the Lebesque measure
{(\textit{cf} Def.~\ref{Defsol}, (d))}, 
\eqref{eq:supersol} entails
\[
\partial^d_t d^{\pso}_E\ge \Div z_E+g- M d^{\pso}_E 
\]
  in $S\times (0,t^*)$. 
 Using  the chain rule (see for instance \cite{AmbDM}), in $S\times (0, t^*)$ we have 
$$
\partial^d_t \widetilde d_E=
\begin{cases}
C & \text{a.e. in }\{(x,t)\,:\,\eta_2+Ct>d^{\pso}_E(x,t)\}\,,\\
\partial^d_t  d^{\pso}_E & \text{$|\partial^d_t  d^{\pso}_E|$-a.e. in }\{(x,t)\,:\,\eta_2+Ct\leq d^{\pso}_E(x,t)\}\,.
\end{cases}
$$
An analogous formula holds for $\partial^d_t \widetilde d_{F^c}$. Recalling that 
$(\Div z_E)^+$ and  {$(\Div z_{F^c})^+ $} belong to  $L^{\infty}(S\times(0,t^*))$  
it follows  that 
\beq\label{zetae0}
\partial^d_t \widetilde d_E \geq \Div z_E+g-M\widetilde d_E \quad\text{and}\quad \partial^d_t \widetilde d_{F^c} \geq \Div z_{F^c}-g-M\widetilde d_{F^c}
\eeq
in the sense of measures in $S\times (0, t^*)$ provided that we choose 
$$
C\geq \|(\Div z_E)^+\|_{L^{\infty}(S\times(0,t^*))}+\|(\Div z_{F^c})^+\|_{L^{\infty}(S\times(0,t^*))}+\|g\|_{L^{\infty}(S\times(0,t^*))}\,.
$$
 Note also that a.e. in $S\times (0, t^*)$
\beq\label{zetae1}
z_E\in \pa \p(\nabla \widetilde d_E) \quad\text{and}\quad z_{F^c}\in \pa \p(\nabla \widetilde d_{F^c})\,.
\eeq

Fix $p>N$ and set $\Psi(s):=(s^+)^p$  and 
$w:=\Psi(\Delta - e^{Mt} (\widetilde d_E + \widetilde d_{F^c}))$.  By \eqref{parabolicbd} we have 
\beq\label{asbefore0}
w=0 \qquad\text{on }\pa S\times (0, t^*)\,.
\eeq
Using as before the chain rule for $BV$ functions, recalling \eqref{zetae0} and the fact that the jump parts of $\pa_t\widetilde d_E$ and $\pa_t\widetilde d_{F^c}$ are nonnegative, in $S\times (0, t^*)$ we have
\begin{multline}\label{asbefore}
\partial_t w\leq 
-\Psi'\big(\Delta - \mathrm{e}^{Mt}(\widetilde d_E + \widetilde d_{F^c})\big)\mathrm{e}^{Mt} [M(\widetilde d_E + \widetilde d_{F^c})+\partial^d_t (\widetilde d_E +  \widetilde d_{F^c})]
\\
\le -\Psi'\big(\Delta - \mathrm{e}^{Mt}(\widetilde d_E + \widetilde d_{F^c})\big)\mathrm{e}^{Mt}\Div(z_E+z_{F^c})\,,
\end{multline}
where in the last inequality we have used \eqref{zetae0}.
 Choose a cut-off function $\eta\in C^{\infty}_c(\R^N)$ such that $0\leq \eta\leq 1$ and $\eta\equiv 1$ on $B_1$. For every $\e>0$ we set $\eta_\e(x):=\eta(\e x)$.
Using \eqref{asbefore0} and \eqref{asbefore}, we have
\begin{align*}
\partial_t \int_{S} w \eta_\e^p dx
&\le -\mathrm{e}^{Mt} \int_{S}\eta_\e^p\Psi'\big(\Delta - \mathrm{e}^{Mt}(\widetilde d_E + \widetilde d_{F^c})\big)\Div(z_E+z_{F^c})\\
&= -\mathrm{e}^{Mt}\int_{S} \eta_\e^p\Psi''\big(\Delta - \mathrm{e}^{Mt}(\widetilde d_E + \widetilde d_{F^c})\big)(\nabla \widetilde d_E+\nabla \widetilde d_{F^c})\cdot (z_E+z_{F^c})\, dx\\
&\hphantom{\leq}\,\,\, +p\mathrm{e}^{Mt}\int_S\eta_\e^{p-1}\,\Psi'\big(\Delta - \mathrm{e}^{Mt}(\widetilde d_E + \widetilde d_{F^c})\big)\nabla\eta_\e\cdot(z_E+z_{F^c})\, dx\\
 &\le p\mathrm{e}^{Mt}\int_S\eta_\e^{p-1}\,\Psi'\big(\Delta - \mathrm{e}^{Mt}(\widetilde d_E + \widetilde d_{F^c})\big)\nabla\eta_\e\cdot(z_E+z_{F^c})\, dx\,,
\end{align*}
where we have also used the inequality $(z_E+z_{F^c})\cdot(\nabla \widetilde d_E+\nabla \widetilde d_{F^c})\geq 0$, which follows from \eqref{zetae1} and the convexity
and symmetry of $\p$. By H\"older Inequality and using the explicit expression of $\Psi$ and $\Psi'$, we get
$$
\partial_t \int_{S} w \, \eta_\e^p dx\leq Cp^2 \|\nabla \eta_\e\|_{L^p(\R^N)}\left(\int_{S} w\,  \eta_\e^p dx\right)^{1-\frac1p}\,,
$$ 
for some constant $C>0$ depending only on the $L^\infty$-norms of $z_E$ and $z_{F^c}$ and on $t^*$. Since $w=0$ at $t=0$, a simple ODE argument then yields
$$
\int_{S} w \, \eta_\e^p dx\leq \left(Cp\|\nabla \eta_\e\|_{L^p(\R^N)} t\right)^p
$$ 
for all $t\in (0, t^*)$.
Observing that $\|\nabla \eta_\e\|_{L^p(\R^N)}^p=\e^{p-N}\|\nabla \eta\|_{L^p(\R^N)}^p\to 0$ and $\eta_\e\nearrow 1$  as $\e\to 0^+$, we conclude that $w=0$,
and in turn  $\widetilde d_E+\widetilde d_{F^c}\geq \Delta \mathrm{e}^{-Mt}$   in $S\times (0,t^*)$. 
In particular, by~\eqref{claim2}, we have shown that 
$d^{\pso}_E+d^{\pso}_{F^c}\geq \Delta \mathrm{e}^{-Mt}$   in $S''\times (0,t^*)$.  In turn, this easily implies 
 that   $\dist (E(t), F^c(t))\geq \Delta \mathrm{e}^{-Mt}$ for $t\in (0, t^*)$ (see the end of the proof of \cite[Theorem 3.3]{CMP4}). 
This concludes the proof of the theorem. 
\end{proof}

The previous theorem easily yields a comparison principle also between level set subsolutions and supersolutions. 

\begin{theorem}\label{th:lscomp}
Let $u^0$, $v^0$ be uniformly continuous functions on $\R^N$  and let $u$, $v$ be respectively a level set subsolution with initial datum $u^0$ and a level set supersolution with initial datum $v^0$, in the sense of Definition~\ref{deflevelset1}. If $u^0\leq v^0$, then $u\leq v$.
\end{theorem}
\begin{proof}
Recall that by Definition~\ref{deflevelset1} there exists a null set $ N_0\subset\R$ such that  for all $\lambda\not\in   N_0$
the sets $\{(x,t): u(x,t)<\lambda\}$ and $\{(x,t): v(x,t)\leq\lambda\}$ are respectively a subflow with initial datum $\{u^0\leq \lambda\}$ and a superflow with initial datum $\{v^0\leq \lambda\}$,  in the sense of Definition~\ref{Defsol}. Fix now $\lambda\in \R$  and 
choose $\lambda<\lambda''<\lambda'$, with $\lambda'$, $\lambda''\not\in N_0$. Since
$\{v^0\leq\lambda''\}\subset \{v^0\leq\lambda'\}\subset \{u^0\leq\lambda'\}$,
we have
$$
\dist^{\pso}(\{v^0\leq\lambda''\}, \{u^0>\lambda'\})\geq\dist^{\pso}(\{v^0\leq\lambda''\}, \{v^0>\lambda'\}):=\Delta>0\,,
$$
where the last inequality follows from the uniform continuity of $v^0$. Thus, by Theorem~\ref{th:compar}
$$
\{(x,t): v(x,t)\leq\lambda\}\subset \{(x,t): v(x,t)\leq\lambda''\}\subset \{(x,t): u(x,t)<\lambda'\}\,.
$$
Letting $\lambda'\searrow \lambda$, with $\lambda'\not\in N_0$, we conclude that 
$\{(x,t): v(x,t)\leq\lambda\}\subseteq\{(x,t): u(x,t)\leq\lambda\}$
for all $\lambda\in \R$, {which is clearly equivalent to $u\le v$.} 
\end{proof}
\subsection{Distributional versus viscosity solutions}
We show here that in the smooth cases, the notion of solution
in Definition~\ref{Defsol}
coincides with the definition of standard viscosity solutions for
geometric motions, as for instance
in~\cite{BarlSouga98}.
\noop{This property, together with
the stability observed later on in Remark~\ref{rm:stability}, will be very
helpful to establish estimates using standard approaches
for viscosity solutions.}
\begin{lemma}\label{lem:visco}
Assume $\p,\psi,\pso\in C^2(\R^N\setminus\{0\})$, and assume that $g$ is continuous also with respect to the time variable.
Let $E$ be a superflow in the sense of Definition~\ref{Defsol}.
Then, $-\chi_E$ is a viscosity supersolution of
\begin{equation}\label{eq:viscoflow}
u_t = \psi(\nabla u)\big( \Div\nabla \p(\nabla u)+g\big)
\end{equation}
{in $\R^N\times (0,T^*)$, and in fact in $\R^N\times (0,T^*]$ whenever $T^*<+\infty$,  where $T^*$ is the  extinction time of $E$ introduced in Definition \ref{Defsol}}.
\end{lemma}
A converse statement is also true, see~\cite{CMP4,CMNP-visco}.
\begin{proof} We follow the proof of a similar statement in~\cite[Appendix]{CMP4}. Let $\vp(x,t)$ be a smooth test function and assume
$-\chi_E-\vp$ has a (strict) local minimum at $(\bar x,\bar t)$,
$0<\bar t\le T^*$. In other words, we can assume that near $(\bar x,\bar t)$,
$-\chi_E(x,t)\ge \vp(x,t)$, while $-\chi_E(\bar x,\bar t)=\vp(\bar x,\bar t)$.
We can also assume that the latter quantity is $-1$
{(\textit{i.e.}, $(\bar x,\bar t)\in E$)}, as if it is zero
then we trivially deduce that $\nabla\vp(\bar x,\bar t)=0$
while $\partial_t\vp(\bar x,\bar t)\ge 0$.

If $\nabla\vp(\bar x,\bar t)=0$,
thanks to~\cite[Prop.~2.2]{BarlesGeorgelin} we can
assume that also the spatial Hessian $D^2\vp(\bar x,\bar t)=0$ (and then
$D^3\vp(\bar x,\bar t)=0$, $D^4\vp(\bar x,\bar t)\le 0$).
As usual, if we assume that $-a=\partial_t\vp(\bar x,\bar t)<0$
and choose $a'<a$,
we observe that near $\bar x$, $\partial_t\vp(x,\bar t)< -a'$ and,
for $t\le\bar t$ close enough to $\bar t$ and $x$ close enough to $\bar x$,
\[
\vp(x,t)=\vp(x,\bar t)+\partial_t\vp(x,\bar t)(t-\bar t)+o(|t-\bar t|)
\ge \vp(x,\bar t) + a'(\bar t-t).
\]
Hence,
one has that near $(\bar x,\bar t)$ (for $t\le \bar t$), for some $\gamma>0$
\[
\vp(x,t)\ge -1+a'(\bar t-t)-\gamma |x-\bar x|^4.
\]
It follows that for such $t$,
$\mathcal{N}\cap \{x:\gamma |x-\bar x|^4< a'(\bar t-t)\}\cap E(t)=\emptyset$,
where $\mathcal{N}$ is a neighborhood of $\bar x$.
For $\bar t-t>0$ small enough we deduce that
$B(\bar x, (a'(\bar t-t)/\gamma)^{1/4})$ does not meet $E(t)$,
in other words $d(\bar x,t)\ge c((a'/\gamma)(\bar t-t))^{1/4}$
for constant $c$ depending only on $\psi$.
It then follows from Lemma~\ref{lem:uniformcontrol},
%below (whose proof, obviously, does not depend on Lemma~\ref{lem:visco}),
and more precisely from~\eqref{eq:controlrightdist},
that (provided $\bar t-t\le\tau_0$ where $\tau_0$ is as in 
Lemma~\ref{lem:uniformcontrol})
\begin{equation*}
d(\bar x,\bar t) \ge (\bar t - t)^{\frac{1}{4}} 
\left(c \left(\frac{a'}{\gamma}\right)^{1/4}e^{-5M(\bar t-t)} - \chi(\bar t-t)^{\frac{1}{4}}\right),
\end{equation*}
which is positive if $t$ is close enough to $\bar t$, a contradiction.
Hence $\partial_t \vp(\bar x,\bar t)\ge 0$.

If, on the other hand,
$\nabla\vp(\bar x,\bar t)\neq 0$ then we can introduce the set $F=\{\vp\le -1\}$, and we have that $F(t)$ is a smooth set near $\bar x$, for
$t\le \bar t$ close to $\bar t$, which contains $E(t)$,
with a contact at $(\bar x,\bar t)$.
We then let $\delta(x,t) = \dist^{\pso}(x,F(t))$, which
at least $C^2$ near $(\bar x,\bar t)$ (as $\psi,\pso$ are $C^2$)
and is touching $d$ from below at all the points
$(\bar x+s\nabla\p(\nu_{F(\bar t)}),\bar t)$ for $s>0$ small.

Assume that
\begin{equation}\label{eq:contravisco}
\partial_t \delta < \psi(\nabla \delta)(\Div\nabla\p(\nabla\delta)
+g) = D^2\p(\nabla\delta):D^2\delta+g
\end{equation}
at $(\bar x,\bar t)$. Then, by continuity,
we can find $\bar s>0$ small and a neighborhood
$B=\{ |x-\bar x|<\rho, \bar t-\rho<t\le\bar t\}$
of $(\bar x,\bar t)$ {in $\R^N\times (0,\bar t]$} where
\[
\partial_t \delta <  D^2\p(\nabla\delta):D^2\delta+g-M\bar s.
\]
Possibly reducing $\rho$ and using (\textit{cf}~Rem.~\ref{rem:sci}) the
left-continuity of $d$, since $d(\bar x,\bar t)=0$, we can also
assume that $d\le \bar s$ in $B$.

We choose then $s<\bar s$ small enough so that
$\bar x^s=\bar x+s\nabla\p(\nu_{F(\bar t)})$ is such that
$|\bar x^s-\bar x|<\rho$, and for $\eta>0$ small we define
$\delta^\eta(x,t) = \delta(x,t)-\eta(|x-\bar x^s|^2+|t-\bar t|^2)/2$.
Then $d-\delta^\eta$ has a unique strict minimum point at $(\bar x^s,\bar t)$
in $B$.
Moreover if $\eta$ is small enough, by continuity, we still have that
\[
\partial_t \delta^\eta <  D^2\p(\nabla\delta^\eta):D^2\delta^\eta+g-M\bar s.
\]
in $B$.

Then we continue as in~\cite[Appendix]{CMP4}:
given $\Psi\in C^\infty(\R)$ nonincreasing, convex, vanishing on $\R_+$
and positive on $(-\infty,0)$, we introduce $w=\Psi(d-\delta^\eta-\e)\chi_B$
for $\e$ small enough. We then show that,
thanks to~\eqref{eq:supersol}, for $\bar t-\rho<t< \bar t$,
\begin{multline*}
\partial_t \int w dx \le \int_B \Psi'(d-\delta^\eta-\e)
(\Div z+g - Md - \Div\nabla\p(\nabla\delta^\eta) - g +M\bar s) dx
\\
\le -\int_B \Psi''(d-\delta^\eta-\e)(\nabla d-\nabla \delta^\eta)\cdot
(z-\nabla\p(\nabla\delta^\eta)) dx + M\int_B \Psi'(d-\delta^\eta-\e)(\bar s-d)dx
\le 0
\end{multline*}
as we have assumed $d\le\bar s$ in $B$.
This is in contradiction with
$\Psi(d(\bar x^s,\bar t)-\delta^\eta(\bar x^s,\bar t)-\e)=\Psi(-\e)>0$,
and it follows that~\eqref{eq:contravisco} cannot hold:
one must have
\[
\partial_t \delta \ge \psi(\nabla \delta)(\Div\nabla\p(\nabla\delta)
+g)
\]
at $(\bar x,\bar t)$. Since this equation is geometric and
the level set $\{\delta\le 0\}$ is $F$, which is the level $-1$ of $\vp$
(near $(\bar x,\bar t)$), we also deduce that at the same point,
\[
\partial_t \vp \ge \psi(\nabla \vp)(\Div\nabla\p(\nabla\vp)
+g)
\]
so that $-\chi_E$ is a supersolution of~\eqref{eq:viscoflow}.
\end{proof}

\section{Minimizing movements}\label{Minimizing movements}
 As in~\cite{CMP4}, {in order to build solutions to our
geometric evolution problem,}
we implement a variant of the Almgren-Taylor-Wang minimizing movements scheme~\cite{ATW} (in short the ATW scheme) introduced in~\cite{Chambolle, CaCha}. In
{Section}~\ref{subsec:ATW} we adapt this construction to take into account the forcing term, as in~\cite{ChNoIFB}.  We start by presenting some preliminary properties of the incremental problem.

\subsection{The incremental problem} We recall
that given $z\in L^{\infty}(\R^N; \R^N)$ with $\Div z\in L^2_{loc}(\R^N)$ and  $w\in BV_{loc}(\R^N)\cap L^2_{loc}(\R^N)$,
 $z\cdot Dw$ denotes the Radon measure associated with the linear functional 
$$
L\vp:=-\int_{\R^N}w\,\vp \Div z\, dx- \int_{\R^N}w\,z\cdot \nabla \vp \, dx \qquad \text{ for all } \vp\in C^{\infty}_c(\R^N),
$$
see~\cite{Anz:83}. 
We recall the following {result} %useful  proposition. 
\begin{proposition}\label{prop:ATW}
{Let $f\in L^2_{loc}(\R^N)$ and $h>0$}.
There exists a field {$z\in L^\infty(\R^N;\R^N)$} and a  
unique function
$u \in BV_{loc}(\R^N)\cap L^2_{loc}(\R^N)$
such that the pair $(u,z)$ satisfies 
\begin{equation}\label{eq:iterk2}
\left\{
\begin{array}{ll}
 -h \, \Div z  + u = f \qquad &\text{ in } \mathcal D'(\R^N),  \\
 \po(z)\le 1 \quad & \text{ a.e. in } \R^N,\\
z\cdot Du = \p(Du) \qquad &\text{ in  the sense of measures}.
\end{array}
\right.
\end{equation}
Moreover, for any $R>0$ and $v\in BV(B_R)$ with $\spt (u-v)\Subset B_R$,
\[
\p(Du)(B_R) +\frac{1}{2h} \int_{B_R}{(u-f)^2}\, dx \le
\p(Dv)(B_R) +\frac{1}{2h} \int_{B_R}{(v-f)^2}\, dx,
\]
and for every $s\in \R$ the set $E_s:=\{x\in \R^N\,:\, u(x)\le s\}$ solves the minimization problem
\[
\min_{F\Delta E_s\Subset B_R}  P_\p(F;B_R) + \frac{1}{h}  \int_{F\cap B_R} (f(x) -s) \, dx.
\]

If $f_1\leq f_2$ and if $u_1$, $u_2$ are the corresponding solutions to \eqref{eq:iterk2} (with $f$ replaced by $f_1$ and $f_2$, respectively), then $u_1\leq u_2$. 

Finally if in addition $f$ is Lipschitz with
{$\psi(\nabla f)\le 1$ for some norm $\psi$,}
then  the unique solution $u$ of \eqref{eq:iterk2}
is also Lipschitz 
and satisfies  {$\psi(\nabla u)\le 1$} a.e.~in $\R^N$.
As a consequence, \eqref{eq:iterk2} is equivalent to 
\begin{equation}\label{eq:iterk3}
\left\{
\begin{array}{lll}
 -h \, \Div z  + u = f  & \text{in }\mathcal D'(\R^N),\\
 z  \in \pa\p(\nabla u)  & \text{a.e. in $\R^N$}.\\ 
\end{array}
\right.
\end{equation}
\end{proposition}

\begin{proof} See \cite[Theorem~2]{CaCha}, \cite[Theorem~3.3]{AlChNo}. %,\cite{}.
\end{proof}

{The comparison property in the previous proposition has
a ``local'' version, which results from the geometric
character of \eqref{eq:iterk2}:}
\begin{lemma}\label{lm:localcomp} Let $f_1$, $f_2\in L^2_{loc}(\R^N)$ and {let $(u_i, z_i)$, $i=1,2$, be  solutions} to 
\eqref{eq:iterk2} with $f$ replaced by $f_i$.
Assume also that for some $\lambda\in \R$,
% the sets $\{u_i\leq \lambda\}$, $i=1,2$,  have finite masure and    
\beq\label{localcomp}
\left|(\{ u_1 \leq \lambda\}\cup\{u_2\leq \lambda\})\setminus \{ f_1 \leq f_2 \} \right|=0\qquad\text{for }i=1,2\,.
\eeq
Then, $\min\{u_1,\lambda\}\leq \min\{u_2 ,\lambda\}$~a.e.
\end{lemma}
\begin{proof}
Let us set $v_i:=\min\{u_i, \lambda\}$ for $i=1,2$ and observe that 
\beq\label{subvi}
z_i\in \pa\p(\nabla v_i)\quad {\textit{a.e.}. }
\eeq
Writing \eqref{eq:iterk2} for $u_i$ and subtracting the equations we get
\beq\label{u1-u2}
-h \Div (z_1-z_2) + (u_1-u_1) = f_1-f_2\,.
\eeq
Let  $\psi$ be a smooth, increasing,
nonnegative function with support in $(0, +\infty)$ and $\eta\in C_c^\infty(\R^N;\R_+)$, {and let $p>d$.}
First notice that
\[
{\int_{\R^N} (v_1-v_2)\psi(v_1-v_2)  \eta^p\, dx= \int_{\{v_1>v_2\}} (v_1-v_2)\psi(v_1-v_2)  \eta^p\, dx \leq
\int_{\R^N} (u_1-u_2 )\psi(v_1-v_2)  \eta^p\, dx}
\]
since it can be easily checked that  $v_1-v_2\leq u_1-u_2 $ in $\{v_1>v_2\}$.
Thus, from \eqref{u1-u2} we deduce that
{\begin{multline*}
p h\int_{\R^N}  (z_1-z_2)\cdot \nabla\eta\, \psi(v_1-v_2)\eta^{p-1} dx
+ h\int_{\R^N}  (z_1-z_2)\cdot (\nabla v_1-\nabla v_2)\psi'(v_1-v_2) \eta^p\, dx
\\+\int_{\R^N} (v_1-v_2) \psi(v_1-v_2)  \eta^p\, dx
\leq \int_{\R^N} (f_1-f_2)\psi(v_1-v_2) \eta^p\, dx.
\end{multline*}
Notice  that the last integral in this equation is nonpositive,
since by \eqref{localcomp}, the set $\{f_1>f_2\}$ is contained
(up to a negligible set) in
$\{u_1>\lambda\}\cap \{u_2>\lambda\}\subseteq \{v_1=v_2\}$.
Hence, using also that $(z_2-z_1)\cdot(\nabla v_2-\nabla v_1)\ge 0$
thanks to~\eqref{subvi}, we deduce
\[
p h\int_{\R^N}  (z_1-z_2)\cdot \nabla\eta\, \psi(v_1-v_2)\eta^{p-1} dx
+\int_{\R^N} (v_1-v_2)\psi(v_1-v_2)\eta^p\, dx  \le 0.
\]}

Letting {$\psi(s)\to (s^+)^{p-1}$} we obtain
{\begin{multline*}
\|(v_1-v_2)^+\eta\|_{L^p(\R^N)}^p \le
- p h \int_{\R^N} (z_1-z_2)\cdot\nabla\eta\, \left((v_1-v_2)^+\eta\right)^{p-1}dx
\\
\le p h \|(z_1-z_2)\nabla \eta\|_{L^p(\R^N)} \|(v_1-v_2)^+\eta\|_{L^p(\R^N)}^{p-1}
\end{multline*}
so that % either $(v_1-v_2)^+\eta=0$, or
\[
\|(v_1-v_2)^+\eta\|_{L^p(\R^N)} \le 2 p h C \|\nabla \eta\|_{L^p(\R^N)}
\]
where $C=\max_{\partial\p(0)} |z|$.
}
Replacing now $\eta(\cdot)$ with $\eta(\cdot/R)$, $R>0$, assuming
$\eta(0)=1$, {we obtain
\[
\|(v_1-v_2)^+\eta(\cdot/R)\|_{L^p(\R^N)} \le
 2 p h C R^{d/p-1}\|\nabla \eta\|_{L^p(\R^N)} \stackrel{R\to\infty}{\longrightarrow} 0
\]
as we have assumed $p>d$. It follows that $(v_1-v_2)^+=0$ a.e.,
which is the thesis of the Lemma.}
\end{proof}

We conclude this subsection with  the following useful lemma, {proved in \cite[page 1576]{CaCha}.}
\begin{lemma}\label{lm:explicit}
Let $R> 0$ and $u$ the solution to \eqref{eq:iterk3}, with $f:=c_1(\po-R)\lor c_2 (\po-R)$, where $0<c_1\leq c_2$. Then $u$ is given by 
 $$
 u(x)=
 \begin{cases}
\sqrt{c_1h}\frac{2N}{\sqrt{N+1}}-c_1R  & \textup{ if } \po(x)\le \sqrt{\frac{h}{c_1}(N+1)},\\
f(x)+h\frac{N-1}{\po(x)} & \textup{ otherwise,}
\end{cases}
 $$ 
as long as $h/c_1\leq R^2/(N+1)$.
\end{lemma}
%\begin{proof}
%See \cite[page 1576]{CaCha}.
%\end{proof}

\subsection{The ATW scheme}\label{subsec:ATW}
Let $\p$, $\psi$ and $g$ satisfy all the assumptions stated in Subsection~\ref{stass}. Set
$$
G(\cdot ,t):= \int_0^t g(\cdot,s) \, ds\,.
$$
%Without loss of generality (and to simplify the presentation) we may assume that $g$ is globally bounded not only in space but also in in time
%(otherwise, the arguments below work provided that one restricts oneself to any time interval $(0, T)$). 
{Let} $E^0\subset \R^N$ be closed. Fix a time-step $h>0$ and set
$E^0_h:=E^0$. We then inductively define $E_h^{k+1}$ (for all $k\in \N$) according to the following  procedure: 
If {$E_h^{k}\not\in\{ \emptyset$, $\R^N\}$}, then let   $(u_h^{k+1},z_h^{k+1}) :\R^N\to\R\times \R^N$ satisfy
\begin{equation}\label{eq:iterk}
\left\{
\begin{array}{lll}
 -h \, \Div z_h^{k+1} + u_h^{k+1} = d^{\pso}_{E_h^k} +  G(\cdot,(k+1)h)-G(\cdot, kh) &\text{in }\mathcal D'(\R^N),  \\
 z_h^{k+1} \in \pa\p(\nabla u_h^{k+1})  & \text{a.e. in $\R^N$},\\ 
\end{array}
\right.
\end{equation}
and  set  $E_h^{k+1}:=\{x:u_h^{k+1}\le 0\}$. \footnote{{Choosing
$E_h^{k+1}=\overline{\{u_h^{k+1}<0\}}$
might provide a different (smaller) solution, which would
enjoy exactly the same properties as the one we (abritrarily) choose.}}
If either $E_h^{k}=\emptyset$ or $E_h^{k}=\R^N$, then set $E_h^{k+1}:=E_h^{k}$. We denote by $T^*_h$ the first discrete time $hk$ such that $E_h^k=\emptyset$, if such a time exists; otherwise we set  $T^*_h=+\infty$.
Analogously, we denote by ${T'_h}^*$ the first discrete time $hk$ such that $E_h^k=\R^N$, if such a time exists; otherwise we set  ${T'_h}^*=+\infty$.
\begin{remark}\label{rm:notation}
In the following, when  changing mobilities, forcing terms, and initial data,  we will sometimes write $(E^0)_{g,h}^{\psi,k}$ in place of $E_h^k$ in order to highlight the dependence of the scheme on $\psi$, $g$, and $E^0$. More generally, given any closed set $H$, $H_{g,h}^{\psi,k}$ will denote
\emph{the k-th minimizing movements starting from $H$ with mobility $\psi$,   forcing term $g$ and time-step $h$},  as described by the algorithm above.

\end{remark}

\begin{remark}[Monotonicity of the scheme]\label{rm:monotone}
From the comparison property stated in Proposition~\ref{prop:ATW} it easily follows that if   $E^0\subseteq F^0$  are closed sets, then (with the notation introduced in the previous remark) $(E^0)_{g,h}^{\psi, k}\subseteq (F^0)_{g,h}^{\psi, k}$ for all $k\in \N$.
In addition, note that $\overline{\bigl((E^0)_{g,h}^{\psi, k}\bigr)^c}=\bigl(\overline{(E^0)^c}\bigr)_{-g,h}^{\psi, k}$ for all $k$. Thus, if 
$\dist(E^0, F^0)>0$, then we may  apply Lemma~\ref{discong} below with $g_1=g_2=g$, $c=0$, and with $\eta$   the Euclidean norm, to deduce  that   $\dist\left ((E^0)_{g,h}^{\psi, k}, (F^0)_{-g,h}^{\psi, k}\right)>0$ for all $k\in \N$. 

The assumption that the sets are at positive distance is necessary,
otherwise one could only conclude, for instance, that the smallest solution of
the ATW scheme from $E^0$ is in the complement of any solution from
$F^0$, etc. 
%In fact also the following strict  monotonicity property holds:  if $ \dist (E^0, F^0)>0$, then for all $k\in \N$ we have $\dist\left ((E^0)_{g,h}^{\psi, k}, (F^0)_{{-g},h}^{\psi, k}\right)>0$ . This follows for instance by applying Lemma~\ref{discong} below with $g_1=g_2=g$, $c=0$, and letting $\eta$ be  the Euclidean norm.
\end{remark}
We now study the space regularity of the functions $u_h^k$ constructed above.
In the following computations,
given any function $f:\R^N\to \R^m$ and $\tau\in \R^N$, we denote $ f_\tau(\cdot):= f(\cdot+\tau)$. Then, the function $(u_h^{k+1})_\tau$ satisfies
\begin{multline}
 -h \, \Div (z_h^{k+1})_\tau +  (u_h^{k+1})_\tau = (d^{\pso}_{E_h^k})_\tau +  (G(\cdot,(k+1)h))_\tau-  (G(\cdot, kh))_\tau
 \le
 \\
d^{\pso}_{E_h^k} + G(\cdot,(k+1)h)-G(\cdot, kh)  + \pso(\tau)(1+Lh)\,,
\end{multline}
where in the last inequality we also used the Lipschitz-continuity of $g$.
By the comparison property stated in Proposition \ref{prop:ATW} we deduce that $(u_h^{k+1})_\tau - (1+Lh) \pso(\tau) \le u_h^{k+1}$. By the arbitrariness of $\tau\in\R^N$,  we get
$\psi(\nabla u_h^{k+1}) \le 1+Lh$, and in turn
\begin{equation}\begin{array}{ll}\label{eq:ineqd}
u_h^{k+1} \le (1+Lh)  d^{\pso}_{E_h^{k+1}} & \textup{ in } \bigl\{x\,:\,d^{\pso}_{E_h^{k+1}}(x)>0\bigr\}\,,\\[2mm]
u_h^{k+1} \ge (1+Lh)  d^{\pso}_{E_h^{k+1}} & \textup{ in } \bigl\{x\,:\,d^{\pso}_{E_h^{k+1}}(x)<0\bigr\}\,.
\end{array}
\end{equation}

We are now in  a position to define the discrete-in-time evolutions constructed via minimizing movements. Precisely, we set  
\begin{equation}\label{discretevol}
\begin{array}{l}
E_h(t):=E_h^{[t/h]},\vspace{2pt}\\
E_h:=\{(x,t): x\in E_h(t)\},\vspace{2pt}\\
 d_h(x,t):=d^{\pso}_{E_h(t)}(x),\vspace{2pt}\\
u_h(x,t):=u_h^{[t/h]}(x),\vspace{2pt} \\
z_h(x,t):=z_h^{[t/h]}(x),
\end{array}
\end{equation}
where $[\cdot]$ stands for the integer part. % of its argument.

We conclude this subsection with the {following remark.}
\begin{remark}[Discrete comparison principle]\label{rm:dcp} Remark~\ref{rm:monotone} now reads as follows:
If $E^0\subseteq F^0$ are closed sets and if we denote by $E_h$ and $F_h$ the  discrete evolutions with initial datum $E^0$ and $F^0$, respectively, then $E_h(t)\subseteq F_h(t)$ for all $t\geq 0$.
Analogously, if   $\dist(E^0,F^0)>0$, $E_h$ is defined with a forcing $g$ and $F_h$ with the forcing $-g$,
then $\dist(E_h(t), F_h(t))>0$ for all $t\geq 0$.
\end{remark}

\subsection{Evolution of $\p$-Wulff  shapes}\label{EWS}
We start paving the way for the convergence analysis of the scheme, by deriving  some estimates on the minimizing movements starting from a Wulff shape. 
We consider as initial set the $\p$-Wulff shape $W^\p(0, R)$, for $R>0$.
% We assume that $0<R\leq 1$.
First, thanks to Lemma~\ref{lm:explicit} (with $c_1=c_2=1$)
(\textit{cf} also~\cite[Appendix B,  Eq.~(39)]{CaCha}),
the solution of \eqref{eq:iterk2}
with $f=\dd^{\po}_{W^\p(0,R)}=\po-R$ is given  {(for $h$ small enough)} by $\po_h-R$, where 
\begin{equation}\label{eq:explicitpoh}
  \po_h(x): =\begin{cases}
    \sqrt{h}\frac{2N}{\sqrt{N+1}} & \textup{ if } \po(x)\le \sqrt{h(N+1)},\\
    \po(x)+h\frac{N-1}{\po(x)} & \textup{ else.}
  \end{cases}
\end{equation}

Observe then that there exist two positive constants $c_1\leq c_2$ such that 
\begin{equation}\label{c12}
  c_1\po\leq \pso\leq  c_2\po\,,
\end{equation}
and in particular
\begin{equation}\label{c12W}
  d^{\pso}_{W^\p(0,R)} \le  c_1(\po-R) \lor c_2(\po-R).
\end{equation}
Thus,  for any $k\in \N$ we have
\begin{equation}\label{conf10000}
  d^{\pso}_{W^\p(0,R)} + G(\cdot,(k+1)h)-G(\cdot, kh)\leq c_1(\po-R)\lor c_2 (\po-R)+\|g\|_\infty h=:f\,.
\end{equation}
Denoting by $u$ the solution to \eqref{eq:iterk3}, with $f$ defined above, then  Lemma~\ref{lm:explicit} yields
 $$
 u(x)=\|g\|_\infty h
 + \begin{cases}
\sqrt{c_1h}\frac{2N}{\sqrt{N+1}}-c_1R & \textup{ if } \po(x)\le \sqrt{\frac{h}{c_1}(N+1)},\\
f(x)+h\frac{N-1}{\po(x)} & \textup{ otherwise,}
\end{cases}
 $$ 
 provided that $h/c_1\leq C(N)R^2$ (here and in the following $C(N)$ denotes a positive constant that depends only on the dimension $N$ and  may change from line to line). 
Notice that $ \{u\leq 0\} = W^{\p}(0, \bar r)$ for 
 $$
 \bar r:=\frac{R-\frac{h}{c_1}\|g\|_\infty+\sqrt{(R-\frac{h}{c_1}\|g\|_\infty)^2-4 \frac{h}{c_1}(N-1)}}{2}.
 $$
Taking into account \eqref{conf10000}, we may apply the comparison principle stated in Proposition~\ref{prop:ATW} to infer  that if $k=[t/h]$ and $E^k_h=E_h(t)=W^\p(0,R)$, then 
 $$
 W^{\p}(0, \bar r)=\{u\leq 0\}\subseteq \{u_h^{k+1}\leq 0\}=E_h(t+h)\,,
 $$
 provided that $h/c_1\leq C(N) R^2$. Since
 $$
 \bar r\geq \sqrt{(R-\frac{h}{c_1}\|g\|_\infty)^2-4 \frac{h}{c_1}(N-1)}\stackrel{(R\leq 1)}{\geq}
 \sqrt{R^2-2\frac{h}{c_1}\big[2(N-1)+\|g\|_\infty\big]}
 $$
 and setting for $0\leq s-t\leq \frac{c_1R^2}{4\big[2(N-1)+\|g\|_\infty\big]}$
 \beq\label{rhs}
 r^R(s):=\sqrt{R^2-2\frac{s-t}{c_1}\big[2(N-1)+\|g\|_\infty\big]}\geq \frac{R}{\sqrt2}\,,
 \eeq
 by iteration  we  deduce that 
 \beq\label{errehstima}
  W^{\p}(0, r^R(s))\subseteq E_h(s)
 \eeq
 for all $0\leq s-t\leq \frac{c_1R^2}{4\big[2(N-1)+\|g\|_\infty\big]}$ and $h\leq c_1C(N)R^2$. 

{In particular,  there exists a constant $C$ depending only on $\|g\|_\infty$, $c_1$ 
 and the dimension $N$, and $h_0>0$, depending also on $R$,  such that for any $y\in \R^N$ and for all $h\leq h_0$ 
 \begin{equation}\label{neweq}
 W^\p(y, R-\tfrac{C}{R}h)\subseteq \bigl( W^\p(y, R) \bigr)_{{g,h}}^{{\psi,1}}\subseteq W^\p(y, R).  
 \end{equation}
}

%%% TO DO MOVE (3.20)
\subsection{Density estimates and barriers}\label{nuova}
In this {section} we collect some preliminary estimates on the incremental problem that will be crucial for the stability properties established in 
{Section}~\ref{subsec:stability}.
The following density lemma  and the subsequent corollary show that the solution to the incremental problem starting from a closed set $E$ cannot be too ``thin''  in $\R^N\setminus E$. The main point is that the estimate turns out to be independent of $h$ and $\psi$.
{We observe
%\begin{remark}\label{rm:deponphi}
that there exist  positive constants $a_1$, $a_2$ such that 
\beq\label{phia12}
a_1|\xi|\leq\p(\xi)\leq a_2 |\xi| \quad\text{for all $\xi\in \R^N$.}
\eeq
}
%A careful inspection of the proof of Lemma~\ref{density} shows that the constant $\sigma>0$ can be chosen as depending  depends only  on the ellipticity constants $a_1$, $a_2$, and the dimension $N$. 
%\end{remark}

\begin{lemma}\label{density}
Let $E\subset \R^N$ be a closed set,  $h>0$, and let $g_h\in L^\infty(\R^N)$ with $\|g_h\|_\infty \le G h$ for some $G>0$. Let $E'$ be a solution to
\begin{equation}\label{minmo}
\min_{ F\Delta E'\subset \subset B_R} P_\p(F;B_R) + \frac{1}{h} \int_{F\cap B_R} (d_E^\psi(x) + g_h(x)) \, dx %%=: F_h(F,B_R)
\end{equation}
for all positive $R$. Then, there exists $\sigma>0$, depending only on 
%the dimension $N$ and the anisotropy $\p$,  
{$N,a_1,a_2$}, 
and $r_0\in(0,1)$, depending {on $N,a_1,a_2,G$},
 with the following property:  if $\bar x$ is such that $|E'\cap B_s(\bar x)|>0$ for all $s>0$  and  
 $B_r(\bar x )\cap E=\emptyset$ with $r\le r_0$, then  
$$
|E'\cap B_r(\bar x)| \ge \sigma r^N.
$$
\end{lemma}
\begin{proof}
 We adapt to our context a classical argument from the regularity theory of  the (quasi) minimizers of the perimeter. As mentioned  before, the main point is to use the fact that the ball  $B_r$ lies outside $E$ to deduce that 
 the constant {$\sigma,r_0$} are independent of $h$ and $\psi$. 
Fix $R>0$  such that $B_r(\bar x) \subset B_R$. For all $s\in(0,r)$, set $E'(s):= E'\setminus B_s(\bar x)$.
Note that  for a.e. $s$ we have {(as $\p$ is even)}
$$
P_\p(E'(s);B_R) = P_\p(E' ;B_R) - P_\p(E' \cap B_s(\bar x)) 
+ { 2\int_{E'\cap\partial B_s(\bar x)} \p(\nu)d\H^{N-1}.}
%P_\p( B_s(\bar x) ; E').
$$
Using also the fact that 
$$
\int_{E'(s)\cap B_R}d^{\psi}_E\, dx\le \int_{E'\cap B_R}d^{\psi}_E\, dx
$$
(since $d^{\psi}_E>0$ in $E^c$), by 
%the minimality inequality $F_h(E',B_R)\leq F_h(E'(s), B_R)$ and
{minimality of $E'$ in~\eqref{minmo} we find:
\[
P_\p(E'\cap B_s(\bar x))+\frac{1}{h}\int_{E'\cap B_R} g_h(x)dx
\le  2\int_{E'\cap\partial B_s(\bar x)} \p(\nu)d\H^{N-1} + \frac{1}{h}\int_{E'(s)\cap B_R}g_h(x)dx.
\]
Using now~\eqref{phia12} and the isoperimetric inequality (whose
constant is denoted $C_N$) we can estimate}%
% the Isoperimetric Inequality (denoting by $C_N$ the isoperimetric constant) we can estimate
\begin{align}
{2 a_2\H^{N-1}(E'\cap \partial B_s(\bar x))} &
{\ge a_1 P(E'\cap B_s(\bar x)) + \frac{1}{h} \int_{E' \cap B_s(\bar x)} g_h(x) \, dx}
\nonumber\\
& \ge {a_1}C_N |E'\cap B_s(\bar x)|^{\frac{N-1}{N}} - G |E'\cap B_s(\bar x)| 
\ge \frac{{a_1}C_N}{2} |E'\cap B_s(\bar x)|^{\frac{N-1}{N}}\,, \label{diffineq}
\end{align}
provided that $|E'\cap B_s(\bar x)|\leq \bigl(\frac{{a_1}C_N}{2G}\bigr)^N$, which is true if $r_0$ is small enough.
Recalling that    $|E'\cap B_s(\bar x)|>0$ for all $s$ and that 
%$P_\p( B_s(\bar x) ; E')\leq c_\p \mathcal{H}^{N-1}(\pa B_s(\bar x) \cap E')$ for some constant $c_\p>0$ depending only on $\p$,
 the inequality \eqref{diffineq} implies in turn that   
$$
\frac{d}{ds} |E'\cap B_s(\bar x)|^{\frac1N} \ge \frac{{a_1}C_N}{4{a_2}N } \text{ for a.e. } s\in(0,r).
$$
The thesis follows by integrating the above differential inequality.
\end{proof}
\begin{remark}
The same argument shows that a similar but $h$-dependent density estimate holds inside $E$.
\end{remark}

{
{We introduce the following notation:}  For any set $A\subset\R^N$, for any norm $\eta$, and for $\rho\in \R$ we denote
\begin{equation}\label{nota}
 (A)^{\eta}_\rho:=\{x\in \R^N\,:\, d_{A}^{\eta}(x)\leq \rho\}\,,
 \end{equation}
 and we will omit $\eta$ in the notation if $\eta$ is the Euclidean norm. 
We also recall the  notation $E_{g,h}^{\psi,k}$ introduced in Remark~\ref{rm:notation} to denote the $k$-th minimizing movement starting from $E$, with mobility $\psi$,   forcing term $g$ and time-step $h$.
}

\begin{corollary}\label{densityc}
Let $g$ and $\psi$ be an admissible forcing term and a mobility, respectively, and let $h>0$. Denote by   $E_{g,h}^{\psi,1}$ the corresponding (single) minimizing movement starting from $E$ (see {Section}~\ref{subsec:ATW} and Remark~\ref{rm:notation}). 
Let $\sigma$ and $r_0$ be the constants provided by Lemma~\ref{density} for $G:= \|g\|_\infty +1$. If $\bar x \in  E_{g,h}^{\psi,1}$
and $B_r(\bar x )\cap E=\emptyset$ with $r\le r_0$, then  
$$
|E_{g,h}^{\psi,1} \cap B_r(\bar x)| \ge \sigma r^N.
$$
\end{corollary}

\begin{proof}
Recall that $E_{g,h}^{\psi,1}= \{u(\cdot) \le 0\}$, where $u$ solves
\begin{equation*}
\left\{
\begin{array}{ll}
 -h \, \Div z  + u = d_E^\psi + \int_0^h g(\cdot,s) \, ds \qquad &\text{ in } \mathcal D'(\R^N),  \\
 \po(z)\le 1 \quad & \text{ a.e. in } \R^N,\\
z\cdot Du = \p(Du) \qquad &\text{ in  the sense of measures}.
\end{array}
\right.
\end{equation*}
Thus, by virtue of Proposition~\ref{prop:ATW},
 setting $E'_\eta:= \{u(\cdot) \le \eta\}$ for $\eta\in(0,h)$, we have that $E'_\eta$ solves \eqref{minmo} with $g_h:= \int_0^h g(\cdot,s) - \eta$. 
%and $E'$ replaced by $E'_\eta$
 Since $\bar x$ belongs to the interior of $E'_\eta$ and  $\|g_h\|_\infty\leq G h$, from Lemma \ref{density} we deduce that 
\begin{equation}\label{minmoden2}
|E_\eta'\cap B_r(\bar x)| \ge \sigma r^N.
\end{equation}
The thesis follows by monotone convergence by letting $\eta\searrow 0^+$.
\end{proof}
\noop{Throughout this {section}, given a set $E$ and $\e\in \R$, the symbol $(E)_\e$ stands for
$$
(E)_\e:= \{x\in \R^N:\, d_E(x)\leq \e\}\,,
$$ 
where $d_E$ is the Euclidean signed distance function from $E$.}

\begin{lemma}\label{lemmetto}
Let $F\subset \R^N$ be a convex set 
and let $r>0$. Then,
$$
| ((F)_\e \setminus F ) \cap B_r(0)|\le C(N) \e r^{N-1} \qquad \text{ for all } \e\ge 0,
$$ 
where $C(N)$ depends only on the dimension $N$.
\end{lemma}

\begin{proof}
Notice that $(F)_s$ is convex for all positive $\e$, so that  $(F)_s\cap B_r$ is a convex set contained in $B_r$, and 
$$
\mathcal H^{N-1}(\partial ((F)_s \cap B_r)) \le \mathcal H^{N-1}(\partial B_r) = C(N) r^{N-1} \qquad \text{ for all } s>0.
$$
Therefore, thanks to the coarea formula,
%\begin{multline*}
\[
| ((F)_\e \setminus F ) \cap B_r(0)| = \int_0^\e  \H^{N-1}(\partial (F)_s \cap B_r) \, ds
%\\
\le \int_0^\e  \H^{N-1}(\partial ((F)_s \cap B_r)) \, ds \le C(N) \e r^{N-1}.
%\end{multline*}
\]
\end{proof}

The next lemma provides a crucial   estimate on the ``expansion''  of any closed set $E$ under a single minimizing movement, provided that $E$ satisfies a uniform exterior Wulff shape condition. The result is achieved by combining a barrier argument with the the density estimate established in Corollary~\ref{densityc}.
%{We  recall that the sets $E_{g,h}^{\psi,k}$, introduced in Remark~\ref{rm:notation}, denote the $k$-th minimizing movement starting from $E$, with mobility $\psi
%$,   forcing term $g$ and time-step $h$. Moreover, we recall that (see \eqref{nota}) for any set $A\subset\R^N$, for any norm $\eta$, and for $\rho\in \R$,  $(A)^{\eta}_\rho:=\{x%\in \R^N\,:\, d_{A}^{\eta}(x)\leq \rho\}$. }

\begin{lemma}\label{lm:crucial} For any $\beta$,  $G$, $\Delta>0$, there exists $h_0 > 0$, depending on the previous constants, on 
the anisotropy $\p$ 
%{(through the constants $a_1,a_2$ in~\eqref{phia12})} not clear
% at this point!
and the dimension $N$, and there exists $M_0>0$ depending on  the same quantities but $\Delta$, with the following property:
Let $\psi$ be a mobility satisfying
\begin{equation}\label{beta0}
\psi\leq \beta \p
\end{equation}
and  let $g$ be an admissible forcing term with $\|g\|_\infty\leq G$.  Then  for any closed set $E\subseteq \R^N$ such that $\R^N\setminus E = \bigcup_{W\in\mathcal G} W$, where $\mathcal G$ is a family of 
(closed) $\p$-Wulff shapes of radius $\Delta$, and for  all $h\le h_0$,  we have  $E_{g,h}^{\psi,1}\subset    (E)^{\po}_{\frac{M_0 h}{\Delta}}$.
\end{lemma}

\begin{proof}
{First notice that \eqref{beta0} is equivalent to 
$
\frac{1}{\beta}\po\leq \pso\,.
$
 Hence recalling \eqref{neweq} and  Lemma~\ref{lemmetto}, there exists a constant $C$, depending only on $\|g\|_\infty$, $\beta$
{($=1/c_1$ in~\eqref{neweq})},   % $\p$,
 and  $N$, such that}
 $$
 W^\p(y, \Delta-\tfrac{C}{\Delta}h)\subseteq \bigl( W^\p(y, \Delta) \bigr)_{{-g},h}^{\psi,1}\subseteq W^\p(y, \Delta)  
 $$
for all $h\leq h_0$. By Lemma~\ref{lemmetto} it follows that 
for a possibly different constant $C$, depending only on
$\|g\|_\infty$, $\beta$, %$\p$,
and  $N$,  we have
\begin{equation}\label{prista}
\Bigl|\bigl(W^\p(y, \Delta)\setminus \bigl( W^\p(y, \Delta) \bigr)_{{-g},h}^{\psi,1} \bigr)\cap B_r(x)\Bigr |  \le   \frac{C}{\Delta} h r^{N-1} \quad \text{for all } x, y\in\R^N,\, r>0,\, h\le h_0\,.
\end{equation}

Observe now that there exists $\theta>0$, depending only on $W^{\p}(0,1)$, such that   
\begin{equation}\label{qq1bis}
\frac{|W^\p(y, R)\cap Q_r(x)|}{|Q_r(x)|}\ge \theta \text{ for all $y, x\in \R^N$,  $x\in W^\p(y, R)$,  $R\geq 1$,  and  $r\in (0,1)$, }
\end{equation}
where $Q_r(x)$ stands for the cube of side $r$ centered at $x$.  Let   $\sigma>0$ be the constant provided by Corollary~\ref{densityc}, 
 %set $\delta:=\frac{\sigma}{4}$, 
and  let $\mathcal{N}$ be the constant provided by 
 the covering Lemma~\ref{lemmarico} below, corresponding to the constant $\theta$  in \eqref{qq1bis}, 
%the choice of $\delta$ just made
{$\delta=\sigma/4$}  and $\mathcal C=W^\p(0,1)$. 
 Set $M_0 := (2C \mathcal{N})/\sigma$, where $C$ is the constant in \eqref{prista} and note that 
 for $r=  \frac{M_0 h}{\Delta}$ we have
\begin{equation}\label{prista2}
\mathcal{N} \frac{C}{\Delta} h r^{N-1} = \mathcal{N} \frac{C}{\Delta^N} h^N M_0^{N-1} =\frac{\sigma}{2} \left(\frac{M_0 h}{\Delta}\right)^N \qquad \text{ for all }h.
\end{equation}

Let  $x\in \R^N$  be such that $W^\p\left(x, \frac{M_0h}{\Delta}\right)\cap E=\emptyset$ for some $h\leq h_0$,  and assume by contradiction that   $x\in E_{g,h}^{\psi,1}$. Without loss of generality we may assume $x=0$. 

By taking $h_0$ smaller if needed, we can also assume that $\frac{\Delta^2}{M_0 h} \ge 1$ and 
$\frac{M_0h}{\Delta}\leq r_0$ for all $h\leq h_0$, where $r_0$ is the radius provided by Corollary~\ref{densityc}. 
Thus, recalling also \eqref{qq1bis}, for $h\leq h_0$,
{applying}  Lemma~\ref{lemmarico} below {(with $\delta=\sigma/4$)}
to the family
$$
\mathcal F=\frac{\Delta}{M_0h}\mathcal G= \left\{\frac{\Delta}{M_0h}W: W\in \mathcal G\right\},
$$
 {we} find a finite subfamily 
 $$
 \mathcal F'= \left\{\frac{\Delta}{M_0h}W^{\p}(x_1,\Delta) , \ldots,  \frac{\Delta}{M_0h}W^{\p}(x_{\mathcal{N}}, \Delta)\right\}
 $$
 of $\mathcal{N}$  elements   such that 
 $$
 \biggl|W^\p(0,1)\setminus \bigcup_{i=1}^{\mathcal{N}}\frac{\Delta}{M_0h}W^{\p}(x_i,\Delta)\biggr|\leq %\delta=
\frac{\sigma}4\,.
 $$
% where the last equality is due to the choice of $\delta$.  
By scaling back we obtain
\begin{equation}\label{avanzapoco2}
\Bigr|W^\p\Bigl(0, \frac{M_0h}{\Delta}\Bigr)\setminus \bigcup_{i=1}^{\mathcal{N}} W^{\p}(x_i,\Delta)\Bigl|\le    \frac{\sigma}{4} \left(\frac{M_0 h}{\Delta}\right)^N. 
\end{equation}      
Note now that by the comparison principle  (see Remark~\ref{rm:monotone})
$$
E_{g,h}^{\psi,1}\cap  W^\p\Bigl(0, \frac{M_0h}{\Delta}\Bigr)\subset  W^\p\Bigl(0, \frac{M_0h}{\Delta}\Bigr)\setminus
 \bigcup_{i=1}^{\mathcal{N}} \bigl(W^{\p}(x_i,\Delta)\bigr)_{{-g},h}^{\psi,1}\,. 
$$
Thus, using also  \eqref{prista}, \eqref{prista2} and \eqref{avanzapoco2} we deduce that
\begin{align*}
&\Bigl|E_{g,h}^{\psi,1}\cap  W^\p\Bigl(0, \frac{M_0h}{\Delta}\Bigr)\Bigr| \leq\Bigr | W^\p\Bigl(0, \frac{M_0h}{\Delta}\Bigr)\setminus
 \bigcup_{i=1}^{\mathcal{N}} \bigl(W^{\p}(x_i,\Delta)\bigr)_{{-g},h}^{\psi,1}\Bigr| \\
 &\le    
\Bigr | W^\p\Bigl(0, \frac{M_0h}{\Delta}\Bigr)\setminus
 \bigcup_{i=1}^{\mathcal{N}} W^{\p}(x_i,\Delta) \Bigr| +    \sum_{i=1}^{\mathcal{N}}\Bigl|\bigl(W^\p(x_i, \Delta)\setminus \bigl( W^\p(x_i, \Delta) \bigr)_{{-g},h}^{\psi,1} \bigr)\cap {W^\p\Bigl(0, \frac{M_0h}{\Delta}\Bigl)}   \Bigr | 
\\
&\le 
\frac{\sigma}{4} \left(\frac{M_0 h}{\Delta}\right)^N +  \mathcal{N} \frac{C}{\Delta} h \left(\frac{M_0 h}{\Delta}\right)^{N-1} =\frac{3}{4}\sigma \left(\frac{M_0 h}{\Delta}\right)^N,
\end{align*}
which contradicts the density estimate provided by Corollary~\ref{densityc}.
\end{proof}

We conclude %the subsection by stating and proving 
with the following covering lemma, which we used in the previous proof:
\begin{lemma}\label{lemmarico}
Let $\mathcal{F}$ be a family of closed convex sets in $\R^N$
which covers a closed convex set $\mathcal{C}$. Assume that there
exists $\theta >0$ such that for all $W\in\mathcal{F}$, $x\in \mathcal{C}$
and $r\le 1$,
\begin{equation}\label{eq:densmin}
\frac{|W\cap Q_r(x)|}{Q_r(x)} \ge\theta,
\end{equation}
where $Q_r(x) = x+[-r/2,r/2]^N$. Then for every $\delta>0$ there
exists $\mathcal{N}=\mathcal{N}(\delta,\theta,\mathcal{C},N)$ and sets $W_i\in \mathcal{F}$, $i=1,\dots,n\le \mathcal{N}$  such that  
\begin{equation}\label{avanzapoco}
\Bigl|B_1(0)\setminus \bigcup_{h=1}^{\mathcal{N}} W_{h}\Bigr|\le \delta. 
\end{equation}       
\end{lemma}
\begin{proof}

given $\e\in \{2^{-i}:i\in\N\}$, let $\mathcal{Q}_{\e}$
be the set of closed cubes
of size $\e$, centered at $\e\Z^N$, which are included in $\mathcal{C}$. Consider
$\e_1$ the largest dyadic value for which,
letting $I_1=\bigcup_{Q\in \mathcal{Q}_{\e_1}} Q$, one has $|\mathcal{C}\setminus I_1|\le 1/2$.

For each $Q=Q_{\e_1}(x)\subset I_1$, we choose $W_Q\in\mathcal{F}$ with $x\in W_Q$
and let $F_1= \bigcup_{Q\in \mathcal{Q}_{\e_1},Q\subset I_1} Q\cap W_Q\subset I_1$, 
$\widetilde F_1 = \bigcup_{Q\in \mathcal{Q}_{\e_1},Q\subset I_1} W_Q$.
By construction and~\eqref{eq:densmin}, $F_1\subseteq I_1$ and
$|F_1|\ge\theta |I_1|$.
Observe moreover that the number $\mathcal{N}_1$ of sets $Q\in\mathcal{Q}_{\e_1}$
with $Q\subset I_1$ depends only on the initial convex set $\mathcal{C}$.

We now assume we have built sets
$I_i,F_i,\widetilde{F_i}$, $i=1,\dots,k-1$, such that
\begin{itemize}
\item[i)] $I_i$ is the union of $\mathcal{N}_i$ dyadic cubes
$Q_{i,1},\dots,Q_{i,\mathcal{N}_i}\in\mathcal{Q}_{\e_i}$
where $\e_i,\mathcal{N}_i$ depend only on $\theta,\mathcal{C}$
and the dimension $N$;
\item[ii)] $F_i\subset I_i$, $F_i=\bigcup_{l=1}^{\mathcal{N}_i} W_{Q_{i,l}}\cap Q_{i,l}$ 
and $\widetilde F_i=\bigcup_{l=1}^{\mathcal{N}_l} W_{Q_{i,l}}$
where $W_{Q_{i,l}}\in\mathcal{F}$ contains the center of the cube $Q_{i,l}$;
\item[ii)] For $j=1,\dots, k-1$,
$I_j\subset \mathcal{C}\setminus(\bigcup_{i=1}^{j-1}\mathring{F}_i)$ (in particular
the sets $F_i$ have disjoint interior) and 
$|\mathcal{C}\setminus(\bigcup_{i=1}^{j-1}F_i)\setminus I_j|\le 2^{-j}$.
\end{itemize}
{We claim}  that we can build $I_k,F_k,\widetilde{F}_k$ which
satisfy the same conditions, with $I_k$ made of $\mathcal{N}_k$
cubes of size $\e_k$, the numbers $\mathcal{N}_k,\e_k$ depending
only on $\theta,\mathcal{C},N$.

In order to do this, we show that we can find
$\e_k<\e_{k-1}$ depending only on $\theta,\mathcal{C},N$
such that if $I_k$ is the union of all the cubes in
$\mathcal{Q}_{\e_k}$ not intersecting $\cup_{i=1}^{k-1}\mathring{F}_i$,
then
$|\mathcal{C}\setminus (\cup_{i=1}^{k-1} F_i)\setminus I_k|\le 2^{-k}$.
The set $\mathcal{C}\setminus (\cup_{i=1}^{k-1} \mathring{F}_i)\setminus I_k$ is made of
all the dyadic cubes of size $\e_k$ centered at $\e_k\Z^N$ which either
\begin{itemize}
\item intersect $\partial \mathcal{C}$: the total measure of such cubes
is bounded by $c\mathcal{H}^{N-1}(\partial \mathcal{C})\e_k$;
\item intersect, for some $i=1,\dots,k-1$,
$\partial (W_{Q_{i,l}}\cap \mathring{Q}_{i,l})$ for some $l\in\{1,\dots,\mathcal{N}_i\}$, where 
$Q_{i,l}$ is a dyadic cube of size $\e_i$ in $I_i$:
the total measure of such small cubes for a given $Q_{i,l}$ is bounded,
thanks to Lemma~\ref{lemmetto}, by $c(\e_i)^{N-1}\e_k=c|Q_{i,l}|(\e_k/\e_i)\le c(\e_k/\e_i)|W_{Q_{i,l}}\cap Q_{i,l}|/\theta$, hence 
the total measure of this region is bounded by
\[
\frac{c}{\theta} \sum_{i=1}^{k-1}\sum_{l=1}^{\mathcal{N}_k} \frac{\e_k}{\e_i}|W_{Q_{i,l}}\cap Q_{i,l}| =
\frac{c}{\theta} \sum_{i=1}^{k-1} \frac{\e_k}{\e_i}|F_i|
\le 
\frac{c|\mathcal{C}|}{\theta}\frac{\e_k}{\e_{k-1}}.
\]
\end{itemize}
We see that volume of the union of all these ``bad'' cubes is less than
\[
\frac{c|\mathcal{C}|}{\theta}\frac{\e_k}{\e_{k-1}}+ c\mathcal{H}^{N-1}(\partial \mathcal{C})\e_k .
\]
Again, one can find $\e_k$ depending only on $\mathcal{C},N,\theta$
such that this quantity is less than $2^{-k}$. It follows that the
total number of cubes in $I_k$, $\mathcal{N}_k$, depends also only
on $\mathcal{C},N,\theta$. We denote $Q_{k,l}$, $l=1,\dots,\mathcal{N}_k$,
the corresponding cubes.

As before we
build then $F_k\subset I_k$ as the union of $Q_{k,l}\cap W_{Q_{k,l}}$
where $W_{Q_{k,l}}\in\mathcal{F}$ is a convex set containing the
center of $Q_{k,l}$, and $\widetilde F_k=\bigcup_{l=1}^{\mathcal{N}_k}W_{Q_{k,l}}$.
{We have then proved that i), ii) and iii) holds true with $k-1$ replaced by $k$.

Notice that,  thanks to~\eqref{eq:densmin},  $|I_k|\le |F_k|/\theta$.
Recall  that the sets $I_k$ and $\mathring{F}_1,\dots,\mathring{F}_{k-1}$
are disjoint. Hence
$|(\mathcal{C}\setminus (\cup_{i=1}^{k-1} F_i))  \setminus I_k| = 
|\mathcal{C}\setminus (\cup_{i=1}^{k-1} F_i)| - |I_k|$} so that
\[
|\mathcal{C}\setminus (\cup_{i=1}^{k-1} F_i)| \le \frac{|F_k|}{\theta}+2^{-k}.
\]
Since these sets are decreasing, and the $\mathring{F}_k\subset \mathcal{C}$ disjoint,
for $K\ge 1$,
\[
K |\mathcal{C}\setminus (\cup_{i=1}^{K} F_i)|
\le \sum_{k=2}^{K+1} |\mathcal{C}\setminus (\cup_{i=1}^{k-1} F_i)| \le \frac{|\mathcal{C}|}{\theta}+1
\]
hence
\[
|\mathcal{C}\setminus (\cup_{i=1}^{K} \widetilde{F}_i)|
\le |\mathcal{C}\setminus (\cup_{i=1}^{K} F_i)| \le \frac{\frac{|\mathcal{C}|}{\theta}+1}{K}.
\]
{If $K\ge (|\mathcal{C}|/\theta+1)/\delta$, \eqref{avanzapoco} holds true and the maximum number of sets
$W\in\mathcal{F}$ used in the construction, which is bounded
by $\sum_{k=1}^K\mathcal{N}_k$, depends only on $N,\delta,\theta,\mathcal{C}$}.
\end{proof}

{
\begin{remark}\label{rm:mzerophi}
A careful study of the previous proof shows that $\mathcal{N}$ only
depends on the convex set $\mathcal{C}$ through $|\mathcal{C}|$ and $\H^{N-1}(\partial\mathcal{C})$.
Hence, as these quantities, for $\mathcal{C}=W^\p(0,1)$,
as well as $\theta$ in~\eqref{qq1bis}
depend continuously on $\p$, one can deduce that
the constants  $M_0$, $h_0$ in Lemma~\ref{lm:crucial} can be chosen
as depending on $\p$ only through the ellipticity constants $a_1$, $a_2$
of \eqref{phia12}.
\end{remark}
}

\section{The crystalline mean curvature flow with a $\p$-regular mobility}\label{sec:phiregular}

Throughout this section we assume  the mobility $\psi$  satisfies the following   regularity assumption with respect to the metric induced by $\po$:
 \begin{definition}\label{def:phiregular}
We will say that a norm $\psi$ is $\p$-regular if  the associated Wulff shape $W^\psi(0,1)$ satisfies a uniform interior $\p$-Wulff shape condition, that is, if  there exists $\e_0>0$ with the following property:  for every $x\in \partial W^\psi(0,1)$ there exists $y\in W^\psi(0,1)$ such that  $W^\p(y, \e_0)\subseteq W^\psi(0,1)$ and $x\in \partial W^\p(y,\e_0)$.
 \end{definition}
{Notice that it is equivalent to saying that $W^\psi(0,1)$ is the
sum of a convex set and  $W^\p(0,\e_0)$, or equivalently that
$\psi(\nu)=\psi_0(\nu)+\e_0\p(\nu)$ for some convex function $\psi_0$.}
We will show that, under the above additional regularity assumption, the ATW scheme converges to a (generically) unique solution of the flow in the sense of Definition~\ref{Defsol}, see Subsections~\ref{subsec:convergence}~ and~\ref{subsec:levelset} below.
We start with some preliminary estimates.

\subsection{Evolution of  $\psi$-Wulff shapes and preliminary estimates}\label{EWSreg}
 
  We now analyze the minimizing movement of a $\psi$-Wulff shape $W^\psi(0, R)$, with $\psi$  $\p$-regular, that is, we assume $E_h(t)= W^\psi(0, R)$  for some time $t\geq 0$.

  By the regularity assumption in Definition~\ref{def:phiregular},
  %there exists $\e>0$ small enough such that, 
  setting $\overline R:=(\e_0 R)\land 1$, we have (with the notation introduced in \eqref{nota}) {that for every $0<r<\overline R$
  $$
  \left(\left(W^{\psi}(0, R)\right)^{\po}_{-\overline R}\right)^{\po}_{r}= \left(W^{\psi}(0, R)\right)^{\po}_{-[\overline R-r]}\,.
  $$ 
Since  for every $x\in \left(W^{\psi}(0, R)\right)^{{\po}}_{-\overline R}$ we have  $W^{\p}(x, \overline R)\subseteq W^{\psi}(0, R)$,
 from } the discrete comparison principle and the analysis performed in Subsection~\ref{EWS} it follows that 
 $$
\left(W^{\psi}(0, R)\right)^{\po}_{-[\overline R-r^{\overline R}(s)]}\subseteq E_h(s)
 $$
 for all $0\leq s-t\leq \frac{c_1{\overline R}^2}{4\big[2(N-1)+\|g\|_\infty\big]}$ and $h\leq c_1C(N){\overline R}^2$, where $r^{\overline R}$ is the function defined in~\eqref{rhs} (with $R$ replaced by $\overline R$).

Now we return to  an arbitrary discrete motion $E_h(\cdot)$.
If for some $(x,t)\in\R^N\times [0,T_h^*)$ we have $d_h(x,t)>R$ (see \eqref{discretevol}), then $W^\psi(x,R)\cap E_h(t)=\emptyset$. Thus, again by the discrete comparison principle and the results of Subsection~\ref{EWS}, we infer that 
$$
\left(W^{\psi}(0, R)\right)^{\po}_{-[\overline R-r^{\overline R}(s)]} \cap E_h(s)=\emptyset
 $$
 for all $0\leq s-t\leq \frac{c_1{\overline R}^2}{4\big[2(N-1)+\|g\|_\infty\big]}$ and $h\leq c_1C(N){\overline R}^2$.
Taking into account also $\eqref{c12}$ and the definition of $r^{\overline R}$, it follows that 
\begin{align*}
d_h(x, s)&\geq d_h(x, t)-c_2\left(\overline R-\sqrt{{\overline R}^2-2\frac{s-t}{c_1}\big[2(N-1)+\|g\|_\infty\big]}\right)\\
&= d_h(x, t)-\frac{c_2}{c_1}\frac{\big[4(N-1)+2\|g\|_\infty\big](s-t)}{\overline R+\sqrt{{\overline R}^2-2\frac{s-t}{c_1}\big[2(N-1)+\|g\|_\infty\big]}}\\
&\geq  d_h(x, t)-\frac{c_2}{c_1}\frac{\big[4(N-1)+2\|g\|_\infty\big](s-t)}{\overline R}\\
&=d_h(x, t)-\frac{c_2}{c_1}\frac{\big[4(N-1)+2\|g\|_\infty\big]}{(\e_0 R)\land 1}(s-t)\,
\end{align*}
 for all $0\leq s-t\leq \frac{c_1[(\e_0^2R^2)\land 1]}{4\big[2(N-1)+\|g\|_\infty\big]}$ and $h\leq c_1C(N)[(\e_0^2R^2)\land 1]$.

Letting $R\nearrow d_h(x,t)$ we obtain
\beq\label{straponzina21}
d_h(x,s)\geq d_h(x, t)-\frac{c_2}{c_1}\frac{\big[4(N-1)+2\|g\|_\infty\big]}{(\e_0 d_h(x, t))\land 1}(s-t)\,
\eeq
 for all $0\leq s-t\leq \frac{c_1\big[\big(\e_0^2d_h^2(x, t)\big)\land 1\big]}{4\big[2(N-1)+\|g\|_\infty\big]}$ and $h\leq c_1C(N)[(\e_0^2d_h^2(x, t))\land 1]$, whenever $d_h(x, t)>0$.
 
By an entirely similar argument if  $d_h(x, t)<0$, then we obtain
\beq\label{straponzina21biz}
d_h(x,s)\leq d_h(x, t)+\frac{c_2}{c_1}\frac{\big[4(N-1)+2\|g\|_\infty\big]}{(\e_0 |d_h(x, t)|)\land 1}(s-t)\,
\eeq
 for all $0\leq s-t\leq \frac{c_1\big[\big(\e_0^2d_h^2(x, t)\big)\land 1\big]}{4\big[2(N-1)+\|g\|_\infty\big]}$ and $h\leq c_1C(N)[(\e_0^2d_h^2(x, t))\land 1]$.

\subsection{Convergence of the ATW scheme}\label{subsec:convergence} For every $h>0$ let  $E_h$ be the discrete evolution defined in \eqref{discretevol}. Clearly, we may extract a  subsequence $\{E_{h_l}\}_{l\in \N}$ such that

\[
E_{h_l}\stackrel{\mathcal K}{\longrightarrow} E\qquad\text{and}\qquad 
{(\mathring{E}_{h_l})}^c\stackrel{\mathcal K}{\longrightarrow} A^c
\]
for a suitable closed set $E$ and a suitable open set $A\subset  E$. 
Define $E(t)$ and $A(t)$ as in \eqref{discretevol}. 

Observe that if $E(t)=\emptyset$ for some $t\ge 0$,
then \eqref{straponzina21} implies that $E(s)=\emptyset$ for all $s\ge t$
so that we can define, as in Definition~\ref{Defsol}, the extinction
time $T^*$ of $E$, and similarly the extinction  time ${T'}^*$ of
 $A^c$. Notice that at least one between  $T^*$ and ${T'}^*$ is $+\infty$. %% We denote by $T_{min}^*$ the smallest one.  never used?
Possibly extracting a further subsequence, we have the following result, which can be proven arguing exactly as in \cite[Proof of Proposition 4.4]{CMP4},
using now \eqref{straponzina21} and \eqref{straponzina21biz}.

\begin{proposition}\label{prop:E}
There exists a countable set $\mathcal N\subset (0, +\infty)$ such that
${d_{h_l}}(\cdot, t)^+\to \dist(\cdot, E(t))$   and  $d_{h_l}(\cdot, t)^-\to \dist(\cdot,A^c)$ locally uniformly  for all $t\in (0, +\infty) \setminus \mathcal{N}$.
Moreover, $ E$ and $ A^c$ satisfy the continuity properties (b) and (c) of Definition~\ref{Defsol}. 
Finally, $E(0)= E^0$ and  $A(0)=\mathring{E^0}$. 
\end{proposition}

\begin{theorem}\label{th:ATW}
The set $E$ is a superflow in the sense of Definition~\ref{Defsol}
with initial datum $E^0$,
while $A$ is a subflow with initial datum ${E}^0$.
\end{theorem}
\begin{proof}
Points \textit{(a), (b)} and \textit{(c)} of Definition~\ref{Defsol} follow from Proposition~\ref{prop:E}.
It remains to show \textit{(d)}.
{We will use the notation in \eqref{discretevol}}.  Possibly extracting a further subsequence and setting $z_{h_l}(\cdot, t):=0$ for $t>T^*_{h_l}$ if $T^*_{h_l}<T^*$, we may assume
that $z_{h_l}$ converges  weakly-$*$ in $L^\infty(\R^N\times (0,T^*))$  to some vector-field $z$
satisfying {$\po(z)\le 1$} almost everywhere.
Recall that by~\eqref{eq:ineqd} we have $u_{h_l}^{k+1} \le (1+Lh_l) d^{\pso}_{E_{h_l}^{k+1}}$,
whenever $\dd_{E_{h_l}^{k+1}}\ge 0$. In turn,  it follows from~\eqref{eq:iterk} that 
\begin{equation}\label{eq:ineqdisc}
 \Div z_{h_l}^{k+1} + \frac{1}{h_l}  \int_{kh_l}^{(k+1)h_l} g(\cdot,s) \, ds  \le \frac{(1+Lh_l) d^{\pso}_{E_{h_l}^{k+1}} - d^{\pso}_{E_{h_l}^k}}{{h_l}}.
\end{equation}
Consider
a nonnegative test function
$\eta\in C_c^\infty((\R^N\times (0,T^*))\setminus E)$.
If $l$ is large enough, then the distance of the support of $\eta$
from $E_{h_l}$ is bounded away from zero. In particular,  $d_{h_l}$ is finite ({as a consequence on \eqref{straponzina21}})
and positive on $\spt\eta$. We deduce from~\eqref{eq:ineqdisc} that
\begin{multline*}
0\le \int \!\!\! \int \eta(x,t)\Bigg[\frac{d_{h_l}(x,t+{h_l})-d_{h_l}(x,t)}{h_l}-\Div z_{h_l}(x,t+{h_l})
\\
-\frac{1}{h_l}  \int_{[\frac{t}{h_l}]h_l}^{([\frac{t}{h_l}] +1)h_l} g(x,s) \, ds + L d_{h_l}(x,t+{h_l}) \Bigg]
dt dx 
\\ 
=-\int  \!\!\!  \int \Bigg[\frac{\eta(x,t)-\eta(x,t-h_l)}{h_l}d_{h_l}(x,t)- z_{h_l}(x,t+h_l)\cdot\nabla \eta(x,t)
\\
- \eta(x,t)  \Big( \frac{1}{h_l} \int_{[\frac{t}{h_l}]h_l}^{([\frac{t}{h_l}] +1)h_l} g(x,s) \, ds + L d_{h_l}(x,t+{h_l})  \Big) \Bigg]
\,dt dx.% \ge 0.
\end{multline*}
Passing to the limit $l\to\infty$ we obtain \eqref{eq:supersol} {with $M=L$.}

Next, we establish an upper bound for $\Div z_{h}$ away from $E^k_{h}$. To this aim
let $x\in \R^N\setminus E^k_{h}$ be such that $d^{\pso}_{E^k_{h}}(x)=: R>0$.

{There exists $\xi \in E^k_h$ with $x\in\partial W^\psi(\xi,R)$, and recalling 
Definition~\ref{def:phiregular} there is $\bar x$ such that, setting $\bar R:=(\e_0 R)\land 1$, 
 $W^\p(\bar x,\bar R)\subset W^\psi(\xi,R)$ and $x\in\partial W^\p(\bar x,\bar R)$. In particular, $d^{\pso}_{E^k_h}-R\le \pso(\cdot-\xi)-R
{ \le d^{\pso}_{W^\psi(\xi,R)} }
\le d^{\pso}_{W^\p(\bar x,\bar R)}$.} Using~\eqref{c12W}, one has for all $y\in\R^N$
\[
d^{\pso}_{E^k_h}(y) 
\le R+c_2(\po(y-\bar x)-\bar R) \lor c_1(\po(y-\bar x)-\bar R).
\]
Thanks to~\eqref{c12},
\[
\int_{kh}^{ (k+1)h}\!\!g(y,s)\, ds\le
\int_{kh}^{ (k+1)h}\!\!g(x,s)\, ds+Lh\pso(y-x)
\le
\int_{kh}^{ (k+1)h}\!\!g(x,s)\, ds+Lh c_2\po(y-x),
\]
hence, since $\po(y-x)\le \po(y-\bar x)+\bar R$, summing the
two previous inequalities we obtain:
\begin{multline*}
d^{\pso}_{E^k_h}(y) +\int_{kh}^{(k+1)h}\!\!g(y,s)\, ds
\\
\le R+{2}c_2Lh\bar R+c_2(1+Lh)(\po(y-\bar x)-\bar R) \lor (c_1+c_2Lh)(\po(y-\bar x)-\bar R)+\int_{kh}^{ (k+1)h}\!\!g(x,s)\, ds.
\end{multline*}
As a consequence (\textit{cf} Lemma~\ref{lm:explicit}), 
\begin{multline*}
u^{k+1}_h(y)\le R + {2}c_2Lh\bar R+\int_{kh}^{ (k+1)h}\!\!g(x,s)\, ds
\\+ c_2(1+Lh)(\po(y-\bar x)-\bar R) \lor (c_1+c_2Lh)(\po(y-\bar x)-\bar R)
 + \frac{h(N-1)}{\po(y-\bar x)}
\end{multline*}
if $\po(y-\bar x)\ge \sqrt{h(N+1)/(c_1+c_2Lh)}$ and as long as this quantity
is less than {$\frac{\bar R}{\sqrt{N+1}}$.} Evaluating this inequality at $y=x$, we deduce that
if $h\le C(N)\bar R^2$,
\begin{align*}
u^{k+1}_h(x) & \le R+ {2}c_2Lh\bar R
+\int_{kh}^{ (k+1)h}\!\!g(x,s)\, ds + \frac{h(N-1)}{\bar R} \\
& \le d^{\pso}_{E^k_{h}}(x) + {2}c_2Lh
+\int_{kh}^{ (k+1)h}\!\!g(x,s)\, ds + \frac{h(N-1)}{(\e_0R)\land 1} ,
\end{align*}
as $\bar R\le 1$. %%the optimum would be $\min\{\e_0 R,\sqrt{(N-1)/(c_2L)}\}$
Thanks to~\eqref{eq:iterk}, it follows
\begin{equation}\label{numero} 
\Div z^{k+1}_h(x)\le  {2}c_2L + \frac{N-1}{(\e_0 d^{\pso}_{E^k_{h}}(x))\land 1}.
\end{equation}
In the limit $h_l\to 0$, we deduce that $\Div z$
is a Radon measure in $\R^N\times (0,T^*)\setminus E$, and 
$(\Div z)^+\in L^\infty(\{(x,t)\in\R^N\times (0,T^*)\,:\, d(x,t)\geq\delta\})$ for every $\delta>0$.

We now provide  a lower ($h$-dependent) bound for $\Div z_{h}$.  To this aim, note that if $d^{\pso}_{E^k_{h}}(x)=:R>0$, then 
$d^{\pso}_{E^k_{h}}\geq R-\pso(\cdot-x)\geq R- c_2\po(\cdot-x)$. Thus, by comparison with the explicit solution given by \eqref{eq:explicitpoh}
(with a change of sign), {and using that $(c_2 \po)_h =c_2 \po_{h/c_2}$, we get}
$$
u_h^{k+1}(x)\geq R-c_2 \po_{\frac{h}{c_2}}(0) - h \|g\|_\infty \ge R- \frac{3N \sqrt{ c_2 h}}{\sqrt{N+1}}, 
$$
for $h=h_l$ small enough. In turn, by \eqref{eq:iterk}, we deduce
$$
\Div z_{h}^{k+1}\geq -\frac{1}{\sqrt h} \frac{3N \sqrt{c_2}}{\sqrt{N+1}} \qquad\text{a.e. in }\{x\,:\, \dd_{E^k_h}(x)>0\}.
$$
Combining the above inequality with \eqref{numero} and  using \eqref{eq:iterk} again,  we deduce that 
for any $\delta>0$
$$
\|u_{h_l}(\cdot, t)-d_{h_l}(\cdot, t-h_l)\|_{L^{\infty}(\{x\,:\,d_{h_l}(x,t-h_l)\geq\delta\})}\leq \sqrt{h_l} \frac{3N \sqrt{c_2}}{\sqrt{N+1}} + o(\sqrt{h_l}),
$$
provided that $l$ is large enough.  In particular, recalling  the convergence properties
of $E_{h_l}$ and $d_{h_l}$ (see also \cite[Equation (4.9)]{CMP4}), we deduce that 
{for all $t\in (0, T^*)\setminus\mathcal{N}$ (where recall that $\mathcal{N}$ is introduced in Proposition \ref{prop:E}),}
\begin{equation}\label{elleuno}
u_{h_l}\to d\qquad\text{a.e. in }\R^N\times (0,T^*)\setminus E,
\end{equation}
with the sequence $\{u_{h_l}\}$ locally (in space and time) uniformly bounded. 

Now, with \eqref{numero} and \eqref{elleuno} at hand, 
we proceed as in the final part of the proof of~\cite[Theorem 4.5]{CMP4}
to show that 
  $\p(\nabla d)= z\cdot \nabla d$, which will imply that $z\in\partial\p(\nabla d)$ (as clearly, $\po(z)\le 1$ a.e.). For this, it is enough to show
that $\p(\nabla d)\le z\cdot \nabla d$.
On one hand, since $z_{h_l}\in\partial\p(\nabla u_{h_l})$, one has for
a nonnegative test function $\eta\in C_c^\infty(E^c;\R_+)$, 
using $u_{h_l}\to d$, that
\[
\int\int \eta\, \p(\nabla d)dxdt\le \liminf_l \int\int \eta\,\p(\nabla u_{h_l})dxdt=
\liminf_l \int\int \eta\, z_{h_l}\cdot\nabla u_{h_l}dxdt.
\]
On the other hand,
\[
\int\int \eta\, z_{h_l}\cdot\nabla u_{h_l}dxdt=
\int\int \eta\, z_{h_l}\cdot \nabla d \,dxdt+
\int\int \eta\, z_{h_l}\cdot \nabla (u_{h_l}-d)dxdt
\]
and as $z_{h_l}\stackrel{*}{\rightharpoonup} z$, we obtain that $\p(\nabla d)\le z\cdot\nabla d$~a.e.,
provided we can show that
\[
\lim_l\int\int \eta\, z_{h_l}\cdot \nabla (u_{h_l}-d)dxdt = 0.
\]
The proof is as in~\cite{CMP4}: we introduce $m_l(t):=\min_{x\in\spt\eta(\cdot,t)}(u_{h_l}(x,t)-d(x,t))$ which is bounded
and goes to zero for {almost} all $t$, and write
\begin{multline*}
\int\int \eta\, (z_{h_l}\cdot \nabla (u_{h_l}-d))dxdt = 
\int\int \eta\, (z_{h_l}\cdot \nabla (u_{h_l}-d-m_h))dxdt 
\\= 
-\int\int z_{h_l}\cdot\nabla\eta\, (u_{h_l}-d-m_h)dxdt
-\int\int \eta\,(u_{h_l}-d-m_h)\Div z_{h_l}dxdt.
\end{multline*}
The first integral in the right-hand side clearly goes to zero,
while, using~\eqref{numero} and $u_{h_l}-d-m_h\ge 0$, the second is bounded
from above by a quantity which vanishes as $l\to \infty$. We deduce
that $\liminf_l\int\int \eta\, z_{h_l}\cdot \nabla (u_{h_l}-d)dxdt \ge 0$,
and the proof of the reverse inequality is identical.

It follows that $z\in \partial\p(\nabla d)$ a.e.~in
$\R^N\times (0,T^*)\setminus E$.
This concludes the proof that $E$ is a superflow.
The proof that $A$ is a subflow is identical.
\end{proof}

\begin{remark}[Stability of sub- and superflows]\label{rm:stability}
From the proof of Theorem~\ref{th:ATW}, we note that  the minimizing movement scheme provides a superflow  such 
that the corresponding field $z$ satisfies (see \eqref{numero})
\beq\label{eq:stability}
\Div z\leq {2}c_2L +\frac{(N-1)}{(\e_0 R)\land 1}\qquad\text{ a.e. in }\{d\geq R\}\,,
\eeq
where $c_2$ is the constant appearing in \eqref{c12} , $L$ is the Lipschitz constant of $g$, and $\e_0$ is given in 
Definition~\ref{def:phiregular}.
Moreover, from \eqref{straponzina21} we deduce that
\beq\label{straponzina21ter}
d(x,s)\geq d(x, t)-\frac{c_2}{c_1}\frac{\big[4(N-1)+2\|g\|_\infty\big]}{(\e_0 d(x, t))\land 1}(s-t)\,
\eeq
 for all $0\leq s-t\leq \frac{c_1\big[\big(\e_0^2d^2(x, t)\big)\land 1\big]}{4\big[2(N-1)+\|g\|_\infty\big]}$, whenever $d(x, t)>0$.   
 An analogous statement clearly holds also for the subflow provided by the ATW scheme.  
 
We remark that {thanks to these estimates,}
the following  stability property holds:
Let $\p_n\to \p$ and $\psi_n\to \psi$ and assume
that $\psi_n$ is $\p_n$-regular  uniformly in $n$: there exists $\e_0>0$ such that the $\psi_n$-Wulff shape satisfies a uniform inner $\p_n$-Wulff shape condition with radius $\e_0$ for all $n$ (see Definition~\ref{def:phiregular})

For every $n$ let $E^n$ be a superflow as in Definition~\ref{Defsol}, with $\p$ and $\psi$ replaced by $\p_n$ and $\psi_n$, respectively, and with initial datum $E^n_0$. Denote by   $d_n$ the corresponding distance function, that is, $d_n(\cdot ,t):=\dist^{\pso_n}(\cdot, E^n(t))$  in $\R^N\setminus E^n(t)$, and by $z_n$ the corresponding 
Cahn-Hoffmann field given by Definition~\ref{Defsol}, and assume that \eqref{eq:stability} and \eqref{straponzina21ter}  hold with $z$ and $d$ replaced by $z_n$ and $d_n$ (and again with  $c_1$, $c_2$, $L$, and $\e_0$ independent of $n$). Finally, assume that 
$E^n \stackrel{\mathcal K}{\longrightarrow} E$ and $E^n_0 \stackrel{\mathcal K}{\longrightarrow} E_0$. Then,  $E$ is a superflow
with respect to the anisotropy $\p$ and the mobility $\psi$ satisfying   $E(0)=E_0$. 
This follows by the same arguments employed in the proofs 
of Proposition~\ref{prop:E}
(see also~\cite[Proof of Proposition 4.4]{CMP4}) and Theorem~\ref{th:ATW}.
Analogous stability properties hold also for subflows.   
\end{remark}

\subsection{Existence and uniqueness of the level set flow}\label{subsec:levelset}
The convergence theorem proved in the previous subsection combined with the comparison principles established in 
Subsection~\ref{sec:comp} yields  existence and uniqueness of the level set formulation of the crystalline curvature flow{, when $\psi$ is $\p$-regular}.
In the following we briefly set up the discrete version of such  formulation and we give the precise statements.  

Let $u^0:\R^N\to \R$ be a uniformly continuous function. Let $E^{0,\lambda}:= \{u^0 \le \lambda\}$ and let $E_{\lambda,h}$ be the corresponding time discrete evolutions, defined according with \eqref{discretevol} with $E^0$ replaced by $E^{0,\lambda}$. {(We recall that if in general there is no uniqueness of the
ATW scheme, our definition in Section~\ref{subsec:ATW} provides a well-defined selection.)}

We introduce the level set discrete evolution $u_h:\R^N\times \R \to \R$ defined by
\begin{equation}\label{defls}
u_h(x,t):= \inf\{\lambda\in\R\,: \, x\in E_{\lambda,h}(t)  \}.
\end{equation}
{
(We warn the reader that the discrete level set function $u_h$ defined above does not coincide with the discrete total variation flow  function (already denoted by $u_h$) defined in \eqref{discretevol}.)}
Note that by construction
\beq\label{eq:bc}
\{u_h(\cdot, t)<\lambda\}\subseteq E_{\lambda,h}(t)\subseteq \{u_h(\cdot, t)\leq\lambda\}\,.
\eeq

Let $\omega$ denote an increasing modulus of continuity for  $u^0$, with respect to the metric induced by $\pso$. Thus, in particular,
if $\lambda_1<\lambda_2$ we have
$$
\dist^{\pso}(E^{0,\lambda_1}, \R^N\setminus E^{0, \lambda_2})\geq \omega^{-1}(\lambda_2-\lambda_1)\,.
$$
Let $L>0$ be the spatial Lipschitz { constant of $g$}  with respect to $\pso$ and choose $\bar h>0$ so small that 
$(1-Lh)^{-\frac1{Lh}}\leq 2 \mathrm{e}$ for all $h\in (0, \bar h)$. 
By Lemma~\ref{discong} below {(with $\eta=\psi$,   $\beta=1$, $M=L$,  $g_1=g_2=g$  and  $c=0$)} we deduce that 
$$
\dist^{\pso}(E_{\lambda_1, h}(t), \R^N\setminus E_{ \lambda_2,h}(t))\geq {\omega^{-1}(\lambda_2-\lambda_1)}(1-Lh)^{[\frac{t}{h}]}\geq \omega^{-1}(\lambda_2-\lambda_1)\mathrm{(2e)}^{-Lt}
$$
for all $t> 0$ and $h\in (0, \bar h)$. 
In turn, it easily follows that 
{$\widehat \omega :\lambda\mapsto\omega\bigl(\mathrm{(2e)}^{Lt}\lambda \bigr)$ }
is a spatial modulus of continuity for 
$u_h(\cdot, t)$. 
As for the continuity in time, we have:
\begin{lemma}\label{lemUCT}
For any $\e>0$, there exists $\tau>0$ and $h_0>0$
 (depending on $\e$) such that for all  $|t-t'|\le \tau$ and $h\le h_0$ we have
$|u_h(\cdot,t)-u_h(\cdot,t')|< \e$.
\end{lemma}

The proof of the lemma follows by standard comparison arguments with the evolution of the $\p$-Wulff shape, whose extinction time can be estimated (see Remark~\ref{dipende} below). We refer to \cite[Lemma 6.13]{CMP3} for the details.

\begin{remark}\label{dipende}
Let us remark that the extinction time of a $\p$-Wulff shape of radius $R$, evolving according to the forced mean crystalline curvature flow with forcing term $g$ and mobility $\psi$, is bounded away from zero by a constant which depends only on $R$, the infinity norm of $g$ and the constant $c_1$ in \eqref{c12} (see section~\ref{EWS}). In turn, $h_0$ and $\tau$ depend only on  $\omega$, $\e$,   $\|g\|_\infty$ and  $c_1$.    
\end{remark}

We are ready to proof the main result of this section. 
\begin{theorem}\label{th:phiregularlevelset} 
Let $\psi$, $g$, and $u^0$ be a $\p$-regular mobility,  an admissible forcing term, and a uniformly continuous function on $\R^N$, respectively. Then the following holds: 

{\rm (i) (Existence and uniqueness)} There exists a unique solution $u$ to the level set flow   with initial datum  $u^0$,  in the sense of Definition~\ref{deflevelset1}.

 {\rm (ii) (Approximation via minimizing movements)} The solution  $u$ is the locally uniform limit in $\R^N\times [0,+\infty)$, as $h\to 0^+$,  of the    level set minimizing movements  $u_h$ defined in \eqref{defls}.

 {\rm (iii) (Properties of the level set flow)}  For all but countably many $\lambda \in \R$, the fattening phenomenon does not occur and,  in fact,
 {$\partial  \{(x,t)\,:\, u(x, t) < \lambda\} = \{(x,t)\,:\, u(x, t) = \lambda\}$, i.e., 
%\begin{align}\label{eq:nonfattening}
%& \{(x,t)\,:\, u(x, t) < \lambda\} = \mathrm{Int\,}(\{(x,t)\,:\, u(x, t) \le \lambda\})\,, \vspace{5pt} \\
%\label{eq:nonfatteningbis}
%& \overline {\{(x,t)\,:\, u(x, t) < \lambda\}} = \{(x,t)\,:\, u(x, t) \le \lambda\}\,.
%\end{align}
{
 \beq
 \begin{array}{rcl}\label{eq:nonfattening}
 \{(x,t)\,:\, u(x, t) < \lambda\}& = &\mathrm{Int\,}(\{(x,t)\,:\, u(x, t) \le \lambda\})\,, \vspace{5pt} \\
 \overline {\{(x,t)\,:\, u(x, t) < \lambda\}} &=& \{(x,t)\,:\, u(x, t) \le \lambda\}\,.
 \end{array}
 \eeq
}

Moreover,  for every  $\lambda$ such that \eqref{eq:nonfattening} holds true} the  sublevel set 
$\{(x,t)\,:\, u(x, t) \le \lambda\}$ is the unique  solution
to  \eqref{oee} in the sense of Definition~\ref{Defsol}, with initial datum $E^{0,\lambda}$, and 
\beq\label{eq:nonfatteningbiz}
E_{\lambda, h}\stackrel{\mathcal K}{\longrightarrow} \{(x,t)\,:\, u(x, t) \le \lambda\}\quad\text{and}\quad
{ (\mathrm{Int\, }E_{\lambda, h})^c}\stackrel{\mathcal K}{\longrightarrow} \{(x,t)\,:\, u(x, t) \geq \lambda\}\,.
\eeq
Finally, for all $\lambda\in \R$ the sets $\{(x,t)\,:\, u(x, t) \le \lambda\}$ and $\{(x,t)\,:\, u(x, t) < \lambda\}$ are respectively the maximal superflow and minimal sublow with initial datum $E^{0,\lambda}$.
\end{theorem}
\begin{proof}
The arguments rely on  Theorems~\ref{th:compar} and \ref{th:ATW},  and are somewhat standard
(see for instance~\cite{ChNoIFB,CMP3}). For the reader's convenience we outline below a self-contained proof.

\noindent {\it Step 1.} (Convergence) By the equicontinuity properties observed before the family  $\{u_h\}$ 
is relatively compact with respect to the local uniform convergence in $\R^N\times[0,+\infty)$. Observe now that if $u$ is a  cluster point for $\{u_h\}$, then by \eqref{eq:bc} and by Theorem~\ref{th:ATW} for all $\lambda\in \R$ there exist  a superflow $E_\lambda$ and a subflow $A_\lambda$, with initial datum $u^0$,  such that
\beq\label{eq:bc2}
\{(x,t)\,:\, u(x, t)<\lambda\}\subseteq A_\lambda\subseteq  E_{\lambda}\subseteq \{(x,t)\,:\, u(x, t)\leq \lambda\}\,.
\eeq
Let $u_1$, $u_2$ be two cluster points for $\{u_h\}$ and for any $\lambda\in \R$ let $A^i_\lambda$, $E^i_\lambda$ be as in 
\eqref{eq:bc2}, with $u$ replaced by $u_i$, $i=1,2$. 
Fix $\lambda<\lambda'$.  Since 
$$
\dist^{\pso}(E^1_{\lambda}(0), \R^N\setminus A^2_{\lambda'}(0))\geq 
\dist^{\pso}(\{u^0\leq \lambda\}, \R^N\setminus \{u^0< \lambda'\})>0\,,
$$
where the last inequality follows from the uniform continuity of $u^0$, 
it follows from Theorem~\ref{th:compar} that  $E^1_{\lambda}(t) \subseteq A^2_{\lambda'}(t)$ and, in turn,  from \eqref{eq:bc2} 
$$
\{u_1(\cdot, t)<\lambda\}\subseteq E^1_{\lambda}(t) \subseteq A^2_{\lambda'}(t) \subseteq \{ u_2(\cdot , t)\leq \lambda'\}\
$$
for all $t>0$. The arbitrariness of $\lambda<\lambda'$ in the above chain of inequalities clearly implies that $u_2\leq u_1$. Exchanging the role of $u_1$ and $u_2$, we get in fact $u_1=u_2$. Thus, there exists a unique  cluster point $u$ and  $u_h\to u$  locally uniformly  in $\R^N\times [0,+\infty)$ as $h\to 0^+$.

\noindent{\it Step 2.} {(Proof of  \eqref{eq:nonfattening})}
%, \eqref{eq:nonfatteningbis})}
For $\lambda\in \R$ set 
$K_\lambda:=\{(x,t)\,: u(x,t)\leq \lambda\}$. Since for any $(x,t)$ the map $\lambda\mapsto \dist((x,t), K_\lambda)$ is non-increasing (here $\dist$ denotes the Euclidean distance in $\R^N\times [0, +\infty)$) and since $\dist (\cdot, K_\lambda)$ is 
(Lipschitz) continuous, we easily deduce the existence of a countable set $N_1\subset \R$ such that for all $\lambda_0\in \R\setminus N_1$ the map $\lambda\mapsto \dist((x,t), K_\lambda)$ is continuous at $\lambda_0$ for all
 $(x,t)\in \R^N\times [0,+\infty)$. In turn, by equicontinuity, it follows that $\dist(\cdot, K_\lambda)\to \dist(\cdot, K_{\lambda_0})$ 
  locally uniformly  in $\R^N\times [0,+\infty)$, or equivalently, 
  $K_\lambda \stackrel{\mathcal K}{\longrightarrow} K_{\lambda_0}$  as $\lambda\to \lambda_0$. In particular, by taking $\lambda_n\nearrow \lambda_0$ and using that $K_{\lambda_n}\subset\{u<\lambda_0\}$ for all $n$, we deduce 
  $\overline {\{(x,t)\,:\, u(x, t) < \lambda_0\}} = \{(x,t)\,:\, u(x, t) \le \lambda_0\}$. Analogously,  one can show that there exists a countable set 
  $N_2\subset\R$ such that for all $\lambda_0\not\in N_2$ we have 
  {$\{(x,t)\,:\, u(x, t) \ge \lambda\} \stackrel{\mathcal K}{\longrightarrow} \{(x,t)\,:\, u(x, t) \ge \lambda_0\}$} as $\lambda\to \lambda_0$, so that $\overline {\{(x,t)\,:\, u(x, t) >\lambda_0\}} = \{(x,t)\,:\, u(x, t) \geq\lambda_0\}$. We conclude that for all $\lambda\not\in N_0:=N_1\cup N_2$, {\eqref{eq:nonfattening} %and \eqref{eq:nonfatteningbis}
holds.}

\noindent{\it Step 3.} (Conclusion)  Fix $\lambda\in \R\setminus N_0$ and let $E_\lambda$ and $(A_\lambda)^c$ be Kuratowski limits along a common subsequence of $E_{\lambda, h}$ and $(\mathrm{Int\, }E_{\lambda, h})^c$, respectively. Then, by Theorem~\ref{th:ATW}, $E_\lambda$ and $A_\lambda$ are a superflow and subflow, respectively, with initial datum $u^0$. Moreover,  
\eqref{eq:bc2} holds. 
Thus, recalling  \eqref{eq:nonfattening}, 
$$
A_\lambda=\{(x,t)\,:\, u(x, t) < \lambda\}\,, \quad E_\lambda=\{(x,t)\,:\, u(x, t) \leq \lambda\}\, \quad \text{and}\quad A_\lambda=\mathrm{Int\, }E_\lambda\,. 
$$
This shows that $E_\lambda$ is a solution to  the curvature flow  with initial datum 
$E^{0,\lambda}$. 

Let now $E'$ be any superflow with initial datum $E^{0,\lambda}$. Then, for  all $\lambda'>\lambda$, with $\lambda'\not\in N_0$, thanks to Theorem~\ref{th:compar} we easily deduce that $E'\subset \{(x,t)\,:\, u(x,t)<\lambda'\}$. Thus, 
{$E'\subset E_\lambda$}. Analogously, if  $A'$ is a subflow with initial datum $E^{0,\lambda}$
one has  $\{(x,t)\,:\, u(x,t)\leq  \lambda'\}\subset A'$ for all $\lambda'<\lambda$, with $\lambda'\not\in N_0$, and thus
{$A_\lambda \subset A'$. 
Therefore, we have
\begin{align*}
& E' \subseteq E_\lambda = \bar A_\lambda \subseteq \bar{A'} \subseteq E',\\
&  A' = \mathrm{Int\, } E' \subseteq \mathrm{Int\, } E_\lambda = A_\lambda \subseteq A'.
\end{align*}
This shows that $E=E', A=A'$, i.e., the uniqueness of the solution to \eqref{oee}, starting from  $E^{0,\lambda}$. }
The same argument above shows that for all $\lambda\in \R$  if  $E'$ is a superflow  with initial datum $E^{0,\lambda}$, then
 $E'\subset \{(x,t)\,:\, u(x,t)\leq \lambda\}$. 
Let now $\lambda_n\searrow \lambda$, $\lambda_n\not\in N_0$ for all $n$. Since  
$\{(x,t)\,:\, u(x, t) \le \lambda_n\} \stackrel{\mathcal K}{\longrightarrow} \{(x,t)\,:\, u(x, t) \le \lambda\}$ it easily follows 
from the stability property stated in Remark~\ref{rm:stability} that   $\{(x,t)\,:\, u(x, t) \le \lambda\}$ is itself a superflow, thus the maximal superflow with initial datum $E^{0,\lambda}$. Analogously, one can show that $\{(x,t)\,:\, u(x,t)<  \lambda\}$ is the minimal subflow with initial datum $E^{0,\lambda}$.

Finally,  the uniqueness of the level set flow follows from the comparison principle proved in  Theorem~\ref{th:lscomp}. 
\end{proof}
% {
% \begin{remark}[Continuity in time and no fattening]
% Let $E$ be any solution to  \eqref{oee}.  By Lemma \ref{lem:uniformcontrol} (Applied to $E$ and to $\overline{E^c}$) it easily follows that the function $t\to E(t)$ is continuous with respect to the Kuratowski, and in fact with respect to the Hausdorff convergence. Moreover, by comparison it easily follows that $A(t) = \text{Int } E(t)$  for all $t>0$. 
% \end{remark}
% }
We conlcude this section with the following remarks.
\begin{remark}[Independence of the initial level set function]\label{rm:geo}
From the minimality and the maximality properties stated at the end of Theorem~\ref{th:phiregularlevelset}, we immediately deduce that if $\{u^0<\lambda\}=\{v^0<\lambda\}$, then $\{u(\cdot, t)<\lambda\}=\{v(\cdot, t)<\lambda\}$ for all $t>0$. Analogously, if 
$\{u^0\leq \lambda\}=\{v^0\leq \lambda\}$, then  $\{u(\cdot, t)\leq \lambda\}=\{v(\cdot, t)\leq \lambda\}$ for all $t>0$.
\end{remark}
\begin{remark}[Stability of level set flows with respect to varying anisotropies and mobilities]\label{rm:stability2}
Let $\{\p_n\}$ and $\{\psi_n\}$ be sequences of anisotropies and mobilities, respectively, such that   $\psi_n$ is $\p_n$-regular  uniformly in $n$
(\textit{cf} Remark~\ref{rm:stability}).
Assume also that 
 $\p_n\to \p$ and  $\psi_n\to \psi$. Let $u_n$ be the unique level set solution in the sense of Definition~\ref{deflevelset1},
with $\p$ and $\psi$ replaced by $\p_n$ and $\psi_n$, respectively, and with initial datum $u^0$. Then $u_n\to u$ locally uniformly, where $u$ is the unique level set solution in the sense of Definition~\ref{deflevelset1}, with anisotropy $\p$,  mobility $\psi$ and initial datum $u^0$. 

To see this, we start by observing that the sequence $\{u_n\}$ is equicontinuous in $\R^N\times[0, T]$ for all $T>0$ (see the discussion at the beginning of Subsection~\ref{subsec:levelset} and before Definition~\ref{deflevelset1}). Thus, up to a not relabeled 
subsequence, we may assume that $u_n\to u$ locally uniformly in $\R^N\times[0,+\infty)$. It is enough to show that $u$ is  a solution in the sense of Definition~\ref{deflevelset1}, since then we conclude by uniqueness.  Let now $N_0$ be a countable set such that 
if $\lambda\not\in N_0$, then \eqref{eq:nonfattening} holds for $u$ and for $u_n$ for all $n$.  Set $E^n:=\{u_n\leq\lambda\}$ and
$d_n(\cdot, t):=\dist^{\pso_n}(\cdot, E^n(t))$ in $\R^N\setminus E^n(t)$. 
%By the assumption on $\p_n$, $\psi_n$, 
{By Theorem~\ref{th:phiregularlevelset} we have that $E^n$ is a superflow with anisotropy $\p_n$, mobility $\psi_n$ and initial datum $E^{0,\lambda}$ for all $n$. %$d_n$ satisfies \eqref{straponzina21ter} with constants independent of $n$.
%, and there exists a corresponding Cahn-Hoffmann field $z_n$ as in Definition~\ref{Defsol} such that \eqref{eq:stability} holds, again with constants independent of $n$. 
Extracting a further subsequence, if needed, we may also assume that 
$E^n\stackrel{\mathcal K}{\longrightarrow} E$ and  $(\mathrm{Int\,}{E}^n)^c\stackrel{\mathcal K}{\longrightarrow} A^c$ for suitable  $E$ and $A$  such that 
$$
\{u<\lambda\}\subseteq A\subseteq E\subseteq \{u\leq\lambda\}\,.
$$
But then, recalling \eqref{eq:nonfattening}, $\{u<\lambda\}=A$, $E=\{u\leq\lambda\}$, and $A=\mathrm{Int\, }E$. Moreover,  by Remark~\ref{rm:stability}, $E$ is a superflow with anisotropy $\p$, mobility $\psi$ and initial datum $E^{0,\lambda}$.
}
 Analogously, one can show that $A$ is  subflow with anisotropy $\p$, mobility $\psi$ and initial datum $E^{0,\lambda}$. We conclude that 
$E=\{u\leq\lambda\}$ is a solution in the sense of Definiton~\ref{Defsol} with initial datum $E^{0,\lambda}$ for all but countably many $\lambda$'s, thus showing that $u$ is a level set solution in the sense of Definition~\ref{deflevelset1}, with initial datum $u^0$.
\end{remark}

\section{The case of general mobilities: Existence and uniqueness by approximation}\label{sec:genmob}
In this section we prove {one of the main results of this paper:} namely the existence via approximation by $\p$-regular mobilities of a  unique solution to the level set crystalline flow with a general mobility. As a byproduct of the proof we will also obtain uniqueness, up to fattening, of the flat flow, i.e. of the flow obtained by the ATW scheme. The main results are stated and proven in Subsection~\ref{subsec:final}. In the next subsection we collect some preliminary stability estimates on the ATW scheme.

\subsection{Stability of the ATW scheme with respect to  changing  mobilities}\label{subsec:stability}

We start with the following remark:
\begin{remark}\label{regE} \textup{For any norm $\eta$ and any closed
set $E\subset\R^N$, {it is easily seen that 
\[
d^\eta_{(E)_r^\eta} \le d^\eta_E - r
\]
where we have used the notation in~\eqref{nota} (with $\po$  replaced by $\eta$).}
% Indeed, if $d^\eta_E\ge 0$
%it just stems from the fact that $\eto(\nabla d^\eta_{(E)_r^\eta})=1$~a.e.~and $d^\eta_{(E)_{r}^\eta}\le -r$ on $\partial E$, while if $d^\eta_E(x)<0$, introducing
%$y\in\partial (E)_r^\eta$ such that $d^\eta_{(E)_r^\eta}(x)=-\eta(x-y)$ and $y'\in [x,y]\cap\partial E$, we observe that $\eta(x-y')\ge -d_E(x)$ and $\eta(y'-y)\ge r$ so that $\eta
%(x-y)=\eta(x-y')+\eta(y'-y)\ge -d_E(x)+r$.
}
\end{remark}

We recall (see Remark~\ref{rm:notation}) that  given a closed set $H$, $H_{g,h}^{\psi,k}$ denotes {a} $k$-th minimizing movement starting from  $H$,  with mobility $\psi$, forcing term $g$, and time step $h$ (the given and anisotropy $\p$).  As already observed,   in the previous notation the dependence on  the anisotropy  $\p$ is  omitted since we think of $\p$ as fixed.
We finally  recall that by an admissible forcing term $g$ we mean a function   satisfying assumptions H1), H2) of 
Subsection~\ref{stass}.

The next lemma establiishes a comparison result for minimizing movements with different 
forcing terms.
\begin{lemma}\label{discong}  
Let $\psi$, $\eta$ be two norms such that    $\psi \le\beta\eta $ 
 for some  $\beta>0$, and  
let $g_1$, $g_2$ be admissible forcing terms  satisfying 
$$
g_2-g_1\le  c<+\infty \textup{ in } \R^N\times[0,+\infty)\,.
$$
 If $E\subset F$ are closed sets with $\dist^{\eto}(E,\R^N\setminus F):=\Delta >0$,  then,  for all $k\in \N$ we have

\begin{equation}\label{eq:stimg1g2}
{
\dist^{\eto}(E_{g_1,h}^{\psi,k},\R^N\setminus F_{g_2,h}^{\psi,k}) \ge
\left(\Delta+\frac{c}{L_{\eto}}\right)(1-\beta L_{\eto}h)^k - \frac{c}{L_{\eto}},
}
\end{equation}
{where $L_{\eto}$ is the Lipschitz constant of $g_1$ and $g_2$ with respect to $\eto$.}
\end{lemma}

\begin{proof}
We start by considering  the case $k=1$. Set $\widehat c:=c+L_{\eto}\Delta$.
Let $(u,z)$ solve
\[
-h  \, \Div z + u = d^{\pso}_E + \int_0^h g_1(\cdot,s) \, ds, \quad z\in\partial\p(\nabla u).
\]
Let $\tau\in\R^N$, with $\eto(\tau)\leq \Delta$.
By our assumptions on $g_1,g_2$, one has for all $s$, $g_1(\cdot-\tau,s)
\ge g_2(\cdot,s)-c-L_{\eto}\eto(\tau)\ge g_2(\cdot,s)-\widehat c$, hence
\begin{align*}
-h \, \Div z(\cdot - \tau) + u(\cdot - \tau) &= d^{\pso}_{E +\tau} + \int_0^h g_1(\cdot - \tau ,s) \, ds
\\
&\ge d^{\pso}_{E+ \tau} +\int_0^h g_2(\cdot ,s)\, ds -  \widehat c h 
\\
&\ge d^{\pso}_{(E)^{\pso}_{\widehat c h}  + \tau} +\int_0^h g_2(\cdot ,s)\, ds,  
\end{align*}
where the last inequality is Remark \ref{regE}.
Thus by  comparison, and using 
%$W^{\psi}(0,1)\subseteq W^{\eta}(0, \beta)$,
{$\psi\le \beta \eta$, }
we deduce that
$$
E^{\psi,1}_{g_1,h} + \tau \subseteq  \bigl( (E)^{\pso}_{\widehat c h}  + \tau \bigr)^{\psi,1}_{g_2,h} \subseteq  \bigl( (E)^{\eto}_{\beta \widehat c h}  + \tau \bigr)^{\psi,1}_{g_2,h}\subseteq  
  F ^{\psi,1}_{g_2,h},
$$
provided that $(E)^{\eto}_{\beta \widehat c h}  + \tau \subseteq F$. 
The latter condition  holds true if  
$\eto(\tau) + \beta \widehat c h \le \Delta$. We deduce that
{
\[
\dist^{\eto} \left( E^{\psi,1}_{g_1,h} , \, \R^N \setminus F^{\psi,1}_{g_2,h} \right) \geq 
\Delta-\beta\widehat ch = \Delta(1-\beta L_{\eto}h) - \beta ch = 
{\Big(\Delta + \frac{c}{L_{\eto}}\Big)}(1-\beta L_{\eto}h) - \frac{c}{L_{\eto}}.
\]
}
The conclusion easily follows by induction.
\end{proof}

In the next lemma we compare the (discrete-time) solutions corresponding to different but close mobilities and  forcing terms.  
\begin{lemma}\label{lem:comparmobilities}
For any $\beta$,  $G$, $\Delta>0$, and $\theta\in (0,1)$, there exist positive   $\delta_ 0$, $h_0$, depending on  all the previous constants,  on the dimension $N$ and on the anisotropy $\p$, and there exists  $c_0>0$ depending  on the same quantities but $\Delta$, with the following property: 
Let $g$ be an admissible forcing term, with $\|g\|_\infty\leq G$,  and let $\psi_1$,   $\psi_2$ be two mobilities satisfying    
 \begin{equation}\label{eq:compani}
\psi_i  \le\ovc\p  \qquad \text{ for }  \,  i\in\{1,\, 2\}
 \end{equation} 
 and 
 \begin{equation}\label{eq:psivicino}
 \psi_2 \le \psi_1 \le (1+\delta)\psi_2  
 \end{equation}
 for some $0<\delta\leq \delta_0$. 
If  $E$ and $F$ are two closed sets with $\dist^{\po}(E,\R^N\setminus F)\ge\Delta$, then,   
setting  $\widetilde g:= g- c_0 \frac{\delta}{\Delta}$, for all $0<h\le h_0$  we have
\beq\label{whichis}
{
\dist^{\po}(E_{g,h}^{\psi_1,k},\R^N\setminus F_{\widetilde g,h}^{\psi_2,k}) 
 \ge \Delta \big(1- \beta L_{\po} h\big)^k  
}
\eeq
for all $k\in \N$ such that the right-hand side of the above inequality is larger than $\theta\Delta$. 
Here $L_{\po}$ denotes the Lipschitz constant of $g$ with respect to the metric $\po$.
\end{lemma}

\begin{proof} 
% We divide the proof into two steps. 
%\noindent{\it Step 1.} %We assume first that $E$ is bounded.  
With the notation introduced in \eqref{nota}, set 
$H:=\bigl((E)_{\Delta}^{\po}\bigr)_{-\Delta}^{\po}$ and note that $E\subseteq H$ and $\dist^{\po}(H, \R^N\setminus F)\geq \Delta$.
Also, it is easy to see that for  $\theta\in (0,1)$ the set $\R^N\setminus H$ can be written as a union of closed $\p$-Wulff shapes
 of radius $\theta \Delta=:\Delta_0$. Thus, by Lemma~\ref{lm:crucial} (and recalling \eqref{eq:compani}) there exists $M_0$, depending on 
 $G$, $\beta$, $\p$, and the dimension $N$, and there exists $h_0$ depending on the same quantities {and} %, but also 
on $\Delta_0$, such that 
 \beq\label{stella}
 H_{\widehat g, h}^{\psi_i, 1}\subset (H)^{\po}_{\frac{M_0h}{  \Delta_0}}\quad
 \text{for $i=1,2$, for $0<h\leq h_0$ and for any admissible $\widehat g$ s.t. $\|\widehat g\|_\infty\leq   G+1$.}
 \eeq
By \eqref{eq:psivicino} it follows
that 
\[
\pso_1\le \pso_2 \le (1+\delta)\pso_1 
\]
and, in turn, one has
\begin{equation}\label{eq:comparsdist}
\begin{cases}
d^{\pso_1}_H (x)\le d^{\pso_2}_H(x)\le (1+\delta)d^{\pso_1}_H (x)& \text{if }x\not\in\mathring{H},\\[2mm]
(1+\delta)d^{\pso_1}_H (x)\le d^{\pso_2}_H(x)\le d^{\pso_1}_H (x)& \text{if }x\in\overline{H}\\
\end{cases}
\end{equation}
so that 
\[
d^{\pso_2}_H-d^{\pso_1}_H\le \delta \bigl(d^{\pso_1}_H \bigr)^+\,.
\]
In particular,
\beq
\label{eq:comparsdist2}
d^{\pso_2}_H\le d^{\pso_1}_H+h\frac{\beta M_0}{\theta } \frac{\delta  }{\Delta}\quad\text{in $(H)^{\po}_{\frac{M_0h}{\Delta_0}}$.} 
\eeq

Set $\delta_0:=\frac{\theta \Delta_0}{2\beta M_0}$, $c_0:=  \frac{\beta M_0}{\theta }$ and note that 
{
\beq\label{deltazero}
c_0 \frac{\delta  }{\Delta}\leq \frac{\theta}{2} \le \frac12 \qquad\text{for $0<\delta\leq \delta_0$.}
\eeq
}

\noop{QUESTO: SI!(corretto) Thus, setting $ \widetilde g:=g-c_0 \frac{\delta  }{\Delta}$, by
\eqref{stella} we have
$$
H_{g,h}^{\psi_1, 1},\, H_{\widetilde g,h}^{\psi_2, 1} \subset
(H)^{\po}_{\frac{M_0h}{\Delta_0}}
$$
provided that $0<\delta\leq \delta_0$, $0<h\leq h_0$
(recall  $\|g\|_\infty\leq G \leq G+\frac12$).
Thus, we may apply \eqref{eq:comparsdist2} and, being $H$  bounded, we
may apply Lemma~\ref{lm:localcomp} to deduce that}

Thus, setting $ \widetilde g:=g-c_0 \frac{\delta  }{\Delta}$, by \eqref{stella} we have
{
$$
H_{{g},h}^{\psi_1, 1},\, H_{\widetilde g,h}^{\psi_2, 1} \subset (H)^{\po}_{\frac{M_0h}{\Delta_0}}
$$
}
provided that $0<\delta\leq \delta_0$, $0<h\leq h_0$ {(recall  $\|g\|_\infty\leq G \leq G+\frac12$).}
{Thus, we may apply \eqref{eq:comparsdist2} and
%, being $H$  bounded, we may apply 
Lemma~\ref{lm:localcomp} to deduce that}
$$
E_{g,h}^{\psi_1, 1}\subseteq H_{g,h}^{\psi_1, 1}\subseteq H_{\widetilde g,h}^{\psi_2, 1}\,.
$$
In turn, by Lemma~\ref{discong} (with $g_1=g_2=\widetilde g$, $c=0$, $\psi=\psi_2$, and $\eta=\p$) we get
$$
\dist^{\po}(E_{g,h}^{\psi_1,1},\R^N\setminus F_{\widetilde g,h}^{\psi_2,1}) \geq 
\dist^{\po}(H_{\widetilde g,h}^{\psi_2,1},\R^N\setminus F_{\widetilde g,h}^{\psi_2,1})
 \ge \Delta(1-\beta L_{\po}h)\,,
$$
which is \eqref{whichis} for $k=1$.

{We can iterate this construction as long as this distance
is larger than $\Delta_0$, deducing that~\eqref{whichis} holds
as long as $(1-\beta L_{\po}h)^k\geq \theta$.}
\end{proof}
Combining Lemmas~\ref{discong} and~\ref{lem:comparmobilities}, we
{obtain} %can prove
 the following proposition.

\begin{proposition}\label{lem:comparmobilities2}
For any $\beta$,  $G$, $\Delta>0$, and $\theta\in (0,1)$, there exist positive   $\delta_ 0$, $h_0$, depending on  all the previous constants,  on the dimension $N$ and on the anisotropy $\p$, and there exists  $c_0>0$ depending  on the same quantities but $\Delta$, with the following property: 
Let $g$ be an admissible forcing term, with $\|g\|_\infty\leq G$,  and let $\psi_1$,   $\psi_2$ be two mobilities satisfying 
\eqref{eq:compani}  and \eqref{eq:psivicino} for some $0<\delta\leq \delta_0$. 
If  $E$ and $F$ are two closed sets with $\dist^{\po}(E,\R^N\setminus F)\ge\Delta$, then
   for all $0<h\le h_0$  we have
\beq\label{whichis2}
\dist^{\po}(E_{g,h}^{\psi_1,k},\R^N\setminus F_{ g,h}^{\psi_2,k}) 
 \ge \Bigl(\Delta+\frac{2c_0\de}{L_{\po}\Delta}\Bigl)(1-\beta L_{\po}h)^k- \frac{2c_0\de}{L_{\po}\Delta}
\eeq
for all $k\in \N$ such that $(1-\beta L_{\po}h)^k\geq \theta$.  
Here $L_{\po}$ denotes the Lipschitz constant of $g$ with respect to the metric $\po$.
\end{proposition}

\begin{proof}

In what follows, $\de_0$, $h_0$, and $c_0$ are the constant provided by Lemma~\ref{lem:comparmobilities}.  With the notation introduced in $\eqref{nota}$, let $H:= (E)_{\frac{\Delta}{2}}^{\po}$, so that   
$$
\dist^{\po}(E,\R^N\setminus H)\ge\frac{\Delta}{2}, \qquad \dist^{\po}(H,\R^N\setminus F)\ge \frac{\Delta}{2}.
$$
Set $\widetilde g:= g- 2 c_0 \frac{\delta}{\Delta}$.
By \eqref{whichis}  (with    $F_{\widetilde g,h}^{\psi_2,k}$ and $\Delta$ replaced 
{with} %by
 $H_{\widetilde g,h}^{\psi_2,k}$ and $\Delta/2$, respectively)  we have 
\begin{equation}\label{dislemma1}
\dist^{\po}(E_{g,h}^{\psi_1,k},\R^N\setminus H_{\widetilde g,h}^{\psi_2,k}) 
 \ge \frac{\Delta}{2}(1-\beta L_{\po}h)^k
\end{equation}
for all $0<h\le h_0$  and
for all $k\in \N$ such that $(1-\beta L_{\po}h)^k\geq \theta$.

Moreover, by Lemma \ref{discong} (with $\eta=\p$, $g_1:=\widetilde g$, $g_2:=g$, $c:=2c_0\frac{\de}{\Delta}$,  and $E$ replaced by $H$)       
we have
\begin{equation}\label{dislemma2}
\dist^{\po}(H_{\widetilde g,h}^{\psi_2,k},\R^N\setminus F_{g,h}^{\psi_2,k}) \ge 
 \Bigl(\frac{\Delta}2+\frac{2c_0\de}{L_{\po}\Delta}\Bigl)(1-\beta L_{\po}h)^k- \frac{2c_0\de}{L_{\po}\Delta}
\end{equation}
for all $ k\in \N$.
Since

$$
\dist^{\po}(E_{g,h}^{\psi_1,k},\R^N\setminus F_{ g,h}^{\psi_2,k}) \ge \dist^{\po}(E_{g,h}^{\psi_1,k},\R^N\setminus H_{\widetilde g,h}^{\psi_2,k})  + 
\dist^{\po}(H_{\widetilde g,h}^{\psi_2,k},\R^N\setminus F_{g,h}^{\psi_2,k}) ,
$$
the conclusion follows directly by \eqref{dislemma1} and \eqref{dislemma2}.
\end{proof}
\begin{remark}[Varying anisotropies]\label{rm:proponphi}
A careful inspection of the proof of Proposition~\ref{lem:comparmobilities2} (and of Lemma~\ref{lem:comparmobilities}) together with Remark~\ref{rm:mzerophi} shows that the constants $\de_0$, $h_0$, $c_0$ can be chosen
as depending on $\p$ only through the ellipticity constants $a_1$, $a_2$ in \eqref{phia12}. This observation implies that  estimate \eqref{whichis2} holds uniformly with respect to converging sequences of anisotropies. 

  More precisely, let $\{\p_n\}$ be a  sequence of anisotropies such that $\p_n\to \p$ and let us denote, temporarily, by $E_{g,h}^{\psi, \p_n,k}$ the $k$-th minimizing movement starting from $E$, with mobility $\psi$,   forcing term $g$, time-step $h$, {\em and anisotropy $\p_n$}.
 Set $a'_1=a_1/2$, $a'_2=2a_2$,  $\beta'=2\beta$ and observe that for $n$ large $\p_n$ satisfies 
\eqref{phia12} with $a'_i$ in place of $a_i$. Moreover, if $\psi$ satisfies \eqref{eq:compani}, then it also satisfies  \eqref{eq:compani}
with $\p$, $\beta$ replaced by $\p_n$, $\beta'$, respectively, provided that $n$ is large enough. Therefore, we may find  $\delta_ 0$, $h_0$, depending on $\beta'$,  $G$, $\Delta$, $a'_1$, $a'_2$ (and the dimension $N$), and $c_0$ depending on all the same quantities but  $\Delta$, such that under the assumptions of Proposition~\ref{lem:comparmobilities2} we have for all $n$ sufficiently large
$$
\dist^{\po}(E_{g,h}^{\psi_1, \p_n,k},\R^N\setminus F_{ g,h}^{\psi_2, \p_n,k}) 
 \ge \Bigl(\Delta+\frac{2c_0\de}{L_{\po}\Delta}\Bigl)(1-\beta L_{\po}h)^k- \frac{2c_0\de}{L_{\po}\Delta}
$$
for all $k\in \N$ such that $(1-\beta L_{\po}h)^k\geq \theta$.  
 
 \end{remark}

\subsection{ Existence and uniqueness by approximation}\label{subsec:final}

In the following, given  a  uniformly continuous function $u^0$ on $\R^N$,  an   admissible forcing term $g$  and  a  mobility    $\psi$,
we denote by $u_h^\psi$  the corresponding level set minimizing movement, defined according to 
\eqref{defls}. Analogously, we use the notation $E_{\lambda, h}^\psi(t)$ (in place of $E_{\lambda, h}(t)$)  to denote the discrete-in-time evolution starting from $E^{0,\lambda}:=\{u^0\leq \lambda\}$ and with mobility $\psi$ (see Subsection~\ref{subsec:levelset}).
In the above notation we have highlighted only the dependence on $\psi$ since in the following we will establish stability properties of flat flows with respect to varying  mobilities.

We recall that  the existence theory  for level set flows (in the sense of Definition~\ref{deflevelset1}) that we have so far works only for $\p$-regular mobilities. The goal of  this section is to extend the existence theory to general mobilities. To this aim, we consider the following notion of {\em solution via approximation}: 
\begin{definition}[Level set flows via approximation]\label{deflevelset2}
 Let $\psi$, $g$, and $u^0$ be a  mobility,  an admissible forcing term, and a uniformly continuous function on $\R^N$, respectively. 
 
 We will say that a  continuous function $u^\psi:\R^N\times [0, +\infty)\to \R$ is a {\em solution via approximation} to the level set flow corresponding to \eqref{oee}, with initial datum $u^0$,  if $u^\psi(\cdot, 0)=u^0$ and if there exists  a sequence $\{\psi_n\}$ of $\p$-regular mobilities  such that $\psi_n\to\psi$  and, denoting by $u^{\psi_n}$ the unique solution to \eqref{oee} (in the sense od Definition~\ref{deflevelset1}) with mobility $\psi_n$ and initial datum $u^0$, we have $u^{\psi_n}\to u^\psi$  locally  uniformly in $\R^N\times[0,+\infty )$.
 \end{definition}
 The next theorem is the main result of this section: it shows that for any mobility $\psi$ a solution-via-approximation $u^\psi$ in the sense of the previous definition always exists; such a solution  is also unique in that it is   independent of the choice of the approximating sequence of $\p$-regular mobilities $\{\psi_n\}$ and, in fact,  coincides with the (unique) limit of the level set minimizing movement scheme $\{u^\psi_h\}$. In particular, in the case of a $\p$-regular mobility the notion of solution via approximation is consistent with that of Definition~\ref{deflevelset1}. 
\begin{theorem}\label{th:maingenmob}
Let $\psi$, $g$, and $u^0$ be as in Definition~\ref{deflevelset2}. Then, there exists a unique solution $u^\psi$ in the sense of Definition~\ref{deflevelset2} with initial datum $u^0$. Moreover,  the following holds:
\begin{itemize}
\item [(i)]  {\rm (Convergence of the level set minimizing movements scheme)}
 The solution $u^\psi$ is the locally uniform limit in $\R^N\times [0,+\infty)$, as $h\to 0^+$,  of the    level set minimizing movements  $u^\psi_h$.  
\item[(ii)] {\rm (Stability) }
Let  $\{\psi_n\}_{n\in\N}$ be a sequence of mobilities such that   $\psi_n\to  \psi$.  Then $u^{\psi_n}$ converge to $u^{\psi}$   uniformly  in $\R^N\times[0,T]$  for all $T>0$ as $n\to\infty$. 
 \end{itemize}
\end{theorem}
\begin{proof}
The strategy is the following: We first show that for any $\psi$ the minimizing movements $u^\psi_h$ converge to a unique function $u^\psi$, as $h\to 0^+$. Then we establish the stability property (ii), which shows, in particular, that $u^\psi$ is a solution in the sense of Definition~\ref{deflevelset2}. We split the proof of  theorem into three steps.\\
{\it Step 1.} We claim that for every $\e$, $\beta$, $T>0$ there exist positive $\delta_0$, $h_0>0$ 
(depending also on $g$, $u^0$, $\p$, and the dimension $N$) such that if $\psi_1$ and $\psi_2$ are two mobilities satisfying \eqref{eq:compani}  and \eqref{eq:psivicino} for some $0<\delta\leq \delta_0$, then 
\beq\label{step1eq}
\|u_h^{\psi_2}-u_h^{\psi_1}\|_{L^{\infty}(\R^N\times[0,T])}\leq \e \qquad\text{for all $0<h\leq h_0$.}
\eeq
To this aim, let $\omega$ be an increasing modulus of continuity for  $u^0$ with respect to $\po$ and recall that for any $\lambda\in \R$
$$
\dist^{\po}(E^{0,\lambda}, \R^N\setminus E^{0,\lambda+\e})\geq \omega^{-1}(\e)\,.
$$
Set 
$$
\theta(T):=\mathrm{(2e)}^{-\beta L_{\po} T}\,,
$$
where $L_{\po}$ denotes the spatial Lipschitz constant of the forcing term $g$ with respect to $\po$, and choose $\bar h>0$ so small that 
$(1-\beta L_{\po} h)^{-\frac1{\beta L_{\po} h}}\leq 2 \mathrm{e}$ for all $h\in (0, \bar h)$. Let $\de_0$, $h_0$, $c_0$ be the positive constants provided by Proposition~\ref{lem:comparmobilities2} and corresponding to the given $\beta$, $G:=\|g\|_\infty$, $\Delta:=\omega^{-1}(\e)$,
and $\theta:= \theta(T)$. Clearly we may assume $h_0\leq \bar h$.
 
 By Proposition~\ref{lem:comparmobilities2},  if $\psi_1$ and $\psi_2$ satisfy \eqref{eq:compani}  and \eqref{eq:psivicino} for some $0<\delta\leq \delta_0$, then  for $h\in (0, h_0]$ we   have 
\begin{align}
\dist^{\po}\bigr(E_{\lambda, h}^{\psi_1} , \R^N\setminus E_{\lambda+\e, h}^{ \psi_2}(t)\bigl) & \geq 
\biggl(\omega^{-1}(\e)+\frac{2c_0\de}{L_{\po}\omega^{-1}(\e)}\biggl)(1-\beta L_{\po}h)^{[\frac{t}h]}- \frac{2c_0\de}{L_{\po}\omega^{-1}(\e)}
  \nonumber \\
&\geq \biggl(\omega^{-1}(\e)+\frac{2c_0\de}{L_{\po}\omega^{-1}(\e)}\biggl)(2\mathrm{e})^{-\beta L_{\po}t}- \frac{2c_0\de}{L_{\po}\omega^{-1}(\e)}, \label{cauchy0}
\end{align}
for all $t\in (0, T]$.  Clearly, by choosing  $\delta_0$ smaller if needed,  we may assume that  right-hand side of \eqref{cauchy0} is positive for all $t\in (0, T]$. Recalling \eqref{eq:bc}, we conclude that
$$
\{u^{\psi_1}_h(\cdot, t)<\lambda\}\subseteq E_{\lambda,h}^{\psi_1}(t)\subseteq E_{\lambda+\e,h}^{ \psi_2}(t)\subseteq\{u^{ \psi_2}_h(\cdot, t)\leq\lambda+\e\}
$$ 
for all  $\lambda\in \R$, $h\in (0, h_0]$, and $t\in [0, T]$. This in turn implies that 
$$
u_h^{ \psi_2}(\cdot, t)\leq u_h^{\psi_1}(\cdot, t)+\e \qquad\text{for all $h\in (0, h_0]$ and $t\in [0, T]$.}
$$ 
We may now repeat the same argument by considering $-u^0$ as initial function, instead of $u^0$. This leads to the inequality
$$
-u_h^{ \psi_2}(\cdot, t)\leq- u_h^{\psi_1}(\cdot, t)+\e \qquad\text{for all $h\in (0, h_0]$ and $t\in [0, T]$,}
$$ 
which together with the previous one proves \eqref{step1eq}.

\noindent {\it Step 2.}  Here  we prove that  $\{u^\psi_h\}_h$ satisfies the Cauchy condition in $L^\infty(K\times[0,T])$ for all compact sets $K\subset\R^N$ and for all $T>0$. 

To this {purpose, let } $T$, $\e>0$,  let $\beta>0$ satisfy $\psi\leq \frac{\beta}2\p$,  and  let  $\delta_0$, $h_0$ be the corresponding constants provided by Step 1. 
Clearly we may choose a $\p$-regular mobility (see Definition~\ref{def:phiregular}) $\widehat \psi$ such that $\psi_1:=\psi$ and $\psi_2:=\widehat \psi$ satisfy \eqref{eq:compani}  and \eqref{eq:psivicino}  for some $0<\delta\leq \delta_0$.
Pick any sequence $h_n\searrow  0$. Then,  we may write
\begin{multline*}
\|u^{\psi}_{h_n}-u^{\psi}_{h_m}\|_{L^{\infty}(K\times [0,T])}\leq \|u^{\psi}_{h_n}-u^{\widehat \psi}_{h_n}\|_{L^{\infty}(K\times [0,T])}\\+\|u^{\widehat\psi}_{h_n}-u^{\widehat\psi}_{h_m}\|_{L^{\infty}(K\times [0,T])}+\|u^{\widehat \psi}_{h_m}-u^{\psi}_{h_m}\|_{L^{\infty}(K\times [0,T])}\,.
\end{multline*}
The first and the third term on the right-hand side  of the above inequality are both less than or equal to $\e$ thanks to \eqref{step1eq}, provided that $h_n$, $h_m\leq h_0$.  Recall now that by 
Theorem~\ref{th:phiregularlevelset}-(ii) the family $\bigl\{u^{\widehat\psi}_h\bigr\}_h$ satisfies the Cauchy condition in $L^\infty(K\times[0,T])$; thus also the middle term on the right-hand side  of the above inequality is smaller than 
$\e$ for $n$ and $m$ large enough. We conclude that $\|u^{\psi}_{h_n}-u^{\psi}_{h_m}\|_{L^{\infty}(K\times [0,T])}\leq 3\e$ for $n$, $m$ sufficiently large. This establishes the claim and shows that $u^\psi_h$ converges locally uniformly in $\R^N\times [0, +\infty)$. We denote its limit by $u^\psi$. 

\noindent {\it Step 3.} Let $\{\psi_n\}_n$ be a sequence of mobilities such that $\psi_n\to \psi$.
First of all, observe that we may find $\lambda_n\to  1^-$such that $\widehat\psi_n:=\lambda_n\psi_n\leq\psi$ for all $n$.

Fix $\e>0$, let $\beta>0$ be as in Step 2, and  let $\delta_0$, $h_0$ be the corresponding constants provided by Step 1.  Note that for any $0<\de\leq \de_0$  we have
\beq\label{txstep1}
\widehat\psi_n\leq \psi\leq (1+\delta)\widehat \psi_n\,, \qquad \widehat\psi_n\leq \psi_n\leq (1+\delta)\widehat \psi_n \quad\text{ and  }\quad 
\psi, \psi_n, \widehat\psi_n\leq \beta \p\,,
\eeq
provided $n$ large enough.
Thus, thanks to Step 1, for all such $n$'s and for all $h\leq h_0$ we have
$$
\|u^{\psi}_{h}-u^{\psi_n}_{h}\|_{L^{\infty}(\R^N\times [0,T])}\leq \|u^{\psi}_{h}-u^{\widehat \psi_n}_{h}\|_{L^{\infty}(\R^N\times [0,T])}+\|u^{\widehat \psi_n}_{h}-u^{\psi_n}_{h}\|_{L^{\infty}(\R^N\times [0,T])}\leq 2\e\,.
$$
Thanks to Step 2 we may send $h\to 0$ in the above inequality to infer 
that $\|u^{\psi}-u^{\psi_n}\|_{L^{\infty}(\R^N\times [0,T])}\leq 2\e$ for all $n$ sufficiently large. This concludes the proof of the theorem.
\end{proof}

In the next theorem   we collect the main properties of the level set solutions introduced in Definition~\ref{deflevelset2}.

To this aim,  we will say that a uniformly continuous initial function $u^0$ is {\em well-prepared} at $\lambda\in\R$ if the following two conditions hold:
\begin{itemize}
\item[(a)] If $H\subset \R^N$ is a closed set such that $\dist(H, \{u_0\geq \lambda\})>0$, then there exists $\lambda'<\lambda$ such that $H\subseteq
\{u_0< \lambda'\}$;
\item[(b)] If $A\subset \R^N$ is an open  set such that  $\dist(\{u_0\leq \lambda\}, \R^N\setminus A)>0$,  then there exists $\lambda'>\lambda$ such that 
$\{u_0\leq \lambda'\}\subset A$.
\end{itemize}
\begin{remark}
Note that the above assumption of well-preparedness is automatically satisfied if the set $\{u_0\leq \lambda\}$ is bounded.
\end{remark}
\begin{theorem}[Properties of the level set flow]\label{th:proplevelset}
Let  $u^\psi$ be  a solution in the sense  Definition~\ref{deflevelset2}, with initial datum $u^0$. The following properties hold true:

  {\rm (i) (Non-fattening level sets and unique flat flows)}   There exists a countable set $N\subset \R$ such that  for all  $\lambda \not \in N$
\beq\label{eq:nonfattening2}
\begin{array}{rcl}  
\{(x,t):\, u^\psi(x, t) < \lambda\} &=& \mathrm{Int\,}(\{(x,t)\,:\, u^\psi(x, t) \le \lambda\})\,,\vspace{5pt}\\ 
\overline {\{(x,t)\,:\, u^\psi(x, t) < \lambda\}} &=& \{(x,t)\,:\, u^\psi(x, t) \le \lambda\}
\end{array}
\eeq
and 
the flat flow starting from $E^{0,\lambda}$ is unique. More precisley,  we have
$$
\qquad E^\psi_{\lambda, h}\stackrel{\mathcal K}{\longrightarrow} \{(x,t)\,:\, u^\psi(x, t) \le \lambda\}\text{ and }
{ (\mathrm{Int\, }E^\psi_{\lambda, h})}^c\stackrel{\mathcal K}{\longrightarrow} \{(x,t)\,:\, u^\psi(x, t) \geq \lambda\}
$$
as $h\to 0^+$.

 {\rm (ii) (Distributional formulation when $\psi$ is $\p$-regular)} If $\psi$ is $\p$-regular, then $u^\psi$ coincides with the  distributional solution in the sense of Definition~\ref{deflevelset1}.

 {\rm (iii) (Comparison)} Assume that $u^0\leq v^0$ and denote the corresponding level set flows by $u^\psi$ and $v^\psi$, respectively. Then $u^\psi\leq v^\psi$.

 {\rm (iv) (Geometricity)} Let $f:\R\to \R$ be increasing and continuous. Then $u^\psi$ is a solution with initial datum $u^0$ if and only if $f\circ u^{\psi}$ is a solution with initial datum $f\circ u^0$.

 {\rm (v) (Independence of the initial level set function)} Assume that $u^0$ and $v^0$ are well-prepared at $\lambda$. If
$\{u^0<\lambda\}=\{v^0<\lambda\}$, then $\{u^\psi(\cdot, t)<\lambda\}=\{v^\psi(\cdot, t)<\lambda\}$ for all $t>0$. Analogously, if 
$\{u^0\leq \lambda\}=\{v^0\leq \lambda\}$, then  $\{u^\psi(\cdot, t)\leq \lambda\}=\{v^\psi(\cdot, t)\leq \lambda\}$ for all $t>0$.
\end{theorem}

\begin{proof}
Property (ii) is obvious. Properties (i)   can be proven arguing as in the proof of Theorem~\ref{th:phiregularlevelset}. Property (iii) follows at once from  { the stability property of flat flows with respect to approximation with smooth mobilities, and from } Theorem~\ref{th:lscomp}. Also property (iv) follows by approximation, since clearly  it is satisfied 
when the mobility $\psi$ is $\p$-regular. 
Let us now prove property (v): Pick $\lambda_1<\lambda$ and note that by uniform continuity and by assumption we have
$$
\dist(\{u^0\leq \lambda_1\},\{v^0\geq \lambda\} )=\dist(\{u^0\leq \lambda_1\},\{u^0\geq \lambda\} )>0
$$
and thus there exists $\lambda_2\in (\lambda_1, \lambda)$ such that $\{u^0\leq \lambda_1\}\subset \{v^0< \lambda_2\}$.
Let now $\{\psi_n\}$ be an approximating sequence of $\p$-regular mobilities and recall that by 
Theorem~\ref{th:phiregularlevelset}-(iii) the set $\{(x,t)\,:\, u^{\psi_n}(x,t)\leq \lambda_1\}$ is a superflow with initial datum 
$\{u^0\leq \lambda_1\}$, while $\{(x,t)\,:\, v^{\psi_n}(x,t)< \lambda_2\}$ is a subflow with initial datum $\{v^0< \lambda_2\}$.
Thus, from Theorem~\ref{th:compar} we deduce that 
$\{(x,t)\,:\, u^{\psi_n}(x,t)< \lambda_1\}\subset  \{(x,t)\,:\, v^{\psi_n}(x,t)< \lambda_2\}$  for all $n$.
In turn, from the latter inclusion  we easily deduce  that 
$$
\{(x,t)\,:\, u^{\psi}(x,t)< \lambda_1\}\subset  \{(x,t)\,:\, v^{\psi}(x,t) \leq \lambda_2\}\subset \{(x,t)\,:\, v^{\psi}(x,t) < \lambda\}\,.
$$
By the arbitrariness of $\lambda_1$, we conclude that $\{(x,t)\,:\, u^{\psi}(x,t)< \lambda\}\subseteq \{(x,t)\,:\, v^{\psi} (x,t)< \lambda\}$.
Symmetrically, also the opposite inclusion holds.  The equality between the closed sub-level sets can be proven analogously. 
\end{proof}

\begin{remark}[Generalized motion] We observe that property (v) above allows one to consider $\Gamma_t:=\{u^\psi(\cdot, t)=0\}$ as defining a  {\em generalized motion} starting from $\Gamma_0:=\{u^0=0\}$. 
\end{remark}

\begin{remark}[Star-shaped sets, convex sets and graphs] A natural question is
to understand under which circumstances fattening does not occur. To the best of
our knowledge,  no general results are available, even for the classical mean
curvature flow. On the other hand, it is well-known  \cite[Sec.~9]{Soner93} that for the motion without forcing,
strictly star-shaped sets do not develop fattening so that, in particular, their evolution is unique. % as long as they remain strictly star-shaped.
The proof of this fact, given for instance
in~\cite{Soner93} for the mean curvature flow, works also for solutions in the sense of Definition~\ref{Defsol} when the mobility  $\psi$ is $\p$-regular, and in turn, by approximation, also for the {\em generalized motion} associated to level set solutions in the sense of Definition~\ref{deflevelset2}, when $\psi$ is general. Uniqueness also holds for motions with a time-dependent forcing
$g(t)$ \cite[Theorem~5]{BeCaChNo-volpres} as long as the set remains strictly
star-shaped.
This remark obviously applies to initial convex sets, which, in
addition, remain convex for all times,
as was shown in~\cite{BelCaChaNo,CaCha,BeCaChNo-volpres} with a spatially constant forcing term.\footnote{Convexity is preserved also with a spatially convex forcing term but uniqueness is not known in this case.}
The case of unbounded initial
convex sets was not considered in these references but can be easily addressed by approximation (and uniqueness still holds with the same proof).

In the same way, if the initial set $E_0=\{x_N\le v^0(x_1,\dots,x_{N-1})\}$ is
 the subgraph of a uniformly continuous functions $v^0$, and the forcing term
does not depend on $x_N$, then
one can show that fattening does not develop and $E(t)$ is still the subgraph of a uniformly continuous function for all $t>0$, as in the classical case~\cite{EckerHuisken,EvansSpruckIII} (see also~\cite{GigaGiga98} for the 2D crystalline case).
%In particular
%the graph of an initial convex function evolves in a unique way and
%remains convex for all times.
\end{remark}

Eventually, we extend the stability property in Theorem~\ref{th:maingenmob}-(ii) to varying anisotropies. 
\begin{proposition}\label{remstab}
Let $\psi$, $g$, and $u^0$ be as in Theorem~\ref{th:maingenmob}, 
let  $\{\psi_n\}$ and $\{\p_n\}$ be a sequences of mobilities 
and anisotropies, respectively, such that
$\psi_n\to  \psi$ and $\p_n\to  \p$ as $n\to +\infty$. Denote by  $u^{\psi_n, \p_n}$ 
the level set solution  in the sense of 
Definition~\ref{deflevelset2} with $\psi$, $\p$ replaced by $\psi_n$, $\p_n$, respectively, and with initial datum $u^0$.  
Then, $u^{\psi_n, \p_n}\to u^{\psi, \p}$ locally uniformly  in $\R^N\times[0,+\infty)$ as $n\to\infty$.
\end{proposition}
\begin{proof} From Steps 1 and 3 in the proof of Theorem~\ref{th:maingenmob} combined with Remark~\ref{rm:proponphi},
we have that for every $\e$, $\beta$, $T>0$ there exist positive $\delta_0$, $h_0>0$ 
(depending also on $g$, $u^0$, $\p$, the dimension $N$, but not on $n$) such that if $\psi_1$ and $\psi_2$ are two mobilities satisfying \eqref{eq:compani}  and 
$\max_{|\xi|=1}|\psi_1(\xi)-\psi_2(\xi)|\leq \de_0$,
 then for all $n$ large enough
$\|u_h^{\psi_2, \p_n}-u_h^{\psi_1,\p_n}\|_{L^{\infty}(\R^N\times[0,T])}\leq \e$  for all $0<h\leq h_0$.
Sending $h\to 0$ we deduce 
$\|u^{\psi_2, \p_n}-u^{\psi_1,\p_n}\|_{L^{\infty}(\R^N\times[0,T])}\leq \e$
for $n$ large enough.  Thus, in particular, we may choose $\de>0$ so small that %denoting 
%by  $\psi_n^{\de}$ the norm whose Wulff shapes   $W^{\psi_n^{\de}}(0,1)$ is 
%    given by $\big(W^{\psi}(0,1)\big)^{\po_n}_{\de}$ (with the notation introduced in \eqref{nota}), 
{letting $\psi_n^{\de}:=\psi + \de\po_n$,}
then
\beq\label{czz}
\|u^{\psi_n, \p_n}-u^{\psi_n^{\de},\p_n}\|_{L^{\infty}(\R^N\times[0,T])}\leq \e\qquad\text{for $n$ large enough.} 
\eeq
Taking $\de$ smaller, if needed, thanks to Theorem~\ref{th:maingenmob}-(ii) we may also impose
\beq\label{czz2}
\|u^{\psi^\de, \p}-u^{\psi,\p}\|_{L^{\infty}(\R^N\times[0,T])}\leq \e\,,
\eeq
where $\psi^\de$ is defined as $\psi_n^\de$, with $\psi_n$ replaced by $\psi$. 
Note now  that by construction $\psi_n^\de$ is $\p_n$-regular with $W^{\psi^\de_n}(0,1)$ satisfying an inner $\p_n$-Wulff shape condition of radius $\de$ (and thus uniformly in $n$) and $\psi_n^\de\to\psi^\de$. By Remark~\ref{rm:stability2} we then have $u^{\psi_n^\de, \p_n}\to u^{\psi^\de, \p}$ locally uniformly. Thus, for any fixed compact set $K\subseteq\R^N$, recalling also \eqref{czz} and \eqref{czz2}, we have
$$
\limsup_{n}\|u^{\psi_n, \p_n}-u^{\psi, \p}\|_{L^{\infty}(K\times [0,T])}\leq 2\e\,.
$$
The conclusion follows by the arbitrariness of $\e$.
\end{proof}

\section{{Concluding remarks}}
We conclude the paper with the following remarks.
\begin{remark}[Comparison with the Giga-\Pozar\ solution]\label{rm:gigapozar}
When $N=3$, $\p$ is purely crystalline and $g\equiv c$, $c\in \R$, the unique level set solution in the sense of Definition~\ref{deflevelset2} coincides with the viscosity solution constructed in \cite{GigaPozar}.  

Let $\p_n$ be a sequence of smooth anisotropies such that $\{\p_n\leq 1\}$ is strictly convex for every $n$ and $\p_n\to \p$. 
 By Lemma~\ref{lem:visco} the unique viscosity level set solution $u_n$ corresponding to the motion 
$$
V=-\psi (\nu)(\kappa_{\p_n}+c)
$$
coincides with the level set solution in the sense of Definition~\ref{deflevelset1}. By Proposition~\ref{remstab}, $u_n\to u$ locally uniformly with $u$ the unique level set solution in the sense of 
Definition~\ref{deflevelset2} corresponding to 
$$
V=-\psi(\nu)(\kappa_{\p}+c)\,.
$$
But thanks to \cite[Theorem 8.9]{GigaPozar}, it turns out that $u$ is also  the viscosity solution in the sense of Giga-\Pozar.
{We expect that this
argument also holds in higher dimension, for the solutions
defined in~\cite{GigaPozar17}.}
\end{remark}

%We conclude the paper with the following observation on phase-field approximations of the crystalline flow in the case constant forcing terms. 

\begin{remark}[Approximation by anisotropic Allen-Cahn equations]\label{rm:allen}
In \cite{GOS} the authors consider the anisotropic Allen-Cahn equation
\begin{equation}\label{eqallen}
v_t = \psi(\nabla v) \left({\rm div}\big( \p(\nabla v)\nabla\p(\nabla v)\big) 
-\frac{1}{\e^2} W'(v)+\frac\lambda\e \,g\right),
\end{equation}
where $\psi,\,\p$ are respectively a smooth mobility and anisotropy, 
$g\equiv c$, $c\in \R$, is a constant forcing term,  
$W$ is a standard double-well potential with zeroes in $\pm 1$, and $\lambda$ is a
suitable constant depending only on $W$.  

Let now $u_0$ be a uniformly continuous function, let $u$ be the corresponding solution
to the level set flow given by Theorem \ref{th:phiregularlevelset}, and let $\gamma:\R\to \R$
be the (unique) solution to $-\gamma''+W'(\gamma)=0$ with $\gamma(0)=0$
and $\lim_{x\to\pm\infty}\gamma(x)=\pm 1$.

In \cite[Theorem 2.2]{GOS} it is shown  that the solutions $u^\e$ to \eqref{eqallen} with initial data
$$
u^\e_0(x) := \gamma\left(\frac 1\e d^{\p^0}_{\{u_0<0\}}(x)\right)
$$
converge as $\e\to 0$ to a family of characteristic functions $\chi_{E(t)}$
with $\partial E(t)\subset \{x:\,u(x,t)=0\}$ for all $t>0$, where $u$ is the level set solution corresponding to \eqref{oee}.
This means that the solutions to \eqref{eqallen} converge to a (generalized)
solution to \eqref{oee} which is contained in the zero-level set of $u$.

In \cite[Theorem 2.4]{GOS} the authors also show that, given two sequences $\psi_n,\,\p_n$ 
of smooth mobilities and anisotropies converging to (possibly nonsmooth) limit 
functions {$\psi,\,\p$, if} the corresponding level set solutions $u_n$,
with initial datum $u_0$, converge to a limit function $u$,
then the corresponding solutions $u^\e_n$ to \eqref{eqallen} converge
as $\e\to 0$ and $n\to \infty$ to a family of characteristic functions $\chi_{E(t)}$
with $\partial E(t)\subset \{x:\,u(x,t)=0\}$ for all $t>0$.

{Thanks to Proposition \ref{remstab}} we know that the solutions $u_n$ do indeed converge
to the unique solution $u$ given by Theorem  \ref{th:maingenmob},
so that the convergence result in \cite[Theorem 2.4]{GOS} applies to our solutions.

Notice also that the solutions $u^\e_n$ converge as $n\to \infty$
to the unique solution $u^\e$ of the (nonsmooth) Allen-Cahn inclusion corresponding to \eqref{eqallen}
(see \cite{BeNo} for a precise definition), so that the convergence result also applies to such 
solutions $u^\e$, thus significantly extending the convergence result in \cite{BeNo}.
\end{remark}

\section*{Acknowledgements}
A.~Chambolle was partially supported by the ANR, programs ANR-12-BS01-0014-01 ``GEOMETRYA'' and ANR-12-BS01-0008-01  ``HJnet''. M.~Novaga was partially supported by an invited professorship of the Ecole Polytechnique, Palaiseau. M~Morini and M.~Ponsiglione were partially supported by the GNAMPA grant  2016 ``Variational methods for nonlocal geometric flows''.

\smallskip

\bibliography{cpam-ccf}

\begin{thebibliography}{10}

\bibitem{AlChNo}
Luis Almeida, Antonin Chambolle, and Matteo Novaga.
\newblock Mean curvature flow with obstacles.
\newblock {\em Ann. Inst. H. Poincar\'e Anal. Non Lin\'eaire}, 29(5):667--681,
  2012.

\bibitem{AlmTay95}
Fred Almgren and Jean~E. Taylor.
\newblock Flat flow is motion by crystalline curvature for curves with
  crystalline energies.
\newblock {\em J. Differential Geom.}, 42(1):1--22, 1995.

\bibitem{ATW}
Fred Almgren, Jean~E. Taylor, and Lihe Wang.
\newblock Curvature-driven flows: a variational approach.
\newblock {\em SIAM J. Control Optim.}, 31(2):387--438, 1993.

\bibitem{AmbrosioDancer}
L.~Ambrosio.
\newblock Geometric evolution problems, distance function and viscosity
  solutions.
\newblock In {\em Calculus of variations and partial differential equations
  ({P}isa, 1996)}, pages 5--93. Springer, Berlin, 2000.

\bibitem{AmbDM}
Luigi Ambrosio and Gianni Dal~Maso.
\newblock A general chain rule for distributional derivatives.
\newblock {\em Proc. Amer. Math. Soc.}, 108(3):691--702, 1990.

\bibitem{AmbrosioSoner}
Luigi Ambrosio and Halil~Mete Soner.
\newblock Level set approach to mean curvature flow in arbitrary codimension.
\newblock {\em J. Differential Geom.}, 43(4):693--737, 1996.

\bibitem{AngGu89}
Sigurd Angenent and Morton~E. Gurtin.
\newblock Multiphase thermomechanics with interfacial structure. {II}.\
  {E}volution of an isothermal interface.
\newblock {\em Arch. Rational Mech. Anal.}, 108(4):323--391, 1989.

\bibitem{AngIlCh}
Sigurd Angenent, Tom Ilmanen, and David~L. Chopp.
\newblock A computed example of nonuniqueness of mean curvature flow in
  {$\mathbb{R}^3$}.
\newblock {\em Comm. Partial Differential Equations}, 20(11-12):1937--1958,
  1995.

\bibitem{Anz:83}
Gabriele Anzellotti.
\newblock Pairings between measures and bounded functions and compensated
  compactness.
\newblock {\em Ann. Mat. Pura Appl. (4)}, 135:293--318, 1983.

\bibitem{BarlesGeorgelin}
Guy Barles and Christine Georgelin.
\newblock A simple proof of convergence for an approximation scheme for
  computing motions by mean curvature.
\newblock {\em SIAM J. Numer. Anal.}, 32(2):484--500, 1995.

\bibitem{BaSoSou}
Guy Barles, Halil~M. Soner, and Panagiotis~E. Souganidis.
\newblock Front propagation and phase field theory.
\newblock {\em SIAM J. Control Optim.}, 31(2):439--469, 1993.

\bibitem{BarlSouga98}
Guy Barles and Panagiotis~E. Souganidis.
\newblock A new approach to front propagation problems: theory and
  applications.
\newblock {\em Arch. Rational Mech. Anal.}, 141(3):237--296, 1998.

\bibitem{BelCaChaNo}
Giovanni Bellettini, Vicent Caselles, Antonin Chambolle, and Matteo Novaga.
\newblock Crystalline mean curvature flow of convex sets.
\newblock {\em Arch. Ration. Mech. Anal.}, 179(1):109--152, 2006.

\bibitem{BeCaChNo-volpres}
Giovanni Bellettini, Vicent Caselles, Antonin Chambolle, and Matteo Novaga.
\newblock The volume preserving crystalline mean curvature flow of convex sets
  in {$\mathbb{R}^N$}.
\newblock {\em J. Math. Pures Appl. (9)}, 92(5):499--527, 2009.

\bibitem{BeNo}
Giovanni Bellettini and Matteo Novaga.
\newblock Approximation and comparison for nonsmooth anisotropic motion by mean
  curvature in {${\bf R}^N$}.
\newblock {\em Math. Models Methods Appl. Sci.}, 10(1):1--10, 2000.

\bibitem{BeNoPa}
Giovanni Bellettini, Matteo Novaga, and Maurizio Paolini.
\newblock On a crystalline variational problem. {I}. {F}irst variation and
  global {$L^\infty$} regularity.
\newblock {\em Arch. Ration. Mech. Anal.}, 157(3):165--191, 2001.

\bibitem{CaCha}
Vicent Caselles and Antonin Chambolle.
\newblock Anisotropic curvature-driven flow of convex sets.
\newblock {\em Nonlinear Anal.}, 65(8):1547--1577, 2006.

\bibitem{Chambolle}
Antonin Chambolle.
\newblock An algorithm for mean curvature motion.
\newblock {\em Interfaces Free Bound.}, 6(2):195--218, 2004.

\bibitem{CMNP-visco}
Antonin Chambolle, Massimiliano Morini, Matteo Novaga, and Marcello
  Ponsiglione.
\newblock Generalized solutions to crystalline curvature flows.
\newblock In preparation, 2017.

\bibitem{CMP3}
Antonin Chambolle, Massimiliano Morini, and Marcello Ponsiglione.
\newblock Nonlocal curvature flows.
\newblock {\em Arch. Ration. Mech. Anal.}, 218:1263--1329, 2015.

\bibitem{CMP4}
Antonin Chambolle, Massimiliano Morini, and Marcello Ponsiglione.
\newblock Existence and uniqueness for a crystalline mean curvature flow.
\newblock {\em Comm. Pure Appl. Math}, 2016.

\bibitem{ChambolleNovaga}
Antonin Chambolle and Matteo Novaga.
\newblock Approximation of the anisotropic mean curvature flow.
\newblock {\em Math. Models Methods Appl. Sci.}, 17(6):833--844, 2007.

\bibitem{ChNoIFB}
Antonin Chambolle and Matteo Novaga.
\newblock Implicit time discretization of the mean curvature flow with a
  discontinuous forcing term.
\newblock {\em Interfaces Free Bound.}, 10(3):283--300, 2008.

\bibitem{ChaNov-Crystal15}
Antonin Chambolle and Matteo Novaga.
\newblock Existence and uniqueness for planar anisotropic and crystalline
  curvature flow.
\newblock In L.~Ambrosio, Y.~Giga, P.~Rybka, and Y.~Tonegawa, editors, {\em
  Variational Methods for Evolving Objects}, volume~67 of {\em Advanced Studies
  in Pure Mathematics}, pages 87--113. Mathematical Society of Japan, 2015.

\bibitem{CGG}
Yun~Gang Chen, Yoshikazu Giga, and Shun'ichi Goto.
\newblock Uniqueness and existence of viscosity solutions of generalized mean
  curvature flow equations.
\newblock {\em J. Differential Geom.}, 33(3):749--786, 1991.

\bibitem{EckerHuisken}
Klaus Ecker and Gerhard Huisken.
\newblock Mean curvature evolution of entire graphs.
\newblock {\em Ann. of Math. (2)}, 130(3):453--471, 1989.

\bibitem{EvansSpruckIII}
L.~C. Evans and J.~Spruck.
\newblock Motion of level sets by mean curvature. {III}.
\newblock {\em J. Geom. Anal.}, 2(2):121--150, 1992.

\bibitem{EvansSpruckI}
Lawrence~C. Evans and Joel Spruck.
\newblock Motion of level sets by mean curvature. {I}.
\newblock {\em J. Differential Geom.}, 33(3):635--681, 1991.

\bibitem{EvansSpruckII}
Lawrence~C. Evans and Joel Spruck.
\newblock Motion of level sets by mean curvature. {II}.
\newblock {\em Trans. Amer. Math. Soc.}, 330(1):321--332, 1992.

\bibitem{FonsecaMuller}
Irene Fonseca and Stefan M{\"u}ller.
\newblock A uniqueness proof for the {W}ulff theorem.
\newblock {\em Proc. Roy. Soc. Edinburgh Sect. A}, 119(1-2):125--136, 1991.

\bibitem{GigaGiga98}
Mi-Ho Giga and Yoshikazu Giga.
\newblock Evolving graphs by singular weighted curvature.
\newblock {\em Arch. Rational Mech. Anal.}, 141(2):117--198, 1998.

\bibitem{GigaGiga01}
Mi-Ho Giga and Yoshikazu Giga.
\newblock Generalized motion by nonlocal curvature in the plane.
\newblock {\em Arch. Ration. Mech. Anal.}, 159(4):295--333, 2001.

\bibitem{GigaGigaPozar}
Mi-Ho Giga, Yoshikazu Giga, and Norbert Po{\v{z}}{\'a}r.
\newblock Periodic total variation flow of non-divergence type in
  {$\mathbb{R}^n$}.
\newblock {\em J. Math. Pures Appl. (9)}, 102(1):203--233, 2014.

\bibitem{GigaBook}
Yoshikazu Giga.
\newblock {\em Surface evolution equations. A level set approach}, volume~99 of
  {\em Monographs in Mathematics}.
\newblock Birkh\"auser Verlag, Basel, 2006.

\bibitem{GiGu96}
Yoshikazu Giga and Morton~E. Gurtin.
\newblock A comparison theorem for crystalline evolution in the plane.
\newblock {\em Quart. Appl. Math.}, 54(4):727--737, 1996.

\bibitem{GiGuMa98}
Yoshikazu Giga, Morton~E. Gurtin, and Jos{\'e} Matias.
\newblock On the dynamics of crystalline motions.
\newblock {\em Japan J. Indust. Appl. Math.}, 15(1):7--50, 1998.

\bibitem{GOS}
Yoshikazu Giga, Takeshi Ohtsuka, and Reiner Sch{\"a}tzle.
\newblock On a uniform approximation of motion by anisotropic curvature by the
  {A}llen-{C}ahn equations.
\newblock {\em Interfaces Free Bound.}, 8(3):317--348, 2006.

\bibitem{GigaPozar}
Yoshikazu Giga and Norbert Po{\v{z}}{\'a}r.
\newblock {A level set crystalline mean curvature flow of surfaces}.
\newblock {Preprint}, {Hokkaido University}, 2016.

\bibitem{GigaPozar17}
Yoshikazu Giga and Norbert Po{\v{z}}{\'a}r.
\newblock Approximation of general facets by admissible facets for anisotropic
  total variation energy and its application to crystalline flow.
\newblock preprint, 2017.

\bibitem{Gurtin93}
Morton~E. Gurtin.
\newblock {\em Thermomechanics of evolving phase boundaries in the plane}.
\newblock Oxford Mathematical Monographs. The Clarendon Press, Oxford
  University Press, New York, 1993.

\bibitem{LS}
Stephan Luckhaus and Thomas Sturzenhecker.
\newblock Implicit time discretization for the mean curvature flow equation.
\newblock {\em Calc. Var. Partial Differential Equations}, 3(2):253--271, 1995.

\bibitem{OS}
Stanley Osher and James~A. Sethian.
\newblock Fronts propagating with curvature-dependent speed: algorithms based
  on {H}amilton-{J}acobi formulations.
\newblock {\em J. Comput. Phys.}, 79(1):12--49, 1988.

\bibitem{Soner93}
Halil~M. Soner.
\newblock Motion of a set by the curvature of its boundary.
\newblock {\em J. Differential Equations}, 101(2):313--372, 1993.

\bibitem{Taylor78}
Jean~E. Taylor.
\newblock Crystalline variational problems.
\newblock {\em Bull. Amer. Math. Soc.}, 84(4):568--588, 1978.

\end{thebibliography}

\end{document}